\documentclass[a4paper,twoside,11pt,english,intlimits]{article}

\usepackage{amsmath,amsthm,amsfonts,amssymb}
\usepackage[utf8]{inputenc}
\usepackage[T1]{fontenc}
\usepackage{times,euler,euscript,stmaryrd}
\usepackage{fancyhdr,titlesec,url,enumerate}
\usepackage[all]{xy}
\usepackage{hyperref}
\usepackage{ifthen}
\usepackage{tikz}
\usepackage{catex}
\usepackage{comment}
\usepackage{wasysym}

\CompileMatrices
\usetikzlibrary{calc}

\DeclareMathAlphabet{\mathscr}{T1}{pzc}{m}{it} 

\titleformat{\section}[block]{\scshape\filcenter\Large}{\thesection.}{.5em}{}
\titleformat{\subsection}[block]{\bfseries\filcenter\large}{\thesubsection.}{.5em}{\medskip}
\titleformat{\subsubsection}[runin]{\bfseries}{\thesubsubsection.}{.5em}{}[.]

\titlespacing{\subsubsection}{0pt}{10pt}{.5em}

\newtheoremstyle{ntheorem}%
	{\topsep}{\topsep}{\itshape}{0pt}{\bfseries}{.}{.5em}%
	{\thmnumber{#2.\hspace{.5em}}\thmname{#1}\thmnote{ (#3)}}
	
\newtheoremstyle{ndefinition}%
	{\topsep}{\topsep}{\normalfont}{0pt}{\bfseries}{.}{.5em}%
	{\thmnumber{#2.\hspace{.5em}}\thmname{#1}\thmnote{ (#3)}}

\theoremstyle{ntheorem}
  	\newtheorem{theorem}[subsubsection]{Theorem}
  	\newtheorem{proposition}[subsubsection]{Proposition}
	\newtheorem{lemma}[subsubsection]{Lemma}
  	\newtheorem{corollary}[subsubsection]{Corollary}
	\newtheorem{example}[subsubsection]{Example}
	\newtheorem{remark}[subsubsection]{Remark}

\theoremstyle{ndefinition}
	\newtheorem{definition}[subsubsection]{Definition}

\fancypagestyle{plain}{
\fancyhf{}\fancyfoot[LC,RC]{}
\fancyfoot[LE,RO]{$\thepage$}

}

\UseTips
\SelectTips{eu}{11}

\newdir{ >}{{}*!/-10pt/@{>}}
\newdir{ -}{{}*!/-10pt/@{}}
\newdir{> }{{}*!/+10pt/@{>}}

\makeatletter

\xyletcsnamecsname@{dir4{}}{dir{}}
\xydefcsname@{dir4{-}}{\line@ \quadruple@\xydashh@}
\xydefcsname@{dir4{.}}{\point@ \quadruple@\xydashh@}
\xydefcsname@{dir4{~}}{\squiggle@ \quadruple@\xybsqlh@}
\xydefcsname@{dir4{>}}{\Tttip@}
\xydefcsname@{dir4{<}}{\reverseDirection@\Tttip@}

\xydef@\quadruple@#1{%
	\edef\Drop@@{%
		\dimen@=#1\relax
		\dimen@=.5\dimen@
		\A@=-\sinDirection\dimen@
		\B@=\cosDirection\dimen@
		\setboxz@h{%
			\setbox2=\hbox{\kern3\A@\raise3\B@\copy\z@}%
			\dp2=\z@ \ht2=\z@ \wd2=\z@ \box2
			\setbox2=\hbox{\kern\A@\raise\B@\copy\z@}%
			\dp2=\z@ \ht2=\z@ \wd2=\z@ \box2
			\setbox2=\hbox{\kern-\A@\raise-\B@\copy\z@}%
			\dp2=\z@ \ht2=\z@ \wd2=\z@ \box2
			\setbox2=\hbox{\kern-3\A@\raise-3\B@ \noexpand\boxz@}%
			\dp2=\z@ \ht2=\z@ \wd2=\z@ \box2
		}%
		\ht\z@=\z@ \dp\z@=\z@ \wd\z@=\z@ \noexpand\styledboxz@
	}%
}

\xydef@\Tttip@{\kern2pt \vrule height2pt depth2pt width\z@
	\Tttip@@ \kern2pt \egroup
	\U@c=0pt \D@c=0pt \L@c=0pt \R@c=0pt \Edge@c={\circleEdge}%
	\def\Leftness@{.5}\def\Upness@{.5}%
	\def\Drop@@{\styledboxz@}\def\Connect@@{\straight@{\dottedSpread@\jot}}}
	
\xydef@\Tttip@@{%
	\dimen@=.25\dimen@
 	\B@=\cosDirection\dimen@
	\setboxz@h\bgroup\reverseDirection@\line@ \wdz@=\z@ \ht\z@=\z@ \dp\z@=\z@
	{\vDirection@(1,-1)\xydashl@ \xyatipfont\char\DirectionChar}%
	{\vDirection@(1,+1)\xydashl@ \xybtipfont\char\DirectionChar}%
}

\xydef@\ar@form{
	\ifx \space@\next \expandafter\DN@\space{\xyFN@\ar@form}%
	\else\ifx ^\next \DN@ ^{\xyFN@\ar@style}\edef\arvariant@@{\string^}%
	\else\ifx _\next \DN@ _{\xyFN@\ar@style}\edef\arvariant@@{\string_}%
	\else\ifx 0\next \DN@ 0{\xyFN@\ar@style}\def\arvariant@@{0}%
	\else\ifx 1\next \DN@ 1{\xyFN@\ar@style}\def\arvariant@@{1}%
	\else\ifx 2\next \DN@ 2{\xyFN@\ar@style}\def\arvariant@@{2}%
	\else\ifx 3\next \DN@ 3{\xyFN@\ar@style}\def\arvariant@@{3}%
	\else\ifx 4\next \DN@ 4{\xyFN@\ar@style}\def\arvariant@@{4}%
	\else\ifx \bgroup\next \let\next@=\ar@style
	\else\ifx [\next \DN@[##1]{\ar@modifiers{[##1]}}
	\else\ifx *\next \DN@ *{\ar@modifiers}%
	\else\addLT@\ifx\next \let\next@=\ar@slide
	\else\ifx /\next \let\next@=\ar@curveslash
	\else\ifx (\next \let\next@=\ar@curveinout 
	\else\addRQ@\ifx\next \addRQ@\DN@{\ar@curve@}%
	\else\addLQ@\ifx\next \addLQ@\DN@{\xyFN@\ar@curve}%
	\else\addDASH@\ifx\next \addDASH@\DN@{\defarstem@-\xyFN@\ar@}%
	\else\addEQ@\ifx\next \addEQ@\DN@{\def\arvariant@@{2}\defarstem@-\xyFN@\ar@}%
	\else\addDOT@\ifx\next \addDOT@\DN@{\defarstem@.\xyFN@\ar@}%
	\else\ifx :\next \DN@:{\def\arvariant@@{2}\defarstem@.\xyFN@\ar@}%
	\else\ifx ~\next \DN@~{\defarstem@~\xyFN@\ar@}%
	\else\ifx !\next \DN@!{\dasharstem@\xyFN@\ar@}%
	\else\ifx ?\next \DN@?{\ar@upsidedown\xyFN@\ar@}%
	\else \let\next@=\ar@error
	\fi\fi\fi\fi\fi\fi\fi\fi\fi\fi\fi\fi\fi\fi\fi\fi\fi\fi\fi\fi\fi\fi\fi \next@}

\makeatother

\newcommand{\dfl}{\Rightarrow}

\newcommand{\qfl}{\xymatrix@1@C=10pt{\ar@4 [r] &}}

\renewcommand{\phi}{\varphi}
\renewcommand{\epsilon}{\varepsilon}

\definecolor{orange}{rgb}{1,0.55,0}
\definecolor{vert}{rgb}{0,0.45,0}

\newcommand{\ifthen}[2]{\ifthenelse{#1}{#2}{}}

\newcommand{\odfl}[1]{\overset{#1}{\dfl}}

\begin{document}

\quad

\vspace{-2cm}

\begin{center}
\begin{Large}
\textsc{Rewriting in higher dimensional linear categories and application to the affine oriented Brauer category}
\end{Large}

\vskip+5pt

\textbf{Cl\'ement Alleaume}

Univ Lyon, Université Claude Bernard Lyon 1, CNRS UMR 5208,

Institut Camille Jordan, 43 blvd. du 11 novembre 1918, F-69622 Villeurbanne cedex, France \textsf{clement.alleaume@univ-st-etienne.fr}
\end{center}

\begin{small}\begin{minipage}{14cm}
\noindent\textbf{Abstract --} 
In this paper, we introduce a rewriting theory of linear monoidal categories. Those categories are a particular case of what we will define as linear~$(n,p)$-categories. We will also define linear~$(n,p)$-polygraphs, a linear adapation of $n$-polygraphs, to present linear~$(n-1,p)$-categories. We focus then on linear~$(3,2)$-polygraphs to give presentations of linear monoidal categories. We finally give an application of this theory in linear~$(3,2)$-polygraphs to prove a basis theorem on the category~$\mathcal{AOB}$ with a new method using a rewriting property defined by van Ostroom: decreasingness.

\end{minipage}\end{small}


\section{Introduction}

Affine walled Brauer algebras were introduced by Rui and Su \cite{RS} in the study of super Schur-Weyl duality. A result of \cite{RS} is the Schur-Weyl duality between general Lie superalgebras and affine walled Brauer algebras. A linear monoidal category, the affine oriented Brauer category~$\mathcal{AOB}$ was introduced in \cite{BCNR} to encode each walled Brauer algebra as one of its morphism spaces. This category is used to prove basis theorems for the affine walled Brauer algebras given in \cite{RS}, which gives an explicit basis for each affine walled Brauer algebra. The proof of this theorem uses an intermediate result on cyclotomic quotients of $\mathcal{AOB}$. For each of those quotients, a basis is given. With these multiple bases, each morphism space of $\mathcal{AOB}$ is given a generating family which is proved to be linearly independent.
\medskip

Our aim is to give a constructive proof of the mentioned basis result. For this, we study $\mathcal{AOB}$ in this article by rewriting methods. Rewriting is a model of computation presenting relations between expressions as oriented computation steps. There are multiple examples of rewriting systems. An abstract rewriting system \cite{Huet} is the data of a set $S$ and a relation~$\rightarrow$ on~$S$ called the rewrite relation. A rewriting sequence from $a$ to~$b$ is a finite sequence $(u_0,u_1,\cdots u_{n-1},u_n)$ of elements of $S$ such that:
$$a=u_0,$$
$$b=u_n,$$
and for any $0 \leqslant k<n$, we have~$u_k \rightarrow u_{k+1}$. A word rewriting system is the data of an alphabet $A$ and a relation~$\Rightarrow$ on the free monoid~$A^*$ over~$A$. We say that there is a rewriting step from a word~$u$ to a word~$v$ if there are words $a$,~$b$,~$u'$ and~$v'$ such that:
$$u=au'b,$$
$$v=av'b,$$
$$u'\Rightarrow v'.$$
A higher dimensional generalization of such rewriting systems has been introduced by Burroni \cite{Burroni93} under the name of polygraph. An $(n+1)$-polygraph is a rewriting system on the $n$-cells of an $n$-category.
\medskip

To study $\mathcal{AOB}$ from a rewriting point of view, we will need to introduce the rewriting systems presenting monoidal linear categories. The objects giving such rewriting systems will be called linear~$(n,p)$-polygraphs which are linear adaptations of $n$-polygraphs. Linear monoidal categories are a special case of what we will call linear~$(n,p)$-categories. In this language, linear monoidal categories are linear~$(2,2)$-categories with only one 0-cell. They are presented by linear~$(3,2)$-polygraphs. Once those objects are defined, we will introduce the rewriting theory of linear~$(3,2)$-polygraphs. Once we have this theory, we will use it to construct bases for the morphism spaces of $\mathcal{AOB}$.
\medskip

Rewriting can offer constructive proofs by giving presentations of objects with certain properties. For example, two properties studied in rewriting systems are termination and confluence. A rewriting system is terminating if it has no infinite rewriting sequence, in which case all computations end. A rewriting system is confluent if any pair of rewriting sequences with the same source can be completed into a pair of rewriting sequences with the same target, in which case all computations lead to the same result. A rewriting system is said to be convergent if it is terminating and confluent.In the case of word rewriting, the property of convergence, the conjunction of termination and confluence, gives a way to decide the word problem, that is, deciding if two words in $A^*$ are equal in the quotient of $A^*$ by the relation~$\Rightarrow$.

\medskip

What we will do in the case of $\mathcal{AOB}$ is giving a confluent presentation~$\overline{\mathrm{AOB}}$ of this linear~$(2,2)$-category with some others properties. Those properties will prove the bases given in \cite{BCNR} are indeed bases. The linear~$(3,2)$-polygraph~$\overline{\mathrm{AOB}}$ will not be terminating, which will prevent us to prove~$\overline{\mathrm{AOB}}$ is confluent by the use of Newman's lemma, a criterion needing termination to prove confluence from a weaker property called local confluence \cite{Huet}. To prove~$\overline{\mathrm{AOB}}$ is confluent, we will use a more general property called decreasingness introduced in \cite{VO}. We will prove~$\overline{\mathrm{AOB}}$ is decreasing and use the theorem from \cite{VO} stating decreasingness implies confluence.
\medskip

In the first section, we recall first the notions of higher dimensional category theory. Then, we define linear~$(n,p)$-categories, which will be our higher dimensional categories with linear structure. After defining them, we recall in the second section the categorical construction of the category of $n$-polygraphs given in \cite{Metayer08}. We define next the categorical construction of the category of linear~$(n,p)$-polygraphs. We give their main rewriting properties, such than \ref{basis} in the case $(n,p)=(3,2)$ in which~$\mathcal{AOB}$ falls.
\medskip

In the third section we will study the decreasingness property defined in the case of abstract rewriting systems by van Ostroom \cite{VO}. Then, in the last section, we recall from \cite{BCNR} the definition of the linear~$(2,2)$-category~$\mathcal{AOB}$. This will lead us to give two linear~$(3,2)$-polygraphs presenting~$\mathcal{AOB}$. Those linear~$(3,2)$-polygraphs will be called~$\mathrm{AOB}$ and~$\overline{\mathrm{AOB}}$. The first one is a translation of the definition of $\mathcal{AOB}$.
\medskip

The main result of this article, Theorem \ref{thconf} states $\overline{\mathrm{AOB}}$ is confluent. It will be proved with the properties of confluence of critical branchings and decreasingness. This theorem gives us the main result of \cite{BCNR} as an entirely constructive consequence given as corollary \ref{Bru}.

\section*{Acknowledgments}

The author would like to thank Labex Milyon (ANR-10-LABX-0070). This work has been supported by the project \emph{Cathre}, ANR-13BS02-0005-02.

\section{Linear~$(n,p)$-categories}

This section and section 3 will recall the basic notions of $n$-category and~$n$-polygraphs. The rewriting techniques used in this paper will be presented in the last two sections.

\subsection{Preliminaries}

We fix $n$ an integer. We denote by $\mathbf{Set}$ the category of sets.

\begin{definition} An \emph{$n$-graph} in a category~$\mathbf{C}$ is a diagram in $\mathbf{C}$:
\begin{gather*}
\begin{array}{c}
\tikz[scale=0.9]{
\node at (0,0) {$G_0$};
\draw[color=black, ->] (1.5,0.25) -- (0.5,0.25);
\node at (1,0.5) {$s_0$};
\draw[color=black, ->] (1.5,-0.25) -- (0.5,-0.25);
\node at (1,-0.5) {$t_0$};
\node at (2,0) {$G_1$};
\draw[color=black, ->] (3.5,0.25) -- (2.5,0.25);
\node at (3,0.5) {$s_1$};
\draw[color=black, ->] (3.5,-0.25) -- (2.5,-0.25);
\node at (3,-0.5) {$t_1$};
\node at (4,0) {$\cdots$};
\draw[color=black, ->] (5.5,0.25) -- (4.5,0.25);
\node at (5,0.5) {$s_{n-2}$};
\draw[color=black, ->] (5.5,-0.25) -- (4.5,-0.25);
\node at (5,-0.5) {$t_{n-2}$};
\node at (6,0) {$G_{n-1}$};
\draw[color=black, ->] (7.5,0.25) -- (6.5,0.25);
\node at (7,0.5) {$s_{n-1}$};
\draw[color=black, ->] (7.5,-0.25) -- (6.5,-0.25);
\node at (7,-0.5) {$t_{n-1}$};
\node at (8,0) {$G_n$};
} \end{array}
\end{gather*}
such that for any $1 \leqslant k \leqslant n-1$, we have~$s_{k-1} \circ s_k=s_{k-1} \circ t_k$ and~$t_{k-1} \circ s_k=t_{k-1} \circ t_k$. Those relations are called the \emph{globular relations}. We just call an $n$-graph in $\mathbf{Set}$ an $n$-graph.

The elements of $G_k$ are called~\emph{$k$-cells}. The maps $s_k$ and~$t_k$ are respectively called~\emph{$k$-source} and~\emph{$k$-target maps}. For any $l$-cell~$u$ of $G$ with~$l > k+1$, we respectively denote by $s_k(u)$ and~$t_k(u)$ the $k$-cells $(s_k \circ \cdots  \circ s_{l-1})(u)$ and~$(t_k \circ \cdots  \circ t_{l-1})(u)$.

A \emph{morphism of $n$-graphs} $F$ from $G$ to~$G'$ is a collection~$(F_k: G_k \rightarrow G'_k)$ of maps such that, for every~$0 < k \leqslant n$, the following diagrams commute:
\begin{gather*}
\begin{array}{c}
\tikz[scale=0.9]{
\node at (0,0) {$G_{k-1}$};
\draw[color=black, ->] (1.5,0) -- (0.5,0);
\node at (1,0.5) {$s_{k-1}$};
\node at (2,0) {$G_k$};
\draw[color=black, ->] (0,-0.3) -- (0,-1.7);
\node at (-0.5,-1) {$F_{k-1}$};
\draw[color=black, ->] (2,-0.3) -- (2,-1.7);
\node at (2.5,-1) {$F_k$};
\node at (0,-2) {$G'_{k-1}$};
\draw[color=black, ->] (1.5,-2) -- (0.5,-2);
\node at (1,-2.25) {$s'_{k-1}$};
\node at (2,-2) {$G'_k$};
} \end{array}
\qquad
\begin{array}{c}
\tikz[scale=0.9]{
\node at (0,0) {$G_{k-1}$};
\draw[color=black, ->] (1.5,0) -- (0.5,0);
\node at (1,0.5) {$t_{k-1}$};
\node at (2,0) {$G_k$};
\draw[color=black, ->] (0,-0.3) -- (0,-1.7);
\node at (-0.5,-1) {$F_{k-1}$};
\draw[color=black, ->] (2,-0.3) -- (2,-1.7);
\node at (2.5,-1) {$F_k$};
\node at (0,-2) {$G'_{k-1}$};
\draw[color=black, ->] (1.5,-2) -- (0.5,-2);
\node at (1,-2.25) {$t'_{k-1}$};
\node at (2,-2) {$G'_k$};
} \end{array}
\end{gather*}

We call $\mathbf{Grph}_n$ the category of $n$-graphs.
\end{definition}

In particular,~$\mathbf{Grph}_0$ is the category~$\mathbf{Set}$ and~$\mathbf{Grph}_1$ is the category of directed graphs.

We can define \emph{$n$-categories} by enrichment. A 0-category is a set. For~$n \geqslant 1$, an $n$-category is a 1-category enriched in $(n-1)$-categories. We denote by $\mathbf{Cat}_n$ the category of $n$-categories and~$n$-functors. This category has a terminal object $I_n$ with only one $k$-cell for~$0 \leqslant k \leqslant n$.

In an $n$-category, for any $0 \leqslant k<n$, we denote the $k$-composition by $\star_k$. For all $0 \leqslant i<j \leqslant n-1$ the following equality, called \emph{exchange relation}, holds:
\begin{equation}\label{exch} (u \star_i v) \star_j (u' \star_i v')=(u \star_j u') \star_i (v \star_j v'). \end{equation}

By forgetting its units and compositions, an $n$-category is in particular an $n$-graph. We now denote by $\mathcal{U}_n$ the forgetful functor from $\mathbf{Cat}_n$ to~$\mathbf{Grph}_n$.

\subsection{Linear~$(n,p)$-categories}

In this section, we fix $n$ an integer. We will define linear~$(n,p)$-categories by induction on~$p \leqslant n$. We will denote by $\mathbf{Vect}$ the category of vector spaces over a fixed field.

\begin{definition} An internal $n$-category in $\mathbf{Vect}$ is the data of:
\begin{itemize}
\item an $n$-graph in $\mathbf{Vect}$:
\begin{gather*}
\begin{array}{c}
\tikz[scale=0.9]{
\node at (0,0) {$V_0$};
\draw[color=black, ->] (1.5,0.25) -- (0.5,0.25);
\node at (1,0.5) {$s_0$};
\draw[color=black, ->] (1.5,-0.25) -- (0.5,-0.25);
\node at (1,-0.5) {$t_0$};
\node at (2,0) {$V_1$};
\draw[color=black, ->] (3.5,0.25) -- (2.5,0.25);
\node at (3,0.5) {$s_1$};
\draw[color=black, ->] (3.5,-0.25) -- (2.5,-0.25);
\node at (3,-0.5) {$t_1$};
\node at (4,0) {$\cdots$};
\draw[color=black, ->] (5.5,0.25) -- (4.5,0.25);
\node at (5,0.5) {$s_{n-2}$};
\draw[color=black, ->] (5.5,-0.25) -- (4.5,-0.25);
\node at (5,-0.5) {$t_{n-2}$};
\node at (6,0) {$V_{n-1}$};
\draw[color=black, ->] (7.5,0.25) -- (6.5,0.25);
\node at (7,0.5) {$s_{n-1}$};
\draw[color=black, ->] (7.5,-0.25) -- (6.5,-0.25);
\node at (7,-0.5) {$t_{n-1}$};
\node at (8,0) {$V_n$};,
} \end{array}
\end{gather*}
\item for each $0 \leqslant k <l \leqslant n$, a linear unit map $v \mapsto 1_v$ from $V_k$ to~$V_{k+1}$. Linearity of the unit maps gives us the following relation:
\begin{equation}\label{idlin} 1_{\lambda u+ \mu v}=\lambda 1_u+ \mu 1_v\end{equation}
for any scalars $\lambda$ and~$\mu$ and any $k$-cells $u$ and~$v$ such that~$p \leqslant k<n$,
\item for each $0 \leqslant k <l \leqslant n$, a linear composition map $\star_k$ from $V_l \times_{V_k} V_l$ to~$V_k$. Linearity of the unit maps gives us the following properties:
\begin{equation}\label{complin1} (f+g)\star_k (f'+g')=f \star_k f'+g \star_k g',\end{equation}
\begin{equation}\label{complin2} \lambda f \star_k \lambda f'=\lambda (f \star_k f').\end{equation}
for any scalar $\lambda$ and any pairs $(f,f')$ and~$(g,g')$ of $k$-composable $l$-cells such that~$p \leqslant k<l \leqslant n$
\end{itemize}
verifying the axioms of $n$-categories.
\end{definition}

Internal $n$-categories in $\mathbf{Vect}$ are also called~$(n+1)$-vector spaces \cite{BaezCrans04}, see also \cite{KMP}.

\begin{definition} A linear~$(n,0)$-category is an internal $n$-category in $\mathbf{Vect}$. Let us assume linear~$(n,p)$-categories are defined for~$p \geqslant 0$. A linear~$(n+1,p+1)$-category is the data of a set $\mathcal{C}_0$ and:
\begin{itemize}
\item for each $a$ and~$b$ in $\mathcal{C}_0$, a linear~$(n,p)$-category~$\mathcal{C}(a,b)$,
\item for each $a$ in $\mathcal{C}_0$, an identity morphism $i_a$ from the terminal $n$-category~$I_n$ to~$\mathcal{C}(a,a)$,
\item for each $a$,~$b$ and~$c$ in $\mathcal{C}_0$, a bilinear composition morphism $\star^{a,b,c}$ from $\mathcal{C}(a,b) \times \mathcal{C}(b,c)$ to~$\mathcal{C}(a,c)$.
\end{itemize}
such that:
\begin{itemize}
\item $\star^{a,c,d} \circ (\star^{a,b,c} \times id_{\mathcal{C}(c,d)})=\star^{a,b,d} \circ (id_{\mathcal{C}(a,b)} \times \star^{b,c,d})$,
\item $\star^{a,a,b} \circ (i_a \times id_{\mathcal{C}(a,b)}) \circ is_l=id_{\mathcal{C}(a,b)}=\star^{a,b,b} \circ (id_{\mathcal{C}(a,b)} \times i_q) \circ is_r$ where $is_l$ and~$is_r$ respectively denote the canonic isomorphisms from $\mathcal{C}(a,b)$ to~$I_n \times \mathcal{C}(a,b)$ and to~$\mathcal{C}(a,b) \times I_n$.
\end{itemize}
\end{definition}

\begin{remark} In particular, a linear~$(n,p)$-category is an $n$-category. \end{remark}

A \emph{morphism of linear~$(n,p)$-categories} from $\mathcal{C}$ to~$\mathcal{C}'$ is an $n$-functor:
\begin{gather*}
\begin{array}{c}
\tikz[scale=0.9]{
\node at (0,0) {$\mathcal{C}_0$};
\draw[color=black, ->] (1.5,0.25) -- (0.5,0.25);
\node at (1,0.5) {$s_0$};
\draw[color=black, ->] (1.5,-0.25) -- (0.5,-0.25);
\node at (1,-0.5) {$t_0$};
\node at (2,0) {$\mathcal{C}_1$};
\draw[color=black, ->] (3.5,0.25) -- (2.5,0.25);
\node at (3,0.5) {$s_1$};
\draw[color=black, ->] (3.5,-0.25) -- (2.5,-0.25);
\node at (3,-0.5) {$t_1$};
\node at (4,0) {$\cdots$};
\draw[color=black, ->] (5.5,0.25) -- (4.5,0.25);
\node at (5,0.5) {$s_{n-2}$};
\draw[color=black, ->] (5.5,-0.25) -- (4.5,-0.25);
\node at (5,-0.5) {$t_{n-2}$};
\node at (6,0) {$\mathcal{C}_{n-1}$};
\draw[color=black, ->] (7.5,0.25) -- (6.5,0.25);
\node at (7,0.5) {$s_{n-1}$};
\draw[color=black, ->] (7.5,-0.25) -- (6.5,-0.25);
\node at (7,-0.5) {$t_{n-1}$};
\node at (8,0) {$\mathcal{C}_n$};
\node at (-0.5,-1.5) {$F_0$};
\draw[color=black, ->] (0,-0.3) -- (0,-2.7);
\node at (1.5,-1.5) {$F_1$};
\draw[color=black, ->] (2,-0.3) -- (2,-2.7);
\node at (5.3,-1.5) {$F_{n-1}$};
\draw[color=black, ->] (6,-0.3) -- (6,-2.7);
\node at (7.5,-1.5) {$F_n$};
\draw[color=black, ->] (8,-0.3) -- (8,-2.7);
\node at (0,-3) {$\mathcal{C}'_0$};
\draw[color=black, ->] (1.5,-2.75) -- (0.5,-2.75);
\node at (1,-2.5) {$s'_0$};
\draw[color=black, ->] (1.5,-3.25) -- (0.5,-3.25);
\node at (1,-3.5) {$t'_0$};
\node at (2,-3) {$\mathcal{C}'_1$};
\draw[color=black, ->] (3.5,-2.75) -- (2.5,-2.75);
\node at (3,-2.5) {$s'_1$};
\draw[color=black, ->] (3.5,-3.25) -- (2.5,-3.25);
\node at (3,-3.5) {$t'_1$};
\node at (4,-3) {$\cdots$};
\draw[color=black, ->] (5.5,-2.75) -- (4.5,-2.75);
\node at (5,-2.5) {$s'_{n-2}$};
\draw[color=black, ->] (5.5,-3.25) -- (4.5,-3.25);
\node at (5,-3.5) {$t'_{n-2}$};
\node at (6,-3) {$\mathcal{C}'_{n-1}$};
\draw[color=black, ->] (7.5,-2.75) -- (6.5,-2.75);
\node at (7,-2.5) {$s'_{n-1}$};
\draw[color=black, ->] (7.5,-3.25) -- (6.5,-3.25);
\node at (7,-3.5) {$t'_{n-1}$};
\node at (8,-3) {$\mathcal{C}'_n$};
} \end{array}
\end{gather*}
such that for any $p \leqslant k \leqslant n$, the map $F_k$ is linear. We denote by $\mathbf{LinCat}_{n,p}$ the category of linear~$(n,p)$-categories. We denote by $\mathcal{U}_{n,p}$ the forgetful functor from $\mathbf{LinCat}_{n,p}$ to~$\mathbf{Grph}_n$.

\begin{example} A linear~$(1,1)$-category is just a linear category or a category enriched in vector spaces as introduced by Mitchell  \cite{Mitchell72}.
\end{example}

\begin{remark} Let $\mathcal{C}$ be a linear~$(n,p)$-category with~$0 \leqslant p<n$. The underlying~$(n-1)$-category to~$\mathcal{C}$ is a linear~$(n-1,p)$-category. Indeed, there is an internal $(n+p-1)$-graph in the category of vector spaces given by:
\begin{gather*}
\begin{array}{c}
\tikz[scale=0.9]{
\node at (0,0) {$\mathcal{C}_p$};
\draw[color=black, ->] (1.5,0.25) -- (0.5,0.25);
\node at (1,0.5) {$s_p$};
\draw[color=black, ->] (1.5,-0.25) -- (0.5,-0.25);
\node at (1,-0.5) {$t_p$};
\node at (2,0) {$\mathcal{C}_{p+1}$};
\draw[color=black, ->] (3.5,0.25) -- (2.5,0.25);
\node at (3,0.5) {$s_{p+1}$};
\draw[color=black, ->] (3.5,-0.25) -- (2.5,-0.25);
\node at (3,-0.5) {$t_{p+1}$};
\node at (4,0) {$\cdots$};
\draw[color=black, ->] (5.5,0.25) -- (4.5,0.25);
\node at (5,0.5) {$s_{n-3}$};
\draw[color=black, ->] (5.5,-0.25) -- (4.5,-0.25);
\node at (5,-0.5) {$t_{n-3}$};
\node at (6,0) {$\mathcal{C}_{n-2}$};
\draw[color=black, ->] (7.5,0.25) -- (6.5,0.25);
\node at (7,0.5) {$s_{n-2}$};
\draw[color=black, ->] (7.5,-0.25) -- (6.5,-0.25);
\node at (7,-0.5) {$t_{n-2}$};
\node at (8,0) {$\mathcal{C}_{n-1}$};
} \end{array}
\end{gather*}
which can be completed into an $(n-1)$-category:
\begin{gather*}
\begin{array}{c}
\tikz[scale=0.9]{
\node at (-8,0) {$\mathcal{C}_0$};
\draw[color=black, ->] (-6.5,0.25) -- (-7.5,0.25);
\node at (-7,0.5) {$s_0$};
\draw[color=black, ->] (-6.5,-0.25) -- (-7.5,-0.25);
\node at (-7,-0.5) {$t_0$};
\node at (-6,0) {$\mathcal{C}_1$};
\draw[color=black, ->] (-4.5,0.25) -- (-5.5,0.25);
\node at (-5,0.5) {$s_1$};
\draw[color=black, ->] (-4.5,-0.25) -- (-5.5,-0.25);
\node at (-5,-0.5) {$t_1$};
\node at (-4,0) {$\cdots$};
\draw[color=black, ->] (-2.5,0.25) -- (-3.5,0.25);
\node at (-3,0.5) {$s_{p-2}$};
\draw[color=black, ->] (-2.5,-0.25) -- (-3.5,-0.25);
\node at (-3,-0.5) {$t_{p-2}$};
\node at (-2,0) {$\mathcal{C}_{p-1}$};
\draw[color=black, ->] (-0.5,0.25) -- (-1.5,0.25);
\node at (-1,0.5) {$s_{p-1}$};
\draw[color=black, ->] (-0.5,-0.25) -- (-1.5,-0.25);
\node at (-1,-0.5) {$t_{p-1}$};
\node at (0,0) {$\mathcal{C}_p$};
\draw[color=black, ->] (1.5,0.25) -- (0.5,0.25);
\node at (1,0.5) {$s_p$};
\draw[color=black, ->] (1.5,-0.25) -- (0.5,-0.25);
\node at (1,-0.5) {$t_p$};
\node at (2,0) {$\mathcal{C}_{p+1}$};
\draw[color=black, ->] (3.5,0.25) -- (2.5,0.25);
\node at (3,0.5) {$s_{p+1}$};
\draw[color=black, ->] (3.5,-0.25) -- (2.5,-0.25);
\node at (3,-0.5) {$t_{p+1}$};
\node at (4,0) {$\cdots$};
\draw[color=black, ->] (5.5,0.25) -- (4.5,0.25);
\node at (5,0.5) {$s_{n-3}$};
\draw[color=black, ->] (5.5,-0.25) -- (4.5,-0.25);
\node at (5,-0.5) {$t_{n-3}$};
\node at (6,0) {$\mathcal{C}_{n-2}$};
\draw[color=black, ->] (7.5,0.25) -- (6.5,0.25);
\node at (7,0.5) {$s_{n-2}$};
\draw[color=black, ->] (7.5,-0.25) -- (6.5,-0.25);
\node at (7,-0.5) {$t_{n-2}$};
\node at (8,0) {$\mathcal{C}_{n-1}$};
} \end{array}
\end{gather*}
The source maps, target maps and composition maps of this $(n-1)$-category meet all the requirements of a linear~$(n-1,p)$-category. \end{remark}

\section{Polygraphs}

\subsection{$n$-Polygraphs}

The notion of polygraph was introduced by Burroni \cite{Burroni93}. It was introduced independently by Street under the name of computad \cite{Street87} as system of generators and relations for higher dimensional categories. Let us recall the constructions of the categories of $n$-categories with a globular extension and~$n$-polygraphs from \cite{Metayer08}.

\begin{definition} The category~$\mathbf{Cat}_n^+$ of $n$-categories with a globular extension is defined by the following pullback diagram:
\begin{gather*}
\begin{array}{c}
\tikz[scale=0.9]{
\node at (0,0) {$\mathbf{Cat}_n^+$};
\draw[color=black, ->] (0.5,0) -- (2,0);
\node at (2.9,0) {$\mathbf{Grph}_{n+1}$};
\draw[color=black, ->] (2.7,-0.5) -- (2.7,-2);
\node at (2.7,-2.5) {$\mathbf{Grph}_n$};
\draw[color=black, ->] (0,-0.5) -- (0,-2);
\node at (0,-2.5) {$\mathbf{Cat}_n$};
\draw[color=black, ->] (0.5,-2.5) -- (2,-2.5);
\node at (2.7,-2.5) {$\mathbf{Grph}_n$};
\draw[color=black] (0.3,-0.5) -- (0.55,-0.5) -- (0.55,-0.25);
\node at (1.25,-3) {$\mathcal{U}_n$};
\node at (3.2,-1.25) {$\mathcal{U}_n^G$};
} \end{array}
\end{gather*}
where $\mathcal{U}_n^G$ is the functor from $\mathbf{Grph}_{n+1}$ to~$\mathbf{Grph}_n$ associating to each $(n+1)$-graph its underlying~$n$-graph by eliminating the $(n+1)$-cells.
\end{definition}

In a diagrammatic way, a globular extension of an $n$-category~$\mathcal{C}$ is the data of a set $\Gamma$ and of two maps $s_n$ and~$t_n$ from $\Gamma$ to~$\mathcal{C}_n$ forming an $(n+1)$-graph:
\begin{gather*}
\begin{array}{c}
\tikz[scale=0.9]{
\node at (0,0) {$\mathcal{C}_0$};
\draw[color=black, ->] (1.5,0.25) -- (0.5,0.25);
\node at (1,0.5) {$s_0$};
\draw[color=black, ->] (1.5,-0.25) -- (0.5,-0.25);
\node at (1,-0.5) {$t_0$};
\node at (2,0) {$\mathcal{C}_1$};
\draw[color=black, ->] (3.5,0.25) -- (2.5,0.25);
\node at (3,0.5) {$s_1$};
\draw[color=black, ->] (3.5,-0.25) -- (2.5,-0.25);
\node at (3,-0.5) {$t_1$};
\node at (4,0) {$\cdots$};
\draw[color=black, ->] (5.5,0.25) -- (4.5,0.25);
\node at (5,0.5) {$s_{n-2}$};
\draw[color=black, ->] (5.5,-0.25) -- (4.5,-0.25);
\node at (5,-0.5) {$t_{n-2}$};
\node at (6,0) {$\mathcal{C}_{n-1}$};
\draw[color=black, ->] (7.5,0.25) -- (6.5,0.25);
\node at (7,0.5) {$s_{n-1}$};
\draw[color=black, ->] (7.5,-0.25) -- (6.5,-0.25);
\node at (7,-0.5) {$t_{n-1}$};
\node at (8,0) {$\mathcal{C}_n$};
\draw[color=black, ->] (9.5,0.25) -- (8.5,0.25);
\node at (9,0.5) {$s_n$};
\draw[color=black, ->] (9.5,-0.25) -- (8.5,-0.25);
\node at (9,-0.5) {$t_n$};
\node at (10,0) {$\Gamma$};
} \end{array}
\end{gather*}
A morphism of $n$-categories with globular extension from $(\mathcal{C},\Gamma)$ to~$(\mathcal{C}',\Gamma')$ is the data of an $n$-functor~$F$ from $\mathcal{C}$ to~$\mathcal{C}'$ and a map $\phi$ from $\Gamma$ to~$\Gamma'$ such that~$(F,\phi)$ makes a morphism of $(n+1)$-graphs.

\begin{definition} Let $(\mathcal{C},\Gamma)$ be an $n$-category with globular extension. The free~$(n+1)$-category over~$(\mathcal{C},\Gamma)$ is the $(n+1)$-category which underlying~$n$-category is $\mathcal{C}$ and which~$(n+1)$-cells are the compositions of elements of $\Gamma$ and elements of the form $1_u$ where $u$ is in $\mathcal{C}_n$. The free functor from $\mathbf{Cat}_n^+$ to~$\mathbf{Grph}_{n+1}$ is denoted by $\mathcal{F}_{n+1}^W$.
\end{definition}


\begin{definition} $\mathbf{Pol}_0$ is the category of sets and the functor~$\mathcal{F}_0$ from $\mathbf{Pol}_0$ to~$\mathbf{Cat}_0$ is the identity functor. Let us assume the category~$\mathbf{Pol}_n$ of $n$-polygraphs and the functor~$\mathcal{F}_n$ from $\mathbf{Pol}_n$ to~$\mathbf{Cat}_n$ are defined. The category~$\mathbf{Pol}_{n+1}$ is defined by the following pullback diagram:
\begin{gather*}
\begin{array}{c}
\tikz[scale=0.9]{
\node at (0,0) {$\mathbf{Pol}_{n+1}$};
\draw[color=black, ->] (0.7,0) -- (5,0);
\node at (5.9,0) {$\mathbf{Grph}_{n+1}$};
\draw[color=black, ->] (5.7,-0.5) -- (5.7,-2);
\node at (5.7,-2.5) {$\mathbf{Grph}_n$};
\draw[color=black, ->] (0,-0.5) -- (0,-2);
\node at (0,-2.5) {$\mathbf{Pol}_n$};
\draw[color=black, ->] (0.5,-2.5) -- (2,-2.5);
\node at (2.5,-2.5) {$\mathbf{Cat}_n$};
\draw[color=black, ->] (3,-2.5) -- (5,-2.5);
\draw[color=black] (0.3,-0.5) -- (0.55,-0.5) -- (0.55,-0.25);
\node at (4,-3) {$\mathcal{U}_n$};
\node at (1,-3) {$\mathcal{F}_n$};
\node at (6.2,-1.25) {$\mathcal{U}_n^G$};
\node at (-0.5,-1.25) {$\mathcal{U}_n^P$};
\node at (2.85,0.5) {$\mathcal{U}_{n+1}^{GP}$};
} \end{array}
\end{gather*}
We denote by $\mathcal{F}_{n+1}^P$ the unique functor making the following diagram commutative:
\begin{gather*}
\begin{array}{c}
\tikz[scale=0.9]{
\node at (-3,2) {$\mathbf{Pol}_{n+1}$};
\draw[color=black, ->] (-2.3,1.5) -- (-0.6,0.4);
\draw[color=black, ->] (-3,1.5) -- (-3,-2);
\node at (-3,-2.5) {$\mathbf{Pol}_n$};
\draw[color=black, ->] (-2.3,-2.5) -- (-0.5,-2.5);
\node at (-1.4,-3) {$\mathcal{F}_n$};
\draw[color=black, ->] (-2.3,2) -- (2,0.4);
\node at (-1.45,0.45) {$\mathcal{F}_{n+1}^P$};
\node at (0,0) {$\mathbf{Cat}_n^+$};
\draw[color=black, ->] (0.5,0) -- (2,0);
\node at (2.9,0) {$\mathbf{Grph}_{n+1}$};
\draw[color=black, ->] (2.7,-0.5) -- (2.7,-2);
\node at (2.7,-2.5) {$\mathbf{Grph}_n$};
\draw[color=black, ->] (0,-0.5) -- (0,-2);
\node at (0,-2.5) {$\mathbf{Cat}_n$};
\draw[color=black, ->] (0.5,-2.5) -- (2,-2.5);
\node at (2.7,-2.5) {$\mathbf{Grph}_n$};
\draw[color=black] (0.3,-0.5) -- (0.55,-0.5) -- (0.55,-0.25);
\node at (1.25,-3) {$\mathcal{U}_n$};
\node at (3.2,-1.25) {$\mathcal{U}_n^G$};
\node at (-3.5,-0.25) {$\mathcal{U}_n^P$};
\node at (0.5,1.45) {$\mathcal{U}_{n+1}^{GP}$};
} \end{array}
\end{gather*}
The functor~$\mathcal{F}_{n+1}$ is defined as the following composite:
\begin{gather*}
\begin{array}{c}
\tikz[scale=0.9]{
\node at (0,0) {$\mathbf{Pol}_{n+1}$};
\draw[color=black, ->] (0.7,0) -- (2.7,0);
\node at (3.2,0) {$\mathbf{Cat}_n^+$};
\draw[color=black, ->] (3.7,0) -- (5.7,0);
\node at (6.4,0) {$\mathbf{Cat}_{n+1}$};
\node at (1.7,0.5) {$\mathcal{F}_{n+1}^P$};
\node at (4.7,0.5) {$\mathcal{F}_{n+1}^W$};
} \end{array}
\end{gather*}
\end{definition}

A 0-polygraph is a set. The free 0-category over a 0-polygraph is the 0-polygraph itself.

Given an $n$-polygraph~$\Sigma$, we denote by $\Sigma^*$ the free~$n$-category over~$\Sigma$. Let us assume $n$-polygraphs and free~$n$-categories over~$n$-polygraphs are defined. An $(n+1)$-polygraph is the data of an $n$-polygraph~$\Sigma$:
\begin{gather*}
\begin{array}{c}
\tikz[scale=0.5]{
\draw[color=black, ->] (0,0.5) -- (0,2.5);
\node at (0,0) {$\Sigma_0$};
\node at (0,3) {$\Sigma_0^*$};
\draw[color=black, ->] (3.5,0.3) -- (0.5,2.8);
\draw[color=black, ->] (3.5,0.7) -- (0.5,3.2);
\node at (2,1) {$s_0$};
\node at (2,2.5) {$t_0$};
\node at (4,0) {$\Sigma_1$};
\draw[color=black, ->] (4,0.5) -- (4,2.5);
\node at (4,3) {$\Sigma_1^*$};
\draw[color=black, ->] (7.5,0.3) -- (4.5,2.8);
\draw[color=black, ->] (7.5,0.7) -- (4.5,3.2);
\node at (6,1) {$s_1$};
\node at (6,2.5) {$t_1$};
\node at (8,0) {$\cdots$};
\node at (8,3) {$\cdots$};
\draw[color=black, ->] (11.5,0.3) -- (8.5,2.8);
\draw[color=black, ->] (11.5,0.7) -- (8.5,3.2);
\node at (9.2,1) {$s_{n-2}$};
\node at (10.2,2.5) {$t_{n-2}$};
\node at (12.3,0) {$\Sigma_{n-1}$};
\draw[color=black, ->] (12,0.5) -- (12,2.5);
\node at (11.8,3) {$\Sigma_{n-1}^*$};
\draw[color=black, ->] (15.7,0.3) -- (12.7,2.8);
\draw[color=black, ->] (15.7,0.7) -- (12.7,3.2);
\node at (13.4,1) {$s_{n-1}$};
\node at (14.4,2.5) {$t_{n-1}$};
\node at (16.2,0) {$\Sigma_n$};
} \end{array}
\end{gather*}
and a globular extension~$\Sigma_{n+1}$ of the free~$n$-category over~$\Sigma$:
\begin{gather*}
\begin{array}{c}
\tikz[scale=0.5]{
\draw[color=black, ->] (0,0.5) -- (0,2.5);
\node at (0,0) {$\Sigma_0$};
\node at (0,3) {$\Sigma_0^*$};
\draw[color=black, ->] (3.5,0.3) -- (0.5,2.8);
\draw[color=black, ->] (3.5,0.7) -- (0.5,3.2);
\node at (2,1) {$s_0$};
\node at (2,2.5) {$t_0$};
\node at (4,0) {$\Sigma_1$};
\draw[color=black, ->] (4,0.5) -- (4,2.5);
\node at (4,3) {$\Sigma_1^*$};
\draw[color=black, ->] (7.5,0.3) -- (4.5,2.8);
\draw[color=black, ->] (7.5,0.7) -- (4.5,3.2);
\node at (6,1) {$s_1$};
\node at (6,2.5) {$t_1$};
\node at (8,0) {$\cdots$};
\node at (8,3) {$\cdots$};
\draw[color=black, ->] (11.5,0.3) -- (8.5,2.8);
\draw[color=black, ->] (11.5,0.7) -- (8.5,3.2);
\node at (9.2,1) {$s_{n-2}$};
\node at (10.2,2.5) {$t_{n-2}$};
\node at (12.3,0) {$\Sigma_{n-1}$};
\draw[color=black, ->] (12,0.5) -- (12,2.5);
\node at (11.8,3) {$\Sigma_{n-1}^*$};
\draw[color=black, ->] (15.7,0.3) -- (12.7,2.8);
\draw[color=black, ->] (15.7,0.7) -- (12.7,3.2);
\node at (13.4,1) {$s_{n-1}$};
\node at (14.4,2.5) {$t_{n-1}$};
\node at (16.2,0) {$\Sigma_n$};
\draw[color=black, ->] (16.2,0.5) -- (16.2,2.5);
\node at (16.2,3) {$\Sigma_n^*$};
\node at (20.5,0) {$\Sigma_{n+1}$};
\draw[color=black, ->] (19.7,0.3) -- (16.7,2.8);
\draw[color=black, ->] (19.7,0.7) -- (16.7,3.2);
\node at (18.2,1) {$s_n$};
\node at (18.2,2.5) {$t_n$};
} \end{array}
\end{gather*}

The elements of $\Sigma_k$ are called~\emph{$k$-cells}.

\subsection{Linear~$(n,p)$-polygraphs}

In this section, we introduce linear~$(n,p)$-polygraphs.

\begin{definition} The category~$\mathbf{LinCat}^+_{n,p}$ of linear~$(n,p)$-categories with a globular extension is defined by the following pullback diagram:
\begin{gather*}
\begin{array}{c}
\tikz[scale=0.9]{
\node at (-0.6,0) {$\mathbf{LinCat}^+_{n,p}$};
\draw[color=black, ->] (0.5,0) -- (2,0);
\node at (2.9,0) {$\mathbf{Grph}_{n+1}$};
\draw[color=black, ->] (2.7,-0.5) -- (2.7,-2);
\node at (2.7,-2.5) {$\mathbf{Grph}_n$};
\draw[color=black, ->] (0,-0.5) -- (0,-2);
\node at (-0.6,-2.5) {$\mathbf{LinCat}_{n,p}$};
\draw[color=black, ->] (0.5,-2.5) -- (2,-2.5);
\node at (2.7,-2.5) {$\mathbf{Grph}_n$};
\draw[color=black] (0.3,-0.5) -- (0.55,-0.5) -- (0.55,-0.25);
\node at (1.25,-3) {$\mathcal{U}_{n,p}$};
\node at (3.2,-1.25) {$\mathcal{U}_n^G$};
} \end{array}
\end{gather*}
\end{definition}

\begin{definition} The free linear~$(0,0)$-category~$\Sigma^\ell$ over the 0-polygraph~$\Sigma$ is the vector space spanned by $\Sigma$. For~$n>0$, the free linear~$(n,n)$-category over the $n$-polygraph~$\Sigma$ is the linear~$(n,n)$-category~$\Sigma^\ell$ such that~$\Sigma^\ell_k=\Sigma^*_k$ for~$0 \leqslant k<n$ and~$\Sigma^\ell_n(u,v)$ is the vector space spanned by $\Sigma^*(u,v)$ for each $(n-1)$-cells $u$ and~$v$.
\end{definition}

\begin{definition} $\mathbf{LinPol}_{n,n}$ is the category of $n$-polygraphs and the functor~$\mathcal{F}_{n,n}$ from $\mathbf{LinPol}_{n,n}$ to~$\mathbf{LinCat}_{n,n}$ is the free functor from $\mathbf{LinPol}_{n,n}$ to~$\mathbf{LinCat}_{n,n}$. Let us assume the category~$\mathbf{LinPol}_{n,p}$ of linear~$(n,p)$-polygraphs and the functor~$\mathcal{F}_{n,p}$ from $\mathbf{LinPol}_{n,p}$ to~$\mathbf{LinCat}_{n,p}$ are defined. The category~$\mathbf{LinPol}_{n+1,p}$ is defined by the following pullback diagram:
\begin{gather*}
\begin{array}{c}
\tikz[scale=0.9]{
\node at (-0.5,0) {$\mathbf{LinPol}_{n+1,p}$};
\draw[color=black, ->] (0.7,0) -- (5,0);
\node at (5.9,0) {$\mathbf{Grph}_{n+1}$};
\draw[color=black, ->] (5.7,-0.5) -- (5.7,-2);
\node at (5.7,-2.5) {$\mathbf{Grph}_n$};
\draw[color=black, ->] (0,-0.5) -- (0,-2);
\node at (-0.5,-2.5) {$\mathbf{LinPol}_{n,p}$};
\draw[color=black, ->] (0.5,-2.5) -- (1.6,-2.5);
\node at (2.5,-2.5) {$\mathbf{LinCat}_{n,p}$};
\draw[color=black, ->] (3.4,-2.5) -- (5,-2.5);
\draw[color=black] (0.3,-0.5) -- (0.55,-0.5) -- (0.55,-0.25);
\node at (4,-3) {$\mathcal{U}_{n,p}$};
\node at (1,-3) {$\mathcal{F}_{n,p}$};
\node at (6.2,-1.25) {$\mathcal{U}_n^G$};
\node at (-0.5,-1.25) {$\mathcal{U}_{n,p}^P$};
\node at (2.85,0.5) {$\mathcal{U}_{n+1,p}^{GP}$};
} \end{array}
\end{gather*}
We denote by $\mathcal{F}_{n+1,p}^P$ the unique functor making the following diagram commutative:
\begin{gather*}
\begin{array}{c}
\tikz[scale=0.9]{
\node at (-0.6,0) {$\mathbf{LinCat}^+_{n,p}$};
\draw[color=black, ->] (0.5,0) -- (2,0);
\node at (2.9,0) {$\mathbf{Grph}_{n+1}$};
\draw[color=black, ->] (2.7,-0.5) -- (2.7,-2);
\node at (2.7,-2.5) {$\mathbf{Grph}_n$};
\draw[color=black, ->] (0,-0.5) -- (0,-2);
\node at (-0.6,-2.5) {$\mathbf{LinCat}_{n,p}$};
\draw[color=black, ->] (0.5,-2.5) -- (2,-2.5);
\node at (2.7,-2.5) {$\mathbf{Grph}_n$};
\draw[color=black] (0.3,-0.5) -- (0.55,-0.5) -- (0.55,-0.25);
\node at (1.25,-3) {$\mathcal{U}_{n,p}$};
\node at (3.2,-1.25) {$\mathcal{U}_n^G$};
\node at (-4,2) {$\mathbf{LinPol}_{n+1,p}$};
\draw[color=black, ->] (-2.6,1.5) -- (-0.6,0.4);
\draw[color=black, ->] (-4,1.5) -- (-4,-2);
\node at (-4,-2.5) {$\mathbf{LinPol}_{n,p}$};
\draw[color=black, ->] (-2.8,2) -- (2,0.4);
\node at (-2,0.6) {$\mathcal{F}_{n+1,p}^P$};
\draw[color=black, ->] (-3.1,-2.5) -- (-1.5,-2.5);
\node at (-2.4,-3) {$\mathcal{F}_{n,p}$};
\node at (-4.5,-0.25) {$\mathcal{U}_{n,p}^P$};
\node at (0.5,1.45) {$\mathcal{U}_{n+1,p}^{GP}$};
} \end{array}
\end{gather*}
The functor~$\mathcal{F}_{n+1,p}$ is defined as the following composite:
\begin{gather*}
\begin{array}{c}
\tikz[scale=0.9]{
\node at (-1,0) {$\mathbf{LinPol}_{n+1,p}$};
\draw[color=black, ->] (0.2,0) -- (2.2,0);
\node at (3.2,0) {$\mathbf{LinCat}_{n,p}^+$};
\draw[color=black, ->] (4,0) -- (6.2,0);
\node at (7.4,0) {$\mathbf{LinCat}_{n+1,p}$};
\node at (1.3,0.5) {$\mathcal{F}_{n+1,p}^P$};
\node at (5.1,0.5) {$\mathcal{F}_{n+1,p}^W$};
} \end{array}
\end{gather*}
\end{definition}

\begin{definition} Let $0 \leqslant p \leqslant n$. The free linear~$(n+1,p)$-category over a linear~$(n,p)$-category with globular extension~$(\mathcal{C},\Gamma)$ is the linear~$(n+1,p)$-category having~$\mathcal{C}$ as an underlying linear~$(n,p)$-category and which~$(n+1)$-cells are defined this way:
\begin{itemize}
\item we construct the set $A_1$ of all $(n+1)$-cells of the form $1_{u_1} \star_n \cdots  \star_0 \alpha \star_0 \cdots  \star_n 1_{u_{2n}}$ where each $u_k$ is in $\mathcal{C}_n$ and~$\alpha$ is in $\Gamma$,
\item we define the set of all formal $n$-compositions of elements of $A_1$ quotiented by the exchange relations (\ref{exch}) to obtain a set $A_2$,
\item if~$p=0$, the $(n+1)$-cells of the free linear~$(n+1,p)$-category are the linear combinations of elements of $A_2$ quotiented by Relations (\ref{idlin}), (\ref{complin1}) and (\ref{complin2}). if~$p>0$, the $(n+1)$-cells of the free linear~$(n+1,p)$-category are the linear combinations of $(p-1)$-composable elements of $A_2$ quotiented by Relations (\ref{idlin}), (\ref{complin1}) and (\ref{complin2}).
\end{itemize}
\end{definition}

A linear~$(n,n)$-polygraph is an $n$-polygraph. Given an $n$-polygraph~$\Sigma$, we denote by $\Sigma^*$ the free~$n$-category over~$\Sigma$. Let us assume linear~$(n,p)$-polygraphs are defined. a linear~$(n+1,p)$-polygraph  is the data of a linear~$(n,p)$-polygraph~$\Sigma$:
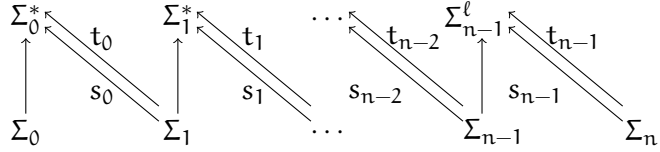
\begin{gather*}
\begin{array}{c}
\tikz[scale=0.5]{
\draw[color=black, ->] (0,0.5) -- (0,2.5);
\node at (0,0) {$\Sigma_0$};
\node at (0,3) {$\Sigma_0^*$};
\draw[color=black, ->] (3.5,0.3) -- (0.5,2.8);
\draw[color=black, ->] (3.5,0.7) -- (0.5,3.2);
\node at (2,1) {$s_0$};
\node at (2,2.5) {$t_0$};
\node at (4,0) {$\Sigma_1$};
\draw[color=black, ->] (4,0.5) -- (4,2.5);
\node at (4,3) {$\Sigma_1^*$};
\draw[color=black, ->] (7.5,0.3) -- (4.5,2.8);
\draw[color=black, ->] (7.5,0.7) -- (4.5,3.2);
\node at (6,1) {$s_1$};
\node at (6,2.5) {$t_1$};
\node at (8,0) {$\cdots$};
\node at (8,3) {$\cdots$};
\draw[color=black, ->] (11.5,0.3) -- (8.5,2.8);
\draw[color=black, ->] (11.5,0.7) -- (8.5,3.2);
\node at (9.2,1) {$s_{n-2}$};
\node at (10.2,2.5) {$t_{n-2}$};
\node at (12.3,0) {$\Sigma_{n-1}$};
\draw[color=black, ->] (12,0.5) -- (12,2.5);
\node at (11.8,3) {$\Sigma_{n-1}^\ell$};
\draw[color=black, ->] (15.7,0.3) -- (12.7,2.8);
\draw[color=black, ->] (15.7,0.7) -- (12.7,3.2);
\node at (13.4,1) {$s_{n-1}$};
\node at (14.4,2.5) {$t_{n-1}$};
\node at (16.2,0) {$\Sigma_n$};
} \end{array}
\end{gather*}
and a globular extension~$\Sigma_{n+1}$ of the free~$(n,p)$-category~$\Sigma^\ell$ over~$\Sigma$:
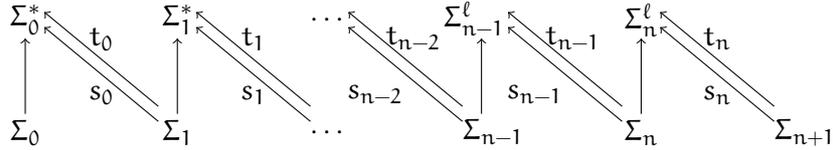
\begin{gather*}
\begin{array}{c}
\tikz[scale=0.5]{
\draw[color=black, ->] (0,0.5) -- (0,2.5);
\node at (0,0) {$\Sigma_0$};
\node at (0,3) {$\Sigma_0^*$};
\draw[color=black, ->] (3.5,0.3) -- (0.5,2.8);
\draw[color=black, ->] (3.5,0.7) -- (0.5,3.2);
\node at (2,1) {$s_0$};
\node at (2,2.5) {$t_0$};
\node at (4,0) {$\Sigma_1$};
\draw[color=black, ->] (4,0.5) -- (4,2.5);
\node at (4,3) {$\Sigma_1^*$};
\draw[color=black, ->] (7.5,0.3) -- (4.5,2.8);
\draw[color=black, ->] (7.5,0.7) -- (4.5,3.2);
\node at (6,1) {$s_1$};
\node at (6,2.5) {$t_1$};
\node at (8,0) {$\cdots$};
\node at (8,3) {$\cdots$};
\draw[color=black, ->] (11.5,0.3) -- (8.5,2.8);
\draw[color=black, ->] (11.5,0.7) -- (8.5,3.2);
\node at (9.2,1) {$s_{n-2}$};
\node at (10.2,2.5) {$t_{n-2}$};
\node at (12.3,0) {$\Sigma_{n-1}$};
\draw[color=black, ->] (12,0.5) -- (12,2.5);
\node at (11.8,3) {$\Sigma_{n-1}^\ell$};
\draw[color=black, ->] (15.7,0.3) -- (12.7,2.8);
\draw[color=black, ->] (15.7,0.7) -- (12.7,3.2);
\node at (13.4,1) {$s_{n-1}$};
\node at (14.4,2.5) {$t_{n-1}$};
\node at (16.2,0) {$\Sigma_n$};
\draw[color=black, ->] (16.2,0.5) -- (16.2,2.5);
\node at (16.2,3) {$\Sigma_n^\ell$};
\node at (20.5,0) {$\Sigma_{n+1}$};
\draw[color=black, ->] (19.7,0.3) -- (16.7,2.8);
\draw[color=black, ->] (19.7,0.7) -- (16.7,3.2);
\node at (18.2,1) {$s_n$};
\node at (18.2,2.5) {$t_n$};
} \end{array}
\end{gather*}

\section{Linear rewriting}

\subsection{Higher dimensional monomials}

We will consider~$n>0$ for this section.

In 2-categories, 2-cells can be represented as planar diagrams with an upper boundary and a lower boundary. The upper boundary will correspond to the 1-source of the 2-cell and the lower boundary will correspond to its 1-target. A generating 2-cell with 1-source $a$ and 1-target $b$ will be pictured as follows:
\begin{gather*}
\begin{array}{c}
\tikz[scale=0.5]{
\draw[color=black] (-0.5,-0.5) -- (-0.5,0.5) -- (0.5,0.5) -- (0.5,-0.5) -- (-0.5,-0.5);
\draw[color=black] (0.4,-0.5) -- (0.4,-1);
\draw[color=black] (-0.4,-0.5) -- (-0.4,-1);
\draw[color=black] (0.4,0.5) -- (0.4,1);
\draw[color=black] (-0.4,0.5) -- (-0.4,1);
\node at (0,0.75) {...};
\node at (0,-0.75) {...};
\node at (0,1.4) {a};
\node at (0,-1.4) {b};
} \end{array}
\end{gather*}
The 0-composition~$\alpha \star_0 \beta$ will be represented by horizontal concatenation
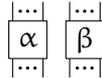
\begin{gather*}
\begin{array}{c}
\tikz[scale=0.5]{
\draw[color=black] (-0.5,-0.5) -- (-0.5,0.5) -- (0.5,0.5) -- (0.5,-0.5) -- (-0.5,-0.5);
\draw[color=black] (0.4,-0.5) -- (0.4,-1);
\draw[color=black] (-0.4,-0.5) -- (-0.4,-1);
\draw[color=black] (0.4,0.5) -- (0.4,1);
\draw[color=black] (-0.4,0.5) -- (-0.4,1);
\node at (0,0.75) {...};
\node at (0,-0.75) {...};
\node at (0,0) {$\alpha$};
\draw[color=black] (1,-0.5) -- (1,0.5) -- (2,0.5) -- (2,-0.5) -- (1,-0.5);
\draw[color=black] (1.9,-0.5) -- (1.9,-1);
\draw[color=black] (1.1,-0.5) -- (1.1,-1);
\draw[color=black] (1.9,0.5) -- (1.9,1);
\draw[color=black] (1.1,0.5) -- (1.1,1);
\node at (1.5,0.75) {...};
\node at (1.5,-0.75) {...};
\node at (1.5,0) {$\beta$};
} \end{array}
\end{gather*}
and the 1-composition~$\alpha \star_1 \beta$ will be represented by vertical concatenation
\begin{gather*}
\begin{array}{c}
\tikz[scale=0.5]{
\draw[color=black] (-0.5,-0.5) -- (-0.5,0.5) -- (0.5,0.5) -- (0.5,-0.5) -- (-0.5,-0.5);
\draw[color=black] (0.4,-0.5) -- (0.4,-1);
\draw[color=black] (-0.4,-0.5) -- (-0.4,-1);
\draw[color=black] (0.4,0.5) -- (0.4,1);
\draw[color=black] (-0.4,0.5) -- (-0.4,1);
\node at (0,0.75) {...};
\node at (0,0) {$\alpha$};
\draw[color=black] (-0.5,-2.5) -- (-0.5,-1.5) -- (0.5,-1.5) -- (0.5,-2.5) -- (-0.5,-2.5);
\draw[color=black] (0.4,-2.5) -- (0.4,-3);
\draw[color=black] (-0.4,-2.5) -- (-0.4,-3);
\draw[color=black] (0.4,-1.5) -- (0.4,-1);
\draw[color=black] (-0.4,-1.5) -- (-0.4,-1);
\node at (0,-1) {...};
\node at (0,-2.75) {...};
\node at (0,-2) {$\beta$};
} \end{array}
\end{gather*}
The exchange relation is diagramatically represented by:
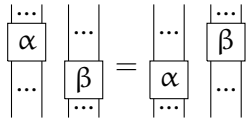
\begin{gather*}
\begin{array}{c}
\tikz[scale=0.5]{
\draw[color=black] (-0.5,-0.5) -- (-0.5,0.5) -- (0.5,0.5) -- (0.5,-0.5) -- (-0.5,-0.5);
\draw[color=black] (0.4,-0.5) -- (0.4,-2);
\draw[color=black] (-0.4,-0.5) -- (-0.4,-2);
\draw[color=black] (0.4,0.5) -- (0.4,1);
\draw[color=black] (-0.4,0.5) -- (-0.4,1);
\node at (0,0.75) {...};
\node at (0,-1.25) {...};
\node at (0,0) {$\alpha$};
\draw[color=black] (1,-1.5) -- (1,-0.5) -- (2,-0.5) -- (2,-1.5) -- (1,-1.5);
\draw[color=black] (1.9,-1.5) -- (1.9,-2);
\draw[color=black] (1.1,-1.5) -- (1.1,-2);
\draw[color=black] (1.9,-0.5) -- (1.9,1);
\draw[color=black] (1.1,-0.5) -- (1.1,1);
\node at (1.5,0.25) {...};
\node at (1.5,-1.75) {...};
\node at (1.5,-1) {$\beta$};
} \end{array}
=
\begin{array}{c}
\tikz[scale=0.5]{
\draw[color=black] (-0.5,-0.5) -- (-0.5,0.5) -- (0.5,0.5) -- (0.5,-0.5) -- (-0.5,-0.5);
\draw[color=black] (0.4,-0.5) -- (0.4,-1);
\draw[color=black] (-0.4,-0.5) -- (-0.4,-1);
\draw[color=black] (0.4,0.5) -- (0.4,2);
\draw[color=black] (-0.4,0.5) -- (-0.4,2);
\node at (0,1.25) {...};
\node at (0,-0.75) {...};
\node at (0,0) {$\alpha$};
\draw[color=black] (1,0.5) -- (1,1.5) -- (2,1.5) -- (2,0.5) -- (1,0.5);
\draw[color=black] (1.9,0.5) -- (1.9,-1);
\draw[color=black] (1.1,0.5) -- (1.1,-1);
\draw[color=black] (1.9,1.5) -- (1.9,2);
\draw[color=black] (1.1,1.5) -- (1.1,2);
\node at (1.5,1.75) {...};
\node at (1.5,-0.25) {...};
\node at (1.5,1) {$\beta$};
} \end{array}
\end{gather*}
If a 2-cell $\alpha$ verifies $s_1(\alpha)=t_1(\alpha)=a$, we use for any $k \in \mathbb{N}$ the following notations:
\begin{equation}\label{convention}
\begin{array}{c}
\tikz[scale=0.5]{
\draw[color=black] (-0.5,-0.5) -- (-0.5,0.5) -- (0.5,0.5) -- (0.5,-0.5) -- (-0.5,-0.5);
\draw[color=black] (0.4,-0.5) -- (0.4,-1);
\draw[color=black] (-0.4,-0.5) -- (-0.4,-1);
\draw[color=black] (0.4,0.5) -- (0.4,1);
\draw[color=black] (-0.4,0.5) -- (-0.4,1);
\node at (0,0.75) {...};
\node at (0,-0.75) {...};
\node at (0,0) {$\alpha$};
\node at (0.9,0.6) {0};
} \end{array}
=1_a~, \qquad
\begin{array}{c}
\tikz[scale=0.5]{
\draw[color=black] (-0.5,-0.5) -- (-0.5,0.5) -- (0.5,0.5) -- (0.5,-0.5) -- (-0.5,-0.5);
\draw[color=black] (0.4,-0.5) -- (0.4,-1);
\draw[color=black] (-0.4,-0.5) -- (-0.4,-1);
\draw[color=black] (0.4,0.5) -- (0.4,1);
\draw[color=black] (-0.4,0.5) -- (-0.4,1);
\node at (0,0.75) {...};
\node at (0,-0.75) {...};
\node at (0,0) {$\alpha$};
\node at (1.6,0.6) {$k+1$};
} \end{array}
=
\begin{array}{c}
\tikz[scale=0.5]{
\draw[color=black] (-0.5,-0.5) -- (-0.5,0.5) -- (0.5,0.5) -- (0.5,-0.5) -- (-0.5,-0.5);
\draw[color=black] (0.4,-0.5) -- (0.4,-1);
\draw[color=black] (-0.4,-0.5) -- (-0.4,-1);
\draw[color=black] (0.4,0.5) -- (0.4,1);
\draw[color=black] (-0.4,0.5) -- (-0.4,1);
\node at (0,0.75) {...};
\node at (0,0) {$\alpha$};
\draw[color=black] (-0.5,-2.5) -- (-0.5,-1.5) -- (0.5,-1.5) -- (0.5,-2.5) -- (-0.5,-2.5);
\draw[color=black] (0.4,-2.5) -- (0.4,-3);
\draw[color=black] (-0.4,-2.5) -- (-0.4,-3);
\draw[color=black] (0.4,-1.5) -- (0.4,-1);
\draw[color=black] (-0.4,-1.5) -- (-0.4,-1);
\node at (0,-1) {...};
\node at (0,-2.75) {...};
\node at (0,-2) {$\alpha$};
\node at (0.9,0.6) {$k$};
} \end{array}
\end{equation}

In linear~$(2,2)$-categories, linear combinations of such diagrams will be used to represent 2-cells. Such a representation will be used later.

\begin{example} Let $\mathcal{C}$ be a linear~$(2,2)$-category with only one 0-cell, one 1-cell~$x$ from and~$\alpha$ a 2-cell from $x$ to~$xx$. By representing~$1_x$ this way:
\[
\begin{tikzpicture} \begin{scope} [ x = 10pt, y = 10pt, join = round, cap = round ] \draw (0.00,0.50)--(0.00,0.00) ;  \draw (0.00,0.00)--(0.00,-0.50) ; \end{scope} \end{tikzpicture}
\]
we can represent $\alpha$ by:
\[
\begin{tikzpicture} \begin{scope} [ x = 10pt, y = 10pt, join = round, cap = round ] \draw (0.50,0.75)--(0.50,0.50) ; \draw [fill = lightgray] (0.50,0.50)--(1.00,0.00)--(0.00,0.00)--cycle ; \draw (0.00,0.00)--(0.00,-0.25) (1.00,0.00)--(1.00,-0.25) ; \end{scope} \end{tikzpicture}
\]
Then, the following diagrams:
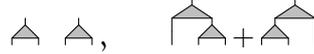
\[
\begin{tikzpicture} \begin{scope} [ x = 10pt, y = 10pt, join = round, cap = round ] \draw (0.50,0.75)--(0.50,0.50) (2.50,0.75)--(2.50,0.50) ; \draw [fill = lightgray] (0.50,0.50)--(1.00,0.00)--(0.00,0.00)--cycle ; \draw [fill = lightgray] (2.50,0.50)--(3.00,0.00)--(2.00,0.00)--cycle ; \draw (0.00,0.00)--(0.00,-0.25) (1.00,0.00)--(1.00,-0.25) (2.00,0.00)--(2.00,-0.25) (3.00,0.00)--(3.00,-0.25) ; \end{scope} \end{tikzpicture}~, \qquad \begin{tikzpicture} \begin{scope} [ x = 10pt, y = 10pt, join = round, cap = round ] \draw (0.75,1.25)--(0.75,1.00) ; \draw [fill = lightgray] (0.75,1.00)--(1.50,0.50)--(0.00,0.50)--cycle ; \draw (0.00,0.50)--(0.00,0.00) (1.50,0.50)--(1.50,0.25) ;  \draw [fill = lightgray] (1.50,0.25)--(2.00,-0.25)--(1.00,-0.25)--cycle ; \draw (0.00,0.00)--(0.00,-0.50) (1.00,-0.25)--(1.00,-0.50) (2.00,-0.25)--(2.00,-0.50) ; \end{scope} \end{tikzpicture}+\begin{tikzpicture} \begin{scope} [ x = 10pt, y = 10pt, join = round, cap = round ] \draw (1.25,1.50)--(1.25,1.25) ; \draw [fill = lightgray] (1.25,1.25)--(2.00,0.75)--(0.50,0.75)--cycle ; \draw (0.50,0.75)--(0.50,0.50) (2.00,0.75)--(2.00,0.25) ; \draw [fill = lightgray] (0.50,0.50)--(1.00,0.00)--(0.00,0.00)--cycle ;  \draw (0.00,0.00)--(0.00,-0.25) (1.00,0.00)--(1.00,-0.25) (2.00,0.25)--(2.00,-0.25) ; \end{scope} \end{tikzpicture}
\]
respectively correspond to~$\alpha \star_0 \alpha$ and~$\alpha \star_1 (1_x \star_0 \alpha)+\alpha \star_1 (\alpha  \star_0 1_x)$.
\end{example}

\begin{definition} Let $\Sigma$ be a linear~$(n,n-1)$-polygraph. A monomial of $\Sigma$ is a $(n-1)$-cell of the free~$(n-1)$-category~$\Sigma_{(n-1)}^*$. We say $\Sigma$ is \emph{left-monomial} if all elements of $\Sigma_n$ have monomial source.
\end{definition}

\begin{remark} Let $\Sigma$ be a linear~$(n,n-1)$-polygraph. Any $n$-cell of $\Sigma^\ell$ has a unique decomposition as a linear combination of monomials.
\end{remark}

\begin{definition} Let $\Sigma$ be an $n$-polygraph. A monomial order on~$\Sigma$ is a well-founded order~$\prec$ on~$\Sigma_n^*$ such that:
\begin{itemize}
\item for any $0 \leqslant k<n$ and any $n$-cells $a$,~$b$,~$u$ and~$v$, we have~$a \star_k u \star_k b \prec a \star_k v \star_k b$ if~$u \prec v$ whenever this makes sense,
\item $\prec$ is strict on~$\Sigma_n^*(a,b)$ for each $(n-1)$-cells $a$ and~$b$.
\end{itemize}
\end{definition}

An $n$-polygraph does not always have a monomial order.

\begin{example} Let $\Sigma$ be a 2-polygraph with only one 0-cell, one 1-cell and two 2-cells represented by:
\begin{gather*}
\begin{array}{c}
\tikz[scale=0.5]{
\draw[color=black] (0,0) arc (0:180:0.5);
\draw[color=black] (0,-0.5) -- (0,0);
\draw[color=black] (-1,-0.5) -- (-1,0);
} \end{array}~, \qquad
\begin{array}{c}
\tikz[scale=0.5]{
\draw[color=black] (0,0) arc (360:180:0.5);
\draw[color=black] (0,0.5) -- (0,0);
\draw[color=black] (-1,0.5) -- (-1,0);
} \end{array}
\end{gather*}

If there is a monomial order~$\prec$ on~$\Sigma$, we have one of the following inequalities:

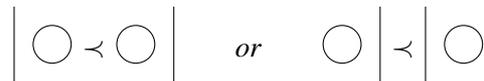
\begin{gather*}
\begin{array}{c}
\tikz[scale=0.5]{
\draw[color=black] (0,-1) -- (0,1);
\draw[color=black] (1,0) circle (0.5);
} \end{array}
\prec
\begin{array}{c}
\tikz[scale=0.5]{
\draw[color=black] (0,-1) -- (0,1);
\draw[color=black] (-1,0) circle (0.5);
} \end{array}
\qquad
\text{or}
\qquad
\begin{array}{c}
\tikz[scale=0.5]{
\draw[color=black] (0,-1) -- (0,1);
\draw[color=black] (-1,0) circle (0.5);
} \end{array}
\prec
\begin{array}{c}
\tikz[scale=0.5]{
\draw[color=black] (0,-1) -- (0,1);
\draw[color=black] (1,0) circle (0.5);
} \end{array}
\end{gather*}

Let us assume the first inequality is true, the other case being symmetric. Then, we have:

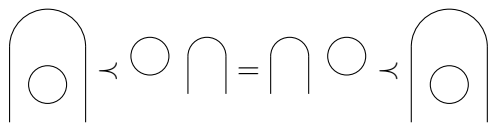
\begin{gather*}
\begin{array}{c}
\tikz[scale=0.5]{
\draw[color=black] (0,-1) -- (0,1);
\draw[color=black] (-2,-1) -- (-2,1);
\draw[color=black] (-1,0) circle (0.5);
\draw[color=black] (0,1) arc (0:180:1);
} \end{array}
\prec
\begin{array}{c}
\tikz[scale=0.5]{
\draw[color=black] (0,-1) -- (0,0);
\draw[color=black] (-1,-1) -- (-1,0);
\draw[color=black] (-2,0) circle (0.5);
\draw[color=black] (0,0) arc (0:180:0.5);
} \end{array}
=
\begin{array}{c}
\tikz[scale=0.5]{
\draw[color=black] (0,-1) -- (0,0);
\draw[color=black] (-1,-1) -- (-1,0);
\draw[color=black] (1,0) circle (0.5);
\draw[color=black] (0,0) arc (0:180:0.5);
} \end{array}
\prec
\begin{array}{c}
\tikz[scale=0.5]{
\draw[color=black] (0,-1) -- (0,1);
\draw[color=black] (-2,-1) -- (-2,1);
\draw[color=black] (-1,0) circle (0.5);
\draw[color=black] (0,1) arc (0:180:1);
} \end{array}
\end{gather*}
which contradicts the existence of a monomial order~$\Sigma_2^*$.
\end{example}

\subsection{Rewriting in linear~$(2,2)$-categories}

We recall that a linear~$(2,2)$-category is a category enriched in linear categories. We will explicit the rewriting systems arising from linear~$(3,2)$-polygraphs. A similar theory exists in the case of $n$-polygraphs \cite{GuiraudMalbos09}. The rewriting theory of linear~$(2,2)$-categories we will present is a linear adaptation of the case of 2-categories.

Let $\Sigma$ be a linear~$(3,2)$-polygraph. The \emph{congruence} generated by $\Sigma$ is the equivalence relation~$\equiv$ on~$\Sigma^\ell_2$ such that~$u \equiv v$ if there is a 3-cell~$\alpha$ in $\Sigma^\ell$ such that~$s_2(\alpha)=u$ and~$t_2(\alpha)=v$. Note that in a linear~$(3,2)$-category, all 3-cells are invertible. if~$\alpha$ is a 3-cell from a 2-cell~$u$ to a 2-cell~$v$, the 3-cell~$1_v+1_u-\alpha$ has 2-source $v$ and 2-target $u$. Let $\Sigma$ be a linear~$(3,2)$-polygraph. A linear~$(2,2)$-category is \emph{presented} by $\Sigma$ if it is isomorphic to the quotient of $\Sigma_2^\ell$ by the congruence generated by $\Sigma$.

We introduce next the notions of rewriting step of a left-monomial linear~$(3,2)$-polygraph to define branchings, termination and conflence in this context. We will fix $\Sigma$ a left-monomial linear~$(3,2)$-polygraph for the remainder of this section.

\begin{definition} A \emph{rewriting step} of $\Sigma$ is a 3-cell of $\Sigma^\ell$ of the form:

$$\lambda m_1 \star_1 (m_2 \star_0 s_2(\alpha) \star_0 m_3) \star_1 m_4 +u \Rrightarrow \lambda m_1 \star_1 (m_2 \star_0 s_2(\alpha) \star_0 m_3) \star_1 m_4 +u$$

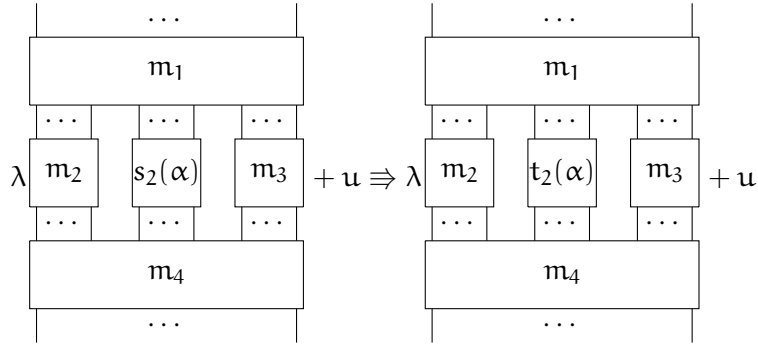
\begin{gather*}
\lambda
\begin{array}{c}
\tikz[scale=0.9]{
\draw[color=black] (-2,-1) -- (2,-1) -- (2,-2) -- (-2,-2) -- (-2,-1);
\draw[color=black] (-2,1) -- (2,1) -- (2,2) -- (-2,2) -- (-2,1);
\draw[color=black] (-1.9,-2) -- (-1.9,-2.5);
\draw[color=black] (1.9,-2) -- (1.9,-2.5);
\draw[color=black] (-1.9,2) -- (-1.9,2.5);
\draw[color=black] (1.9,2) -- (1.9,2.5);
\draw[color=black] (-0.5,-0.5) -- (0.5,-0.5) -- (0.5,0.5) -- (-0.5,0.5) -- (-0.5,-0.5);
\draw[color=black]  (-0.4,0.5) -- (-0.4,1);
\draw[color=black]  (0.4,0.5) -- (0.4,1);
\draw[color=black]  (-0.4,-0.5) -- (-0.4,-1);
\draw[color=black]  (0.4,-0.5) -- (0.4,-1);
\draw[color=black] (-2,-0.5) -- (-1,-0.5) -- (-1,0.5) -- (-2,0.5) -- (-2,-0.5);
\draw[color=black]  (-1.9,0.5) -- (-1.9,1);
\draw[color=black]  (-1.1,0.5) -- (-1.1,1);
\draw[color=black]  (-1.9,-0.5) -- (-1.9,-1);
\draw[color=black]  (-1.1,-0.5) -- (-1.1,-1);
\draw[color=black] (2,-0.5) -- (1,-0.5) -- (1,0.5) -- (2,0.5) -- (2,-0.5);
\draw[color=black]  (1.9,0.5) -- (1.9,1);
\draw[color=black]  (1.1,0.5) -- (1.1,1);
\draw[color=black]  (1.9,-0.5) -- (1.9,-1);
\draw[color=black]  (1.1,-0.5) -- (1.1,-1);
\node at (0,0) {$s_2(\alpha)$};
\node at (0,0.75) {$\cdots$};
\node at (0,-0.75) {$\cdots$};
\node at (1.5,0) {$m_3$};
\node at (1.5,0.75) {$\cdots$};
\node at (1.5,-0.75) {$\cdots$};
\node at (-1.5,0) {$m_2$};
\node at (-1.5,0.75) {$\cdots$};
\node at (-1.5,-0.75) {$\cdots$};
\node at (0,1.5) {$m_1$};
\node at (0,2.25) {$\cdots$};
\node at (0,-1.5) {$m_4$};
\node at (0,-2.25) {$\cdots$};
} \end{array}
+u
\Rrightarrow
\lambda
\begin{array}{c}
\tikz[scale=0.9]{
\draw[color=black] (-2,-1) -- (2,-1) -- (2,-2) -- (-2,-2) -- (-2,-1);
\draw[color=black] (-2,1) -- (2,1) -- (2,2) -- (-2,2) -- (-2,1);
\draw[color=black] (-1.9,-2) -- (-1.9,-2.5);
\draw[color=black] (1.9,-2) -- (1.9,-2.5);
\draw[color=black] (-1.9,2) -- (-1.9,2.5);
\draw[color=black] (1.9,2) -- (1.9,2.5);
\draw[color=black] (-0.5,-0.5) -- (0.5,-0.5) -- (0.5,0.5) -- (-0.5,0.5) -- (-0.5,-0.5);
\draw[color=black]  (-0.4,0.5) -- (-0.4,1);
\draw[color=black]  (0.4,0.5) -- (0.4,1);
\draw[color=black]  (-0.4,-0.5) -- (-0.4,-1);
\draw[color=black]  (0.4,-0.5) -- (0.4,-1);
\draw[color=black] (-2,-0.5) -- (-1,-0.5) -- (-1,0.5) -- (-2,0.5) -- (-2,-0.5);
\draw[color=black]  (-1.9,0.5) -- (-1.9,1);
\draw[color=black]  (-1.1,0.5) -- (-1.1,1);
\draw[color=black]  (-1.9,-0.5) -- (-1.9,-1);
\draw[color=black]  (-1.1,-0.5) -- (-1.1,-1);
\draw[color=black] (2,-0.5) -- (1,-0.5) -- (1,0.5) -- (2,0.5) -- (2,-0.5);
\draw[color=black]  (1.9,0.5) -- (1.9,1);
\draw[color=black]  (1.1,0.5) -- (1.1,1);
\draw[color=black]  (1.9,-0.5) -- (1.9,-1);
\draw[color=black]  (1.1,-0.5) -- (1.1,-1);
\node at (0,0) {$t_2(\alpha)$};
\node at (0,0.75) {$\cdots$};
\node at (0,-0.75) {$\cdots$};
\node at (1.5,0) {$m_3$};
\node at (1.5,0.75) {$\cdots$};
\node at (1.5,-0.75) {$\cdots$};
\node at (-1.5,0) {$m_2$};
\node at (-1.5,0.75) {$\cdots$};
\node at (-1.5,-0.75) {$\cdots$};
\node at (0,1.5) {$m_1$};
\node at (0,2.25) {$\cdots$};
\node at (0,-1.5) {$m_4$};
\node at (0,-2.25) {$\cdots$};
} \end{array}
+u
\end{gather*}
where $\alpha$ is in $\Sigma_3$, the $m_i$ are monomials,~$\lambda$ is a nonzero scalar and~$u$ is a 2-cell such that the monomial~$\lambda m_2 \star_1 (m_1 \alpha m_4) \star_1 m_3$ does not appear in the monomial decomposition of $u$. A rewriting sequence of $\Sigma$ is a finite or infinite sequence:
$$u_0 \Rrightarrow \cdots  \Rrightarrow u_n \Rrightarrow \cdots $$
of rewriting steps of $\Sigma$.
\end{definition}

For any 2-cells $u$ and~$v$, we say $u$ rewrites into~$v$ if there is a non-degenerate rewriting sequence from $u$ to~$v$. A 2-cell is a \emph{normal form} if it can not be rewritten. We say $\Sigma$ is normalizing if each 2-cell rewrites into a normal form or is a normal form.

A \emph{branching} of $\Sigma$ is a pair of rewriting sequences of $\Sigma$ with the same 2-source:

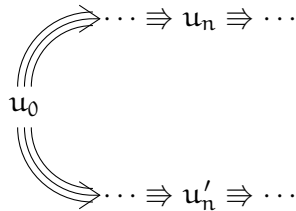
\begin{gather*}
\begin{array}{c}
\tikz[scale=0.9]{
\node at (0,0) {$u_0$};
\draw[color=black] (-0.1,0.3) arc (180:90:1.1);
\draw[color=black] (0,0.3) arc (180:90:1);
\draw[color=black] (0.1,0.3) arc (180:90:0.9);
\draw[color=black] (1,1.3) -- (1.1,1.3);
\draw[color=black] (0.8,1.1) -- (1.1,1.3) -- (0.8,1.5);
\node at (2.6,1.3) {$\cdots  \Rrightarrow u_n \Rrightarrow \cdots$};
\draw[color=black] (-0.1,-0.3) arc (180:270:1.1);
\draw[color=black] (0,-0.3) arc (180:270:1);
\draw[color=black] (0.1,-0.3) arc (180:270:0.9);
\draw[color=black] (1,-1.3) -- (1.1,-1.3);
\draw[color=black] (0.8,-1.1) -- (1.1,-1.3) -- (0.8,-1.5);
\node at (2.6,-1.3) {$\cdots  \Rrightarrow u'_n \Rrightarrow \cdots$};
} \end{array}
\end{gather*}

A \emph{local branching} of $\Sigma$ is a pair of rewriting  steps of $\Sigma$ with the same 2-source.

A branching~$(\alpha,\beta)$ is \emph{confluent} if it can be completed
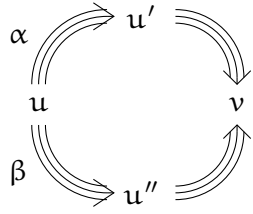
\begin{gather*}
\begin{array}{c}
\tikz[scale=0.9]{
\node at (0,0) {$u$};
\draw[color=black] (-0.1,0.3) arc (180:90:1.1);
\draw[color=black] (0,0.3) arc (180:90:1);
\draw[color=black] (0.1,0.3) arc (180:90:0.9);
\draw[color=black] (1,1.3) -- (1.1,1.3);
\draw[color=black] (0.8,1.1) -- (1.1,1.3) -- (0.8,1.5);
\node at (1.5,1.3) {$u'$};
\draw[color=black] (-0.1,-0.3) arc (180:270:1.1);
\draw[color=black] (0,-0.3) arc (180:270:1);
\draw[color=black] (0.1,-0.3) arc (180:270:0.9);
\draw[color=black] (1,-1.3) -- (1.1,-1.3);
\draw[color=black] (0.8,-1.1) -- (1.1,-1.3) -- (0.8,-1.5);
\node at (1.5,-1.3) {$u''$};
\node at (-0.3,1) {$\alpha$};
\node at (-0.3,-1) {$\beta$};
\draw[color=black] (2,1.4) arc (90:0:1);
\draw[color=black] (2,1.3) arc (90:0:0.9);
\draw[color=black] (2,1.2) arc (90:0:0.8);
\draw[color=black] (2.9,0.4) -- (2.9,0.3);
\draw[color=black] (2.7,0.5) -- (2.9,0.3) -- (3.1,0.5);
\node at (2.9,0) {$v$};
\draw[color=black] (2,-1.4) arc (270:360:1);
\draw[color=black] (2,-1.3) arc (270:360:0.9);
\draw[color=black] (2,-1.2) arc (270:360:0.8);
\draw[color=black] (2.9,-0.4) -- (2.9,-0.3);
\draw[color=black] (2.7,-0.5) -- (2.9,-0.3) -- (3.1,-0.5);
} \end{array}
\end{gather*}
into a branching made of rewriting sequences with the same 2-target. A linear~$(3,2)$-polygraph is \emph{confluent at} a 2-cell~$u \in \Sigma_2^\ell$ if all its branchings of source $u$ are confluent. A linear~$(3,2)$-polygraph is confluent if all its branchings are confluent. A linear~$(3,2)$-polygraph is \emph{locally confluent at} a 2-cell~$u \in \Sigma_2^\ell$ if all its local branchings of source $u$ are confluent. A linear~$(3,2)$-polygraph is \emph{locally confluent} if all its local branchings are confluent.

The linear~$(3,2)$-polygraph $\Sigma$ is \emph{terminating} if it has no infinite rewriting sequence.

The linear~$(3,2)$-polygraph $\Sigma$ is \emph{convergent} if it is terminating and confluent.

\begin{proposition}[Noetherian induction principle] Let $\Sigma$ be a terminating left-monomial  linear~$(3,2)$-polygraph. Let $P$ be a property on the 2-cells of $\Sigma^\ell$ such that for any $u \in \Sigma_2^\ell$, if for all $v$ such that~$u$ rewrites into~$v$, the property $P(v)$ is true,~$P(u)$ is true. Then,~$P(u)$ is true for any $u \in \Sigma_2^\ell$.
\end{proposition}

\begin{proof} Let us assume $P(u_0)$ is false for some $u_0 \in \Sigma_2^\ell$. So, there exists $u_1 \in \Sigma_2^\ell$ such that~$u_0$ rewrites into~$u_1$ and~$P(u_1)$ is false. This gives us an infinite rewriting sequence:
$$u_0 \Rrightarrow \cdots  \Rrightarrow u_n \Rrightarrow \cdots $$
where $P(u_n)$ is false for each $n \in \mathbb{N}$. The existence of this rewriting sequence contradicts the termination hypothesis. Noetherian induction principle is thus proved.
\end{proof}

Termination of a rewriting system implies equivalence between confluence and local confluence \cite{Newman42}. This result is called Newman's lemma. We have the same result in the case of linear~$(3,2)$-polygraphs.

\begin{lemma} Let $\Sigma$ be a terminating left-monomial linear~$(3,2)$-polygraph. Then $\Sigma$ is confluent if and only if~$\Sigma$ if locally confluent.
\end{lemma}

\begin{proof} Confluence implies local confluence. We have to prove the converse implication, assuming~$\Sigma$ is terminating. Let us assume $\Sigma$ is locally confluent and let $u$ be in $\Sigma_2^\ell$. if~$u$ is a normal form, all branchings with source $u$ are confluent because there is no branching with source $u$. if~$u$ is not a normal form, and if for any $v$ such that~$u$ rewrites into~$v$ all branchings with source $v$ are confluent, any non-trivial branching~$(\alpha,\beta)$ is confluent by the following confluence diagram:
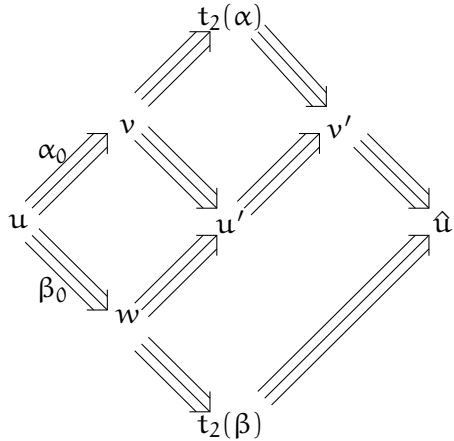
\begin{gather*}
\begin{array}{c}
\tikz[scale=0.9]{
\node at (0,0) {$u$};
\draw[color=black] (0.2,0.2) -- (1.3,1.3);
\draw[color=black] (0.1,0.3) -- (1.1,1.3);
\draw[color=black] (0.3,0.1) -- (1.3,1.1);
\draw[color=black] (1,1.3) -- (1.3,1.3) -- (1.3,1);
\node at (1.6,1.4) {$v$};
\node at (0.5,1) {$\alpha_0$};
\draw[color=black] (1.8,1.8) -- (2.8,2.8);
\draw[color=black] (1.7,1.9) -- (2.6,2.8);
\draw[color=black] (1.9,1.7) -- (2.8,2.6);
\draw[color=black] (2.5,2.8) -- (2.8,2.8) -- (2.8,2.5);
\node at (3.1,3) {$t_2(\alpha)$};
\draw[color=black] (0.2,-0.2) -- (1.3,-1.3);
\draw[color=black] (0.1,-0.3) -- (1.1,-1.3);
\draw[color=black] (0.3,-0.1) -- (1.3,-1.1);
\draw[color=black] (1,-1.3) -- (1.3,-1.3) -- (1.3,-1);
\node at (1.6,-1.4) {$w$};
\node at (0.5,-1) {$\beta_0$};
\draw[color=black] (1.8,-1.8) -- (2.8,-2.8);
\draw[color=black] (1.7,-1.9) -- (2.6,-2.8);
\draw[color=black] (1.9,-1.7) -- (2.8,-2.6);
\draw[color=black] (2.5,-2.8) -- (2.8,-2.8) -- (2.8,-2.5);
\node at (3.1,-3) {$t_2(\beta)$};
\draw[color=black] (2.9,0.2) -- (1.8,1.3);
\draw[color=black] (2.9,0.4) -- (1.9,1.4);
\draw[color=black] (2.7,0.2) -- (1.7,1.2);
\draw[color=black] (2.9,0.5) -- (2.9,0.2) -- (2.6,0.2);
\draw[color=black] (2.9,-0.2) -- (1.8,-1.3);
\draw[color=black] (2.9,-0.4) -- (1.9,-1.4);
\draw[color=black] (2.7,-0.2) -- (1.7,-1.2);
\draw[color=black] (2.9,-0.5) -- (2.9,-0.2) -- (2.6,-0.2);
\node at (3.1,0) {$u'$};
\draw[color=black] (4.5,1.7) -- (3.5,2.8);
\draw[color=black] (4.5,1.9) -- (3.6,2.9);
\draw[color=black] (4.3,1.7) -- (3.4,2.7);
\draw[color=black] (4.5,2) -- (4.5,1.7) -- (4.2,1.7);
\draw[color=black] (3.3,0.2) -- (4.4,1.3);
\draw[color=black] (3.2,0.3) -- (4.2,1.3);
\draw[color=black] (3.4,0.1) -- (4.4,1.1);
\draw[color=black] (4.1,1.3) -- (4.4,1.3) -- (4.4,1);
\node at (4.7,1.4) {$v'$};
\draw[color=black] (5,1.2) -- (6,0.2);
\draw[color=black] (5.1,1.3) -- (6,0.4);
\draw[color=black] (4.9,1.1) -- (5.8,0.2);
\draw[color=black] (5.7,0.2) -- (6,0.2) -- (6,0.5);
\draw[color=black] (5.8,-0.2) -- (3.5,-2.5);
\draw[color=black] (6,-0.2) -- (3.6,-2.6);
\draw[color=black] (6,-0.4) -- (3.7,-2.7);
\draw[color=black] (5.7,-0.2) -- (6,-0.2) -- (6,-0.5);
\node at (6.2,0) {$\hat{u}$};
} \end{array}
\end{gather*}

where $\alpha_0$ and~$\beta_0$ are respectively the first rewriting step of $\alpha$ and~$\beta$.
\end{proof}

This proof for linear~$(3,2)$-polygraphs by Noetherian induction works like in the case of abstract rewriting systems \cite{Huet}.

\begin{definition} An aspherical branching of $\Sigma$ is a branching made up of two identical rewriting steps. A Peiffer branching is a branching of the form:
$$t_2(\alpha) \star_1 s_2(\beta)+h \Lleftarrow s_2(\alpha) \star_1 s_2(\beta)+h \Rrightarrow s_2(\alpha) \star_1 t_2(\beta)+h$$
where $\alpha$ and~$\beta$ are rewriting steps of $\Sigma$ with monomial source. Note that all branchings of the form:
$$t_2(\alpha) \star_1 s_2(\beta)+h \Lleftarrow s_2(\alpha) \star_0 s_2(\beta)+h \Rrightarrow s_2(\alpha) \star_1 t_2(\beta)+h$$
are also Peiffer branchings beacause of the relation:
$$s_2(\alpha) \star_0 s_2(\beta)=(s_2(\alpha) \star_0 1_{s_1(\beta)}) \star_1 (1_{s_1(\alpha)} \star_0 s_2(\beta))$$

An \emph{additive branching} is a branching  of the form:
$$t_2(\alpha) + s_2(\beta) \Lleftarrow s_2(\alpha) + s_2(\beta) \Rrightarrow s_2(\alpha) + t_2(\beta)$$
where $\alpha$ and~$\beta$ are rewriting steps of $\Sigma$. An overlapping branching is a branching that is not aspherical, Peiffer or additive. \end{definition}

\subsubsection{Notation}\label{notaCP} Let us partition $\Sigma_3$ in multiple families and assume an overlapping branching $(\alpha ,\beta )$ of $\Sigma$ is confluent. If $\alpha$ is in the family $i$ and $\beta$ is in the family $j$, we denote:
$$u \Rrightarrow^i_j v$$
where $u$ is the source of $(\alpha ,\beta )$ and $v$ is a 2-cell attained after completion of $(\alpha ,\beta )$.

\begin{definition} Let $\sqsubseteq$ be the order on the monomials of $\Sigma$ such that we have $f \sqsubseteq g$ if we have~$g=m_1 \star_1 (m_2 \star_0 f \star_0 m_3) \star_1 m_4$ for some monomials $m_i$. A \emph{critical branching} is an overlapping branching with monomial source such that its source is minimal for~$\sqsubseteq$. \end{definition}

\begin{definition} A 3-cell of $\Sigma^\ell$ is \emph{elementary} if it is of the form $\lambda m_1 \star_1 (m_2 \star_0 \alpha \star_0 m_3) \star_1 m_4+u$ where $\alpha$ is in $\Sigma_3$, the $m_i$ are monomials,~$\lambda$ is a nonzero scalar, and~$u$ is a 2-cell. \end{definition}

\begin{example} Let $A$ and~$B$ be two monomials of $\Sigma^\ell$ such that there is a rewriting step from $A$ to~$B$. Then, there is a 3-cell of $\Sigma^\ell$ from $2A$ to~$A+B$ which is not a rewriting step. But this 3-cell is elementary. \end{example}

\begin{lemma}\label{AP} Let $\alpha$ be an elementary 3-cell. Then, there exists two rewriting sequences $\beta$ and~$\gamma$ of length at most 1 such that~$\alpha=\beta \star_2 \gamma^{-1}$. \end{lemma}

\begin{proof} Let us write $\alpha=\alpha'+g$ where $\alpha'$ is a rewriting step from a 2-cell~$u$ to a 2-cell~$f$ and~$g$ is a 2-cell. Let us write $g=\lambda u+h$ where $f$ does not appear in the monomial decomposition of $h$. Then,~$(\lambda+1)u+h$ rewrites into~$(\lambda+1)f+h$ either by an identity or a rewriting step and~$u+g$ rewrites into~$(\lambda+1)u+h$ either by an identity or a rewriting step.
\end{proof}

\begin{definition} Let $\Sigma$ be a left-monomial linear~$(3,2)$-polygraph. The \emph{rewrite order} of $\Sigma$ is the relation~$\preccurlyeq'_{\Sigma}$ on~$\Sigma_2^{\ell}$ defined by $v \preccurlyeq'_{\Sigma} u$ if~$u$ rewrites into~$v$ or~$u=v$. The \emph{canonical rewrite order} of $\Sigma$ is the minimal binary relation~$\preccurlyeq_{\Sigma}$ such that:
\begin{itemize}
\item if~$v \preccurlyeq'_{\Sigma} u$, then $v \preccurlyeq_{\Sigma} u$,
\item if~$v \preccurlyeq_{\Sigma} u$,~$v' \preccurlyeq_{\Sigma} u'$, and~$u$ and~$u'$ do not have any common monomial in their decomposition, then $v+v' \preccurlyeq_{\Sigma} u+u'$.
\end{itemize}
The \emph{strict canonical rewrite order} of $\Sigma$ is the strict order~$\prec_{\Sigma}$ defined for any 2-cells $u$ and~$v$ by $v \prec_{\Sigma} u$ if we have~$v \preccurlyeq_{\Sigma} u$ but not $u \preccurlyeq_{\Sigma} v$. The \emph{semistrict rewrite order} of $\Sigma$ is the binary relation~$\prec'_{\Sigma}$ on~$\Sigma_2^{\ell}$ defined by $v \prec'_{\Sigma} u$ if~$u$ rewrites into~$v$. \end{definition}

\begin{remark} if~$\Sigma$ is terminating, those relations are order relations and~$\prec_{\Sigma}$ is well-founded. In general, only~$\prec_{\Sigma}$ is an order. \end{remark}

\begin{lemma}\label{ing} Let $f$ be a 2-cell of a left-monomial linear~$(3,2)$-polygraph~$\Sigma$. Let us assume that all critical branchings of $\Sigma$ are confluent and~$\Sigma$ is locally confluent at every 2-cell~$g$ such that~$g \prec_{\Sigma} f$. Then,~$\Sigma$ is locally confluent at~$f$. \end{lemma}

\begin{proof} A local branching with source $f$ is either an aspherical branching, an overlapping branching, an additive branching or a Peiffer branching. Aspherical branchings are always confluent. Let $(\alpha,\beta)$ be an additive branching with source $f$. Let us write $f=s_2(\alpha')+s_2(\beta')$ where $\alpha'$ is the rewriting step~$\alpha-s_2(\beta)$ and~$\beta'$ is the rewriting step~$\beta-s_2(\alpha)$. If we do not have~$f \prec_{\Sigma} t_2(\alpha')+t_2(\beta')$, we have either~$t_2(\alpha')+s_2(\beta') \preccurlyeq'_{\Sigma} s_2(\alpha')+t_2(\beta')$ or~$s_2(\alpha')+t_2(\beta') \preccurlyeq'_{\Sigma} t_2(\alpha')+s_2(\beta')$, and~$(\alpha,\beta)$ is confluent. Else, we have the following confluence diagram:
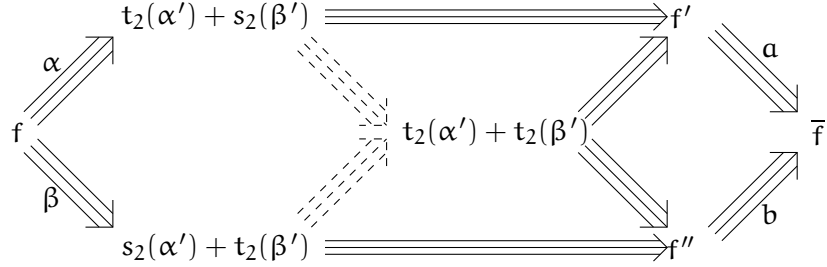
\begin{gather*}
\begin{array}{c}
\tikz[scale=0.9]{
\node at (0,0) {$f$};
\node at (0.5,1) {$\alpha$};
\node at (0.5,-1) {$\beta$};
\draw[color=black] (1.4,1) -- (1.4,1.4) -- (1,1.4);
\draw[color=black] (0.2,0.2) to (1.4,1.4);
\draw[color=black] (0.3,0.1) to (1.4,1.2);
\draw[color=black] (0.1,0.3) to (1.2,1.4);
\node at (2.9,1.7) {$t_2(\alpha')+s_2(\beta')$};
\draw[color=black] (1.4,-1) -- (1.4,-1.4) -- (1,-1.4);
\draw[color=black] (0.2,-0.2) to (1.4,-1.4);
\draw[color=black] (0.3,-0.1) to (1.4,-1.2);
\draw[color=black] (0.1,-0.3) to (1.2,-1.4);
\node at (2.9,-1.7) {$s_2(\alpha')+t_2(\beta')$};
\node at (7,0) {$t_2(\alpha')+t_2(\beta')$};
\draw[color=black, dashed] (5.4,-0.5) -- (5.4,-0.1) -- (5,-0.1);
\draw[color=black, dashed] (4.2,-1.3) to (5.4,-0.1);
\draw[color=black, dashed] (4.3,-1.4) to (5.4,-0.3);
\draw[color=black, dashed] (4.1,-1.2) to (5.2,-0.1);
\draw[color=black, dashed] (5.4,0.5) -- (5.4,0.1) -- (5,0.1);
\draw[color=black, dashed] (4.2,1.3) to (5.4,0.1);
\draw[color=black, dashed] (4.3,1.4) to (5.4,0.3);
\draw[color=black, dashed] (4.1,1.2) to (5.2,0.1);
\draw[color=black] (9.3,1.9) -- (9.5,1.7) -- (9.3,1.5);
\draw[color=black] (4.5,1.7) to (9.5,1.7);
\draw[color=black] (4.5,1.8) to (9.4,1.8);
\draw[color=black] (4.5,1.6) to (9.4,1.6);
\node at (9.7,1.7) {$f'$};
\draw[color=black] (9.3,-1.9) -- (9.5,-1.7) -- (9.3,-1.5);
\draw[color=black] (4.5,-1.7) to (9.5,-1.7);
\draw[color=black] (4.5,-1.8) to (9.4,-1.8);
\draw[color=black] (4.5,-1.6) to (9.4,-1.6);
\node at (9.7,-1.7) {$f''$};
\draw[color=black] (9.5,1) -- (9.5,1.4) -- (9.1,1.4);
\draw[color=black] (8.3,0.2) to (9.5,1.4);
\draw[color=black] (8.4,0.1) to (9.5,1.2);
\draw[color=black] (8.2,0.3) to (9.3,1.4);
\draw[color=black] (9.5,-1) -- (9.5,-1.4) -- (9.1,-1.4);
\draw[color=black] (8.3,-0.2) to (9.5,-1.4);
\draw[color=black] (8.4,-0.1) to (9.5,-1.2);
\draw[color=black] (8.2,-0.3) to (9.3,-1.4);
\draw[color=black] (11.4,-0.7) -- (11.4,-0.3) -- (11,-0.3);
\draw[color=black] (10.2,-1.5) to (11.4,-0.3);
\draw[color=black] (10.3,-1.6) to (11.4,-0.5);
\draw[color=black] (10.1,-1.4) to (11.2,-0.3);
\draw[color=black] (11.4,0.7) -- (11.4,0.3) -- (11,0.3);
\draw[color=black] (10.2,1.5) to (11.4,0.3);
\draw[color=black] (10.3,1.6) to (11.4,0.5);
\draw[color=black] (10.1,1.4) to (11.2,0.3);
\node at (11.7,0) {$\overline{f}$};
\node at (11,1.2) {$a$};
\node at (11,-1.2) {$b$};
} \end{array}
\end{gather*}
Here, the dotted 3-cells are elementary, the 3-cells $a$ and~$b$ are rewriting sequences, and the other 3-cells are rewriting steps or identities, whose existence are guaranteed by Lemma \ref{AP}.

Assume there is a Peiffer branching with source $f$. Let us write $f=u \star_1 v+h$ where $u$ and~$v$ are respectively monomial sources of rewriting steps $\alpha$ and~$\beta$. If we do not have~$f \prec_{\Sigma} t_2(\alpha) \star_1 t_2(\beta)+h$, we have either~$u \star_1 \beta +h \preccurlyeq'_{\Sigma} t_2(\alpha) \star_1 v+h$ or~$t_2(\alpha) \star_1 v+h \preccurlyeq'_{\Sigma} u \star_1 \beta +h$, and the Peiffer branching is confluent. Else, we have the following confluence diagram:
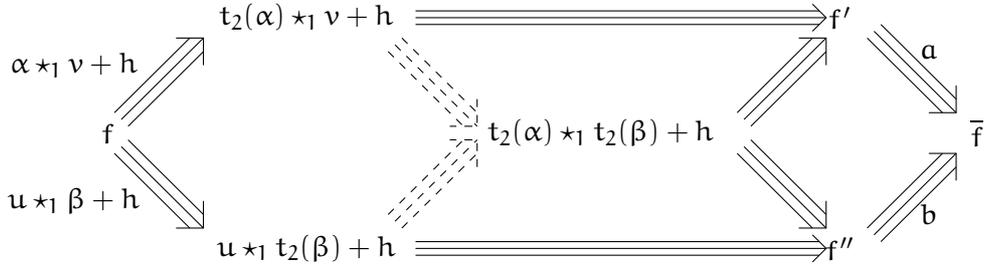
\begin{gather*}
\begin{array}{c}
\tikz[scale=0.9]{
\node at (0,0) {$f$};
\node at (-0.5,1) {$\alpha \star_1 v+h$};
\node at (-0.5,-1) {$u \star_1 \beta +h$};
\draw[color=black] (1.4,1) -- (1.4,1.4) -- (1,1.4);
\draw[color=black] (0.2,0.2) to (1.4,1.4);
\draw[color=black] (0.3,0.1) to (1.4,1.2);
\draw[color=black] (0.1,0.3) to (1.2,1.4);
\node at (2.9,1.7) {$t_2(\alpha) \star_1 v+h$};
\draw[color=black] (1.4,-1) -- (1.4,-1.4) -- (1,-1.4);
\draw[color=black] (0.2,-0.2) to (1.4,-1.4);
\draw[color=black] (0.3,-0.1) to (1.4,-1.2);
\draw[color=black] (0.1,-0.3) to (1.2,-1.4);
\node at (2.9,-1.7) {$u \star_1 t_2(\beta)+h$};
\node at (7.2,0) {$t_2(\alpha) \star_1 t_2(\beta)+h$};
\draw[color=black, dashed] (5.4,-0.5) -- (5.4,-0.1) -- (5,-0.1);
\draw[color=black, dashed] (4.2,-1.3) to (5.4,-0.1);
\draw[color=black, dashed] (4.3,-1.4) to (5.4,-0.3);
\draw[color=black, dashed] (4.1,-1.2) to (5.2,-0.1);
\draw[color=black, dashed] (5.4,0.5) -- (5.4,0.1) -- (5,0.1);
\draw[color=black, dashed] (4.2,1.3) to (5.4,0.1);
\draw[color=black, dashed] (4.3,1.4) to (5.4,0.3);
\draw[color=black, dashed] (4.1,1.2) to (5.2,0.1);
\draw[color=black] (10.3,1.9) -- (10.5,1.7) -- (10.3,1.5);
\draw[color=black] (4.5,1.7) to (10.5,1.7);
\draw[color=black] (4.5,1.8) to (10.4,1.8);
\draw[color=black] (4.5,1.6) to (10.4,1.6);
\node at (10.7,1.7) {$f'$};
\draw[color=black] (10.3,-1.9) -- (10.5,-1.7) -- (10.3,-1.5);
\draw[color=black] (4.5,-1.7) to (10.5,-1.7);
\draw[color=black] (4.5,-1.8) to (10.4,-1.8);
\draw[color=black] (4.5,-1.6) to (10.4,-1.6);
\node at (10.7,-1.7) {$f''$};
\draw[color=black] (10.5,1) -- (10.5,1.4) -- (10.1,1.4);
\draw[color=black] (9.3,0.2) to (10.5,1.4);
\draw[color=black] (9.4,0.1) to (10.5,1.2);
\draw[color=black] (9.2,0.3) to (10.3,1.4);
\draw[color=black] (10.5,-1) -- (10.5,-1.4) -- (10.1,-1.4);
\draw[color=black] (9.3,-0.2) to (10.5,-1.4);
\draw[color=black] (9.4,-0.1) to (10.5,-1.2);
\draw[color=black] (9.2,-0.3) to (10.3,-1.4);
\draw[color=black] (12.4,-0.7) -- (12.4,-0.3) -- (12,-0.3);
\draw[color=black] (11.2,-1.5) to (12.4,-0.3);
\draw[color=black] (11.3,-1.6) to (12.4,-0.5);
\draw[color=black] (11.1,-1.4) to (12.2,-0.3);
\draw[color=black] (12.4,0.7) -- (12.4,0.3) -- (12,0.3);
\draw[color=black] (11.2,1.5) to (12.4,0.3);
\draw[color=black] (11.3,1.6) to (12.4,0.5);
\draw[color=black] (11.1,1.4) to (12.2,0.3);
\node at (12.7,0) {$\overline{f}$};
\node at (12,1.2) {$a$};
\node at (12,-1.2) {$b$};
} \end{array}
\end{gather*}
Assume there is an overlapping branching with source $f$. Let us write $f=\lambda m_1 (m_2 \star_1 u \star_1 m_3) m_4+h$ where the $m_i$ are monomials,~$u$ is the source of a critical branching and~$\lambda$ is a nonzero scalar. The overlapping branching can be written $(\alpha +h, \beta +h)$ where $(\alpha , \beta)$ is confluent by hypothesis to a common target $g$. if~$t_2(\alpha)=t_2(\beta)$ or if we do not have~$f \prec_{\Sigma} g+h$, we have either~$t_2(\alpha)+h \preccurlyeq'_{\Sigma} t_2(\beta)+h$ or~$t_2(\beta)+h \preccurlyeq'_{\Sigma} t_2(\alpha)+h$, and the overlapping branching is confluent. Else, because $\prec_{\Sigma}$ is well-founded, there is a monomial~$m$ of $t_2(\alpha)-t_2(\beta)$ such that:
\begin{itemize}
\item $m$ can be rewritten into a linear combination of the others monomials of $t_2(\alpha)-t_2(\beta)$,
\item $m$ only appears in $t_2(\alpha)$ or~$t_2(\beta)$.
\end{itemize}
This makes the branching~$(\alpha +h, \beta +h)$ confluent and concludes the last case of local branching.
\end{proof}

\begin{theorem}\label{CP} Let $\Sigma$ be a left-monomial linear~$(3,2)$-polygraph such that~$\prec_{\Sigma}$ is well-founded. Then,~$\Sigma$ is locally confluent if and only if its critical branchings are confluent. \end{theorem}

\begin{proof} if~$\Sigma$ is locally confluent, then its critical branchings are confluent. Conversely, let us assume that all critical branchings of $\Sigma$ are confluent. Induction on~$\prec_{\Sigma}$ can be used to prove that~$\Sigma$ is confluent. Indeed,~$\Sigma$ is locally confluent at every minimal 2-cell for~$\prec_{\Sigma}$ and lemma \ref{ing} concludes the proof.
\end{proof}

\begin{lemma}\label{ker} Let $\Sigma$ be a confluent left-monomial linear~$(3,2)$-polygraph. Let $\mathcal{C}$ be the linear~$(2,2)$-category presented by $\Sigma$. Then, for any 1-cells $u$ and~$v$ of $\mathcal{C}$ with same 0-source and 0-target, the linear map $\tau$ from $\Sigma_2^\ell(u,v)$ to~$\mathcal{C}(u,v)$ sending each 2-cell to its congruence class has for kernel the subspace of $\Sigma_2^\ell(u,v)$ made of all 2-cells having~$0$ as a normal form. \end{lemma}

\begin{proof} $\Sigma$ being left-monomial,~$0$ is a normal form. If a 2-cell rewrites into~$0$, it is congruent to~$0$. Converesly, let us assume a non-zero 2-cell~$f$ is in $Ker(\tau)$. Confluence of $\Sigma$ makes $f$ rewrite into 0. This concludes the proof. \end{proof}

\begin{proposition}\label{basis} Let $\Sigma$ be a confluent and normalizing left-monomial linear~$(3,2)$-polygraph. Let $\mathcal{C}$ be the linear~$(2,2)$-category presented by $\Sigma$. Then, for any 1-cells $u$ and~$v$ of $\mathcal{C}$ with same 0-source and 0-target, the set of monomials of $\Sigma$ in normal form with 1-source $u$ and 1-target $v$ gives a basis of $\mathcal{C}(u,v)$. \end{proposition}

\begin{proof} Let us fix two 1-cells $u$ and~$v$ of $\mathcal{C}$ with same 0-source and 0-target. Every 2-cell of $\Sigma^\ell$ with 1-source $u$ and 1-target $v$ has a normal form because $\Sigma$ is terminating. Each normal form is a linear combination of monomials in normal form because $\Sigma$ is left-monomial. So, the family of monomials in normal form is generating.
The family of monomials in normal forms is free because $\Sigma$ is confluent, which allows us to use lemma \ref{ker}. This concludes the proof. \end{proof}

\subsection{Confluence by decreasingness}

We fix in this section~$\Sigma$ a left-monomial linear~$(3,2)$-polygraph.

\begin{definition}\label{decr} Let $R$ be the set of all rewriting steps of $\Sigma$. We say $\Sigma$ is decreasing if there is a well-founded order~$\prec$ on a partition~$\mathcal{P}$ of $R$ such that, for every~$I$ and~$J$ in $\mathcal{P}$, every local branching~$(i,j)$ with~$i \in I$ and~$j \in J$ can be completed into a confluence diagram
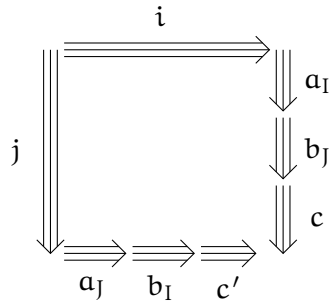
\begin{gather*}
\begin{array}{c}
\tikz[scale=0.9]{
\node at (-0.5,1.5) {$j$};
\draw (-0.1,3) -- (-0.1,0.1);
\draw (0,3) -- (0,0);
\draw (0.1,3) -- (0.1,0.1);
\draw (-0.2,0.2) -- (0,0) -- (0.2,0.2);
\node at (1.6,3.5) {$i$};
\draw (0.2,2.9) -- (3.1,2.9);
\draw (0.2,3) -- (3.2,3);
\draw (0.2,3.1) -- (3.1,3.1);
\draw (3,2.8) -- (3.2,3) -- (3,3.2);
\node at (3.9,2.5) {$a_I$};
\draw (3.3,3) -- (3.3,2.2);
\draw (3.4,3) -- (3.4,2.1);
\draw (3.5,3) -- (3.5,2.2);
\draw (3.2,2.3) -- (3.4,2.1) -- (3.6,2.3);
\node at (3.9,1.5) {$b_J$};
\draw (3.3,2) -- (3.3,1.2);
\draw (3.4,2) -- (3.4,1.1);
\draw (3.5,2) -- (3.5,1.2);
\draw (3.2,1.3) -- (3.4,1.1) -- (3.6,1.3);
\node at (3.9,0.5) {$c$};
\draw (3.3,1) -- (3.3,0.1);
\draw (3.4,1) -- (3.4,0);
\draw (3.5,1) -- (3.5,0.1);
\draw (3.2,0.2) -- (3.4,0) -- (3.6,0.2);
\node at (0.6,-0.5) {$a_J$};
\draw (0.2,-0.1) -- (1,-0.1);
\draw (0.2,0) -- (1.1,0);
\draw (0.2,0.1) -- (1,0.1);
\draw (0.9,-0.2) -- (1.1,0) -- (0.9,0.2);
\node at (1.6,-0.5) {$b_I$};
\draw (1.2,-0.1) -- (2,-0.1);
\draw (1.2,0) -- (2.1,0);
\draw (1.2,0.1) -- (2,0.1);
\draw (1.9,-0.2) -- (2.1,0) -- (1.9,0.2);
\node at (2.6,-0.5) {$c'$};
\draw (2.2,-0.1) -- (2.9,-0.1);
\draw (2.2,0) -- (3,0);
\draw (2.2,0.1) -- (2.9,0.1);
\draw (2.8,-0.2) -- (3,0) -- (2.8,0.2);
} \end{array}
\end{gather*}
such that:
\begin{itemize}
\item $a_I$ is a rewriting sequence such that for each rewriting step~$k$ in $a_I$, there exists $K$ in $\mathcal{P}$ such that~$k \in K$ and~$K \prec I$.
\item $a_J$ is a rewriting sequence such that for each rewriting step~$k$ in $a_J$, there exists $K$ in $\mathcal{P}$ such that~$k \in K$ and~$K \prec J$.
\item $b_I$ is either an identity or an element of $I$.
\item $b_J$ is either an identity or an element of $J$.
\item $c$ and~$c'$ are rewriting sequences such that for each rewriting step~$k$ in $c$ or~$c'$, there exists $K$ in $\mathcal{P}$ such that~$k \in K$ and~$K \prec I$ or~$K \prec J$.
\end{itemize}
\end{definition}

\begin{example}\label{exdecr} Let $\Sigma_{ex}$ be the linear~$(3,2)$-polygraph with only one 0-cell, one 1-cell~$a$, two 2-cells represented by:
\begin{gather*}
\begin{array}{c}
\tikz[scale=0.5]{
\draw[color=black] (0,0) arc (0:180:0.5);
\draw[color=black] (0,-0.5) -- (0,0);
\draw[color=black] (-1,-0.5) -- (-1,0);
} \end{array}~, \qquad
\begin{array}{c}
\tikz[scale=0.5]{
\draw[color=black] (0,0) arc (360:180:0.5);
\draw[color=black] (0,0.5) -- (0,0);
\draw[color=black] (-1,0.5) -- (-1,0);
} \end{array}
\end{gather*}
and the following 3-cell:
\begin{gather*}
\begin{array}{c}
\tikz[scale=0.5]{
\draw[color=black] (0,-1) -- (0,1);
\draw[color=black] (1,0) circle (0.5);
} \end{array}
\Rrightarrow
\begin{array}{c}
\tikz[scale=0.5]{
\draw[color=black] (0,-1) -- (0,1);
\draw[color=black] (-1,0) circle (0.5);
} \end{array}
\end{gather*}
A 2-cell is said in semi-normal form if it cannot be rewritten by using a rewriting step of one of the following forms:
\begin{gather*}
\begin{array}{c}
\tikz[scale=0.5]{
\draw[color=black] (0,-1) -- (0,0);
\draw[color=black] (-1,-1) -- (-1,0);
\draw[color=black] (1,0) circle (0.5);
\draw[color=black] (0,0) arc (0:180:0.5);
} \end{array}
\Rrightarrow
\begin{array}{c}
\tikz[scale=0.5]{
\draw[color=black] (0,-1) -- (0,1);
\draw[color=black] (-2,-1) -- (-2,1);
\draw[color=black] (-1,0) circle (0.5);
\draw[color=black] (0,1) arc (0:180:1);
} \end{array}~,
\qquad
\begin{array}{c}
\tikz[scale=0.5]{
\draw[color=black] (0,1) -- (0,0);
\draw[color=black] (-1,1) -- (-1,0);
\draw[color=black] (1,0) circle (0.5);
\draw[color=black] (0,0) arc (360:180:0.5);
} \end{array}
\Rrightarrow
\begin{array}{c}
\tikz[scale=0.5]{
\draw[color=black] (0,-1) -- (0,1);
\draw[color=black] (-2,-1) -- (-2,1);
\draw[color=black] (-1,0) circle (0.5);
\draw[color=black] (0,-1) arc (360:180:1);
} \end{array}
\end{gather*}
We give a partition~$\{ P_n|n \in \mathbb{N} \}$ of the set of rewriting steps of $\Sigma_{ex}$  such that, for any rewriting step~$\alpha$, we have~$\alpha \in P_n$ if the shortest rewriting sequence from $t_2(\alpha)$ to a semi-normal form is of length~$n$. This is a partition of the set of rewriting steps of $\Sigma_{ex}$ because every 2-cell rewrites into a semi-normal form. We define on~$\{ P_n|n \in \mathbb{N} \}$ the order~$\prec$ by $P_n \prec P_m$ if~$n<m$. With this well-ordered partition given,~$\Sigma_{ex}$ is decreasing. This translates the fact any 2-cell has a unique semi-normal form.
\end{example}

The following result is proved in \cite{VO} in the case of abstract rewriting systems. The proof can be adapted to the case of linear~$(3,2)$-polygraphs.

\begin{theorem}\label{VO} Let $\Sigma$ be a left-monomial libear $(3,2)$-polygraph. if~$\Sigma$ is decreasing, then $\Sigma$ is confluent. \end{theorem}

\begin{proof} Let us assume $\Sigma$ is decreasing for a partition~$\mathcal{P}$ of the set of its rewriting steps and an order~$\prec$ on~$\mathcal{P}$. We introduce the following map $|\cdots  |$ from the free monoid~$\mathcal{P}^*$ over~$\mathcal{P}$ to~$\mathbb{N}^{\mathcal{P}}$:
\begin{itemize}
\item if~$\epsilon$ is the empty word of $\mathcal{P}^*$, for every~$K$ in $\mathcal{P}$, we have~$|\epsilon |(K)=0$,
\item for every~$I$ in $\mathcal{P}$ and every~$K$ in $\mathcal{P}$, we have~$|I|(I)=1$ and~$|I|(K)=0$ if~$I \neq K$,
\item for every~$I$ in $\mathcal{P}$ and every word~$\sigma$ of $\mathcal{P}^*$, we have~$|I\sigma|=|I|+|\sigma^{(I)}|$ where $\sigma^{(I)}$ denotes the word~$\sigma$ without the letters $J$ such that~$J \prec I$.
\end{itemize}
We remark that for every words $\sigma_1$ and~$\sigma_2$ of $\mathcal{P}^*$, we have:
$$|\sigma_1\sigma_2|=|\sigma_1|+|\sigma_2|^{(\sigma_1)},$$
where:
$$|\sigma_2|^{(\sigma_1)}(K)=\begin{cases} 0 &\mbox{if there exists $I$ in $\mathcal{P}$ such that~$K \prec I$ and~$|\sigma_1|(I) \neq 0$} \\
|\sigma_2|(K) & \mbox{otherwise}\end{cases} $$

Then, we extend~$|\cdots  |$ to the set of finite rewriting sequences of $\Sigma$ by defining~$|r_1\cdots r_n|=|K_1\cdots K_n|$ for every rewriting sequence
\begin{gather*}
\begin{array}{c}
\tikz[scale=0.5]{
\node at (0,0) {$u_0$};
\node at (1,0) {$\Rrightarrow$};
\node at (1,0.5) {$r_1$};
\node at (2,0) {$\cdots$};
\node at (3,0) {$\Rrightarrow$};
\node at (3,0.5) {$r_n$};
\node at (4,0) {$u_n$};
} \end{array}
\end{gather*}
such that for each $1\leqslant i \leqslant n$ we have~$r_i \in K_i$.

We extend finally~$|\cdots  |$ to the set of branchings of $\Sigma$ made of finite rewriting sequences by defining~$|(\alpha,\beta)|=|\alpha|+|\beta|$ for every finite rewriting sequences $\alpha$ and~$\beta$.

We now define a strict order~$\prec'$ on~$\mathbb{N}^{\mathcal{P}}$. For any $M$ and~$N$ in $\mathbb{N}^{\mathcal{P}}$, we define $M\prec'N$ if there exist~$X$,~$Y$ and~$Z$ in $\mathbb{N}^{\mathcal{P}}$ such that:
\begin{itemize}
\item $M=Z+X$,
\item $N=Z+Y$,
\item $Y$ is not zero,
\item for every~$I$ in $\mathcal{P}$ such that~$M(I)\neq 0$, there exists $J$ in $\mathcal{P}$ such that~$N(J)\neq 0$ and~$I\prec J$.
\end{itemize}
The order~$\prec'$ is well-founded because $\prec$ is. We call $\preccurlyeq'$ the symmetric closure of $\prec'$.

To prove that~$\Sigma$ is confluent, it is sufficient to prove that every branching~$(\alpha,\beta)$ made of finite rewriting sequences can be completed into a confluence diagram $(\alpha \star_2\tau,\beta \star_2 \sigma)$
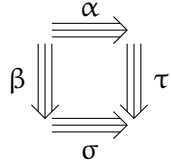
\begin{gather*}
\begin{array}{c}
\tikz[scale=0.9]{
\draw[color=black] (0,0) -- (0,-1.1);
\draw[color=black] (-0.1,0) -- (-0.1,-1);
\draw[color=black] (0.1,0) -- (0.1,-1);
\draw[color=black] (-0.2,-0.9) -- (0,-1.1) -- (0.2,-0.9);
\node at (-0.4,-0.55) {$\beta$};
\draw[color=black] (0.1,0.2) -- (1.2,0.2);
\draw[color=black] (0.1,0.1) -- (1.1,0.1);
\draw[color=black] (0.1,0.3) -- (1.1,0.3);
\draw[color=black] (1,0) -- (1.2,0.2) -- (1,0.4);
\node at (0.65,0.5) {$\alpha$};
\draw[color=black] (1.3,0) -- (1.3,-1.1);
\draw[color=black] (1.2,0) -- (1.2,-1);
\draw[color=black] (1.4,0) -- (1.4,-1);
\draw[color=black] (1.1,-0.9) -- (1.3,-1.1) -- (1.5,-0.9);
\node at (1.7,-0.55) {$\tau$};
\draw[color=black] (0.1,-1.2) -- (1.2,-1.2);
\draw[color=black] (0.1,-1.3) -- (1.1,-1.3);
\draw[color=black] (0.1,-1.1) -- (1.1,-1.1);
\draw[color=black] (1,-1.4) -- (1.2,-1.2) -- (1,-1);
\node at (0.65,-1.6) {$\sigma$};
} \end{array}
\end{gather*}
such that:
\begin{equation}\label{cnf2} |\alpha \star_2\tau| \preccurlyeq' |(\alpha,\beta)|, \end{equation}
\begin{equation}\label{cnf3} |\beta \star_2\sigma| \preccurlyeq' |(\alpha,\beta)|. \end{equation}
We will prove this fact by induction on~$|(\alpha,\beta)|$. This is trivial when $|(\alpha,\beta)|$ is minimal because this is the case of a trivial branching made of two identities.

Let us now assume that for each branching~$(\alpha',\beta')$ made of finite rewriting sequences such that~$|(\alpha',\beta')|\prec'|(\alpha,\beta)|$, we can complete $(\alpha',\beta')$ into a confluence diagram verifying (\ref{cnf2}) and (\ref{cnf3}). For every diagram of the following form:
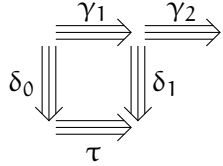
\begin{gather*}
\begin{array}{c}
\tikz[scale=0.9]{
\draw[color=black] (0,0) -- (0,-1.1);
\draw[color=black] (-0.1,0) -- (-0.1,-1);
\draw[color=black] (0.1,0) -- (0.1,-1);
\draw[color=black] (-0.2,-0.9) -- (0,-1.1) -- (0.2,-0.9);
\node at (-0.4,-0.55) {$\delta_0$};
\draw[color=black] (0.1,0.2) -- (1.2,0.2);
\draw[color=black] (0.1,0.1) -- (1.1,0.1);
\draw[color=black] (0.1,0.3) -- (1.1,0.3);
\draw[color=black] (1,0) -- (1.2,0.2) -- (1,0.4);
\node at (0.65,0.5) {$\gamma_1$};
\draw[color=black] (1.3,0) -- (1.3,-1.1);
\draw[color=black] (1.2,0) -- (1.2,-1);
\draw[color=black] (1.4,0) -- (1.4,-1);
\draw[color=black] (1.1,-0.9) -- (1.3,-1.1) -- (1.5,-0.9);
\node at (1.7,-0.55) {$\delta_1$};
\draw[color=black] (0.1,-1.2) -- (1.2,-1.2);
\draw[color=black] (0.1,-1.3) -- (1.1,-1.3);
\draw[color=black] (0.1,-1.1) -- (1.1,-1.1);
\draw[color=black] (1,-1.4) -- (1.2,-1.2) -- (1,-1);
\node at (0.65,-1.6) {$\tau$};
\draw[color=black] (1.4,0.2) -- (2.5,0.2);
\draw[color=black] (1.4,0.1) -- (2.4,0.1);
\draw[color=black] (1.4,0.3) -- (2.4,0.3);
\draw[color=black] (2.3,0) -- (2.5,0.2) -- (2.3,0.4);
\node at (1.95,0.5) {$\gamma_2$};
} \end{array}
\end{gather*}
such that~$|\gamma_1 \star_2 \delta_1| \preccurlyeq' |(\delta_0,\gamma_1)|$, we have~$|(\delta_1,\gamma_2)|\prec'|(\delta_0,\gamma_1 \star_2 \gamma_2)|$ if~$\gamma_1$ is not an identity. Indeed:
$$|(\delta_1,\gamma_2)|\prec' |\gamma_1|+|(\delta_1,\gamma_2)|^{(\gamma_1)}=|\gamma_1 \star_2 \delta_1|+|\gamma_2|^{(\gamma_1)} \preccurlyeq' |(\delta_0,\gamma_1)|+|\gamma_2|^{(\gamma_1)} = |(\delta_0,\gamma_1 \star_2 \gamma_2)|.$$

Let us also remark that for every diagram of the form:
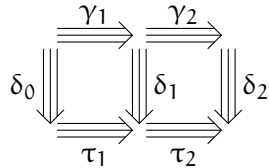
\begin{gather*}
\begin{array}{c}
\tikz[scale=0.9]{
\draw[color=black] (0,0) -- (0,-1.1);
\draw[color=black] (-0.1,0) -- (-0.1,-1);
\draw[color=black] (0.1,0) -- (0.1,-1);
\draw[color=black] (-0.2,-0.9) -- (0,-1.1) -- (0.2,-0.9);
\node at (-0.4,-0.55) {$\delta_0$};
\draw[color=black] (0.1,0.2) -- (1.2,0.2);
\draw[color=black] (0.1,0.1) -- (1.1,0.1);
\draw[color=black] (0.1,0.3) -- (1.1,0.3);
\draw[color=black] (1,0) -- (1.2,0.2) -- (1,0.4);
\node at (0.65,0.5) {$\gamma_1$};
\draw[color=black] (1.3,0) -- (1.3,-1.1);
\draw[color=black] (1.2,0) -- (1.2,-1);
\draw[color=black] (1.4,0) -- (1.4,-1);
\draw[color=black] (1.1,-0.9) -- (1.3,-1.1) -- (1.5,-0.9);
\node at (1.7,-0.55) {$\delta_1$};
\draw[color=black] (0.1,-1.2) -- (1.2,-1.2);
\draw[color=black] (0.1,-1.3) -- (1.1,-1.3);
\draw[color=black] (0.1,-1.1) -- (1.1,-1.1);
\draw[color=black] (1,-1.4) -- (1.2,-1.2) -- (1,-1);
\node at (0.65,-1.6) {$\tau_1$};
\draw[color=black] (1.4,0.2) -- (2.5,0.2);
\draw[color=black] (1.4,0.1) -- (2.4,0.1);
\draw[color=black] (1.4,0.3) -- (2.4,0.3);
\draw[color=black] (2.3,0) -- (2.5,0.2) -- (2.3,0.4);
\node at (1.95,0.5) {$\gamma_2$};
\draw[color=black] (2.6,0) -- (2.6,-1.1);
\draw[color=black] (2.5,0) -- (2.5,-1);
\draw[color=black] (2.7,0) -- (2.7,-1);
\draw[color=black] (2.4,-0.9) -- (2.6,-1.1) -- (2.8,-0.9);
\node at (3,-0.55) {$\delta_2$};
\draw[color=black] (1.4,-1.2) -- (2.5,-1.2);
\draw[color=black] (1.4,-1.3) -- (2.4,-1.3);
\draw[color=black] (1.4,-1.1) -- (2.4,-1.1);
\draw[color=black] (2.3,-1.4) -- (2.5,-1.2) -- (2.3,-1);
\node at (1.95,-1.6) {$\tau_2$};
} \end{array}
\end{gather*}
such that:
\begin{itemize}
\item $|\delta_0 \star_2 \tau_1| \preccurlyeq' |(\delta_0,\gamma_1)|$ and~$|\gamma_1 \star_2 \delta_1| \preccurlyeq' |(\delta_0,\gamma_1)|$,
\item $|\delta_1 \star_2 \tau_2| \preccurlyeq' |(\delta_1,\gamma_2)|$ and~$|\gamma_2 \star_2 \delta_2| \preccurlyeq' |(\delta_1,\gamma_2)|$
\end{itemize}
We have the \emph{pasting property} $|\delta_0 \star_2 \tau_1 \star_2 \tau_2| \preccurlyeq' |(\delta_0,\gamma_1 \star_2 \gamma_2)|$ and~$|\gamma_1 \star_2 \gamma_2 \star_2 \delta_2| \preccurlyeq' |(\delta_0,\gamma_1 \star_2 \gamma_2)|$. Indeed:
$$|\delta_0 \star_2 \tau_1 \star_2 \tau_2|=|\delta_0 \star_2 \tau_1|+|\tau_2|^{(\delta_0 \star_2 \tau_1)} \preccurlyeq' |(\delta_0,\gamma_1)|+|\tau_2|^{(\delta_0 \star_2 \tau_1)(\gamma_1)} \preccurlyeq' |(\delta_0,\gamma_1 \star_2 \gamma_2)|,$$
$$|\gamma_1 \star_2 \gamma_2 \star_2 \delta_2|=|\gamma_1 \star_2 \gamma_2|+|\delta_2|^{(\gamma_1)(\gamma_2)} \preccurlyeq' |\gamma_1 \star_2 \gamma_2|+|\delta_2|^{(\gamma_2)} \preccurlyeq' |(\delta_0,\gamma_1 \star_2 \gamma_2)|.$$

To prove~$(\alpha,\beta)$ is confluent when $\alpha$ and~$\beta$ have nonzero length, we consider~$\alpha_0$ the first step of $\alpha$ and~$\beta_0$ the first step of $\beta$. We have then the following confluence diagram:
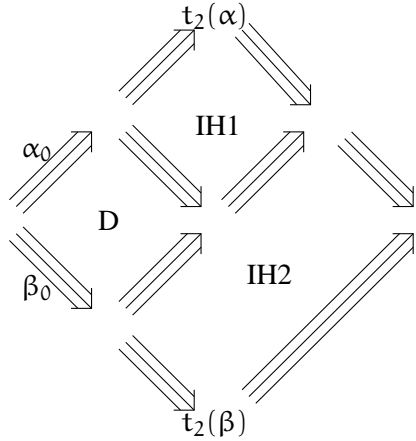
\begin{gather*}
\begin{array}{c}
\tikz[scale=0.9]{
\draw[color=black] (0.2,0.2) -- (1.3,1.3);
\draw[color=black] (0.1,0.3) -- (1.1,1.3);
\draw[color=black] (0.3,0.1) -- (1.3,1.1);
\draw[color=black] (1,1.3) -- (1.3,1.3) -- (1.3,1);
\node at (3.15,1.4) {IH1};
\node at (0.5,1) {$\alpha_0$};
\draw[color=black] (1.8,1.8) -- (2.8,2.8);
\draw[color=black] (1.7,1.9) -- (2.6,2.8);
\draw[color=black] (1.9,1.7) -- (2.8,2.6);
\draw[color=black] (2.5,2.8) -- (2.8,2.8) -- (2.8,2.5);
\node at (3.1,3) {$t_2(\alpha)$};
\draw[color=black] (0.2,-0.2) -- (1.3,-1.3);
\draw[color=black] (0.1,-0.3) -- (1.1,-1.3);
\draw[color=black] (0.3,-0.1) -- (1.3,-1.1);
\draw[color=black] (1,-1.3) -- (1.3,-1.3) -- (1.3,-1);
\node at (0.5,-1) {$\beta_0$};
\draw[color=black] (1.8,-1.8) -- (2.8,-2.8);
\draw[color=black] (1.7,-1.9) -- (2.6,-2.8);
\draw[color=black] (1.9,-1.7) -- (2.8,-2.6);
\draw[color=black] (2.5,-2.8) -- (2.8,-2.8) -- (2.8,-2.5);
\node at (3.1,-3) {$t_2(\beta)$};
\draw[color=black] (2.9,0.2) -- (1.8,1.3);
\draw[color=black] (2.9,0.4) -- (1.9,1.4);
\draw[color=black] (2.7,0.2) -- (1.7,1.2);
\draw[color=black] (2.9,0.5) -- (2.9,0.2) -- (2.6,0.2);
\draw[color=black] (2.9,-0.2) -- (1.8,-1.3);
\draw[color=black] (2.9,-0.4) -- (1.9,-1.4);
\draw[color=black] (2.7,-0.2) -- (1.7,-1.2);
\draw[color=black] (2.9,-0.5) -- (2.9,-0.2) -- (2.6,-0.2);
\node at (1.55,0) {D};
\draw[color=black] (4.5,1.7) -- (3.5,2.8);
\draw[color=black] (4.5,1.9) -- (3.6,2.9);
\draw[color=black] (4.3,1.7) -- (3.4,2.7);
\draw[color=black] (4.5,2) -- (4.5,1.7) -- (4.2,1.7);
\draw[color=black] (3.3,0.2) -- (4.4,1.3);
\draw[color=black] (3.2,0.3) -- (4.2,1.3);
\draw[color=black] (3.4,0.1) -- (4.4,1.1);
\draw[color=black] (4.1,1.3) -- (4.4,1.3) -- (4.4,1);
\draw[color=black] (5,1.2) -- (6,0.2);
\draw[color=black] (5.1,1.3) -- (6,0.4);
\draw[color=black] (4.9,1.1) -- (5.8,0.2);
\draw[color=black] (5.7,0.2) -- (6,0.2) -- (6,0.5);
\draw[color=black] (5.8,-0.2) -- (3.5,-2.5);
\draw[color=black] (6,-0.2) -- (3.6,-2.6);
\draw[color=black] (6,-0.4) -- (3.7,-2.7);
\draw[color=black] (5.7,-0.2) -- (6,-0.2) -- (6,-0.5);
\node at (3.9,-0.8) {IH2};
} \end{array}
\end{gather*}
Where D verifies (\ref{cnf2}) and (\ref{cnf3}) by decreasingness of $\Sigma$, the diagram IH1 exists because of the induction hypothesis and verifies (\ref{cnf2}) and (\ref{cnf3}) by the pasting property. Finally, the induction hypothesis allows us to construct IH2. The pasting property proves all the diagram verifies (\ref{cnf2}) and (\ref{cnf3}). This proves in particular by well-founded induction~$\Sigma$ is confluent. \end{proof}

\begin{example} The linear~$(3,2)$-polygraph~$\Sigma_{ex}$ from \ref{exdecr} is decreasing and is therefore confluent. Because every 2-cell rewrites into a semi-normal form and the semi-normal form of a 2-cell is unique, we easily conclude $\Sigma_{ex}$ is indeed confluent.
\end{example}


\section{The affine oriented Brauer category}

\subsection{Dotted oriented Brauer diagrams with bubbles}

After recalling the definition of the affine oriented Brauer category~$\mathcal{AOB}$ from \cite{BCNR}, we will present it by a linear~$(3,2)$-polygraph.

\subsubsection{Dotted oriented Brauer diagrams with bubbles}

A \emph{dotted oriented Brauer diagram with bubbles} is a planar diagram such that:
\begin{itemize}
\item edges are oriented,
\item edges are either bubbles or have a boundary as source and target,
\item each edge is decorated with an arbitrary number of dots not allowed to pass through the crossings.
\end{itemize}

\subsubsection{Equivalence of dotted oriented Brauer diagrams with bubbles}\label{equivdiag}

Two dotted oriented Brauer diagrams are \emph{equivalent} if one can be transformed into the other with isotopies and Reidemeister moves. A description of those moves can be found in \cite{Turaev}. An isotopy can move a dot along an edge but cannot make a dot go through a crossing. A dotted oriented Brauer diagram is \emph{normally ordered} if:
\begin{itemize}
\item all bubbles are clockwise,
\item all bubbles are in the leftmost side region,
\item all dots are either on a bubble or a segment pointing toward a boundary.
\end{itemize}

\begin{example} Here is an example of dotted oriented Brauer diagram:
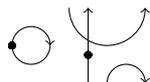
\begin{gather*}
\begin{array}{c}
\tikz[scale=0.5]{
\draw[color=black, ->] (0,0) -- (0,2);
\draw[color=black, ->] (0.5,0) arc (180:0:0.5);
\draw[color=black, ->] (-0.5,2) arc (180:360:1);
\draw[color=black, ->] (-1,1) arc (360:0:0.5);
\draw[thick] plot[mark=*] (-2,1);
\draw[thick] plot[mark=*] (0,0.75);
} \end{array}
\end{gather*}
This diagram is not normally ordered. The following one is normally ordered:
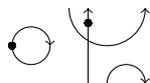
\begin{gather*}
\begin{array}{c}
\tikz[scale=0.5]{
\draw[color=black, ->] (0,0) -- (0,2);
\draw[color=black, ->] (0.5,0) arc (180:0:0.5);
\draw[color=black, ->] (-0.5,2) arc (180:360:1);
\draw[color=black, ->] (-1,1) arc (360:0:0.5);
\draw[thick] plot[mark=*] (-2,1);
\draw[thick] plot[mark=*] (0,1.6);
} \end{array}
\end{gather*}
Those diagrams are not equivalent.
\end{example}

Dotted oriented Brauer diagrams with bubbles will be described as 2-cells of a 2-category with vertical and horizontal concatenation respectively standing as 1-composition and 0-composition.

\begin{definition}\label{defAOB} The affine oriented Brauer category~$\mathcal{AOB}$ is the linear~$(2,2)$-category with one 0-cell, two generating 1-cells and with 2-cells from $a$ to~$b$ given by linear combinations of dotted oriented Brauer diagrams with bubbles with source $a$ and target $b$ subject to the following relations:
\begin{itemize}
\item invariance by equivalence given in \ref{equivdiag},
\item
\begin{gather*}
\begin{array}{c}
\tikz[scale=0.5]{
\draw[color=black, ->] (0,0) -- (1,2);
\draw[color=black, ->] (1,0) -- (0,2);
\draw[thick] plot[mark=*] (0.75,0.5);
} \end{array}
=
\begin{array}{c}
\tikz[scale=0.5]{
\draw[color=black, ->] (0,0) -- (1,2);
\draw[color=black, ->] (1,0) -- (0,2);
\draw[thick] plot[mark=*] (0.25,1.5);
} \end{array}
+
\begin{array}{c}
\tikz[scale=0.5]{
\draw[color=black, ->] (0,0) -- (0,2);
\draw[color=black, ->] (1,0) -- (1,2);
} \end{array}
\end{gather*}
\end{itemize}
\end{definition}

\subsubsection{An equational presentation of~$\mathcal{AOB}$} There is a presentation by generators and relations of the linear $(2,2)$-category~$\mathcal{AOB}$ given in \cite{BCNR}. The category~$\mathcal{AOB}$ is the linear~$(2,2)$-category with only one 0-cell, two generating 1-cells $\wedge$ and~$\vee$ and five generating 2-cells
$$1 \odfl{c} \wedge \vee, \qquad \vee \wedge \odfl{d} 1, \qquad \wedge \wedge \odfl{s} \wedge \wedge, \qquad \wedge \vee \odfl{t} \vee \wedge,  \qquad \wedge \odfl{x} \wedge$$
respectively represented by:
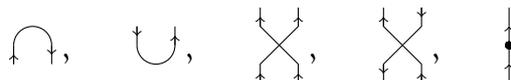
\begin{gather*}
\begin{array}{c}
\tikz[scale=0.5]{
\draw[color=black, ->] (0,0) -- (0,0.25);
\draw[color=black] (0,0) -- (0,0.5);
\draw[color=black, ->] (1,0.5) -- (1,0.25);
\draw[color=black] (1,0) -- (1,0.5);
\draw[color=black] (1,0.5) arc (0:180:0.5);
} \end{array}
~, \qquad
\begin{array}{c}
\tikz[scale=0.5]{
\draw[color=black, ->] (0,0.5) -- (0,0.25);
\draw[color=black] (0,0) -- (0,0.5);
\draw[color=black, ->] (1,0) -- (1,0.25);
\draw[color=black] (1,0) -- (1,0.5);
\draw[color=black] (1,0) arc (360:180:0.5);
} \end{array}
~, \qquad
\begin{array}{c}
\tikz[scale=0.5]{
\draw[color=black] (0,0) -- (0,0.5) -- (1,1.5) -- (1,2);
\draw[color=black] (1,0) -- (1,0.5) -- (0,1.5) -- (0,2);
\draw[color=black, ->] (0,0) -- (0,0.25);
\draw[color=black, ->] (1,1.5) -- (1,1.75);
\draw[color=black, ->] (1,0) -- (1,0.25);
\draw[color=black, ->] (0,1.5) -- (0,1.75);
} \end{array}
~, \qquad
\begin{array}{c}
\tikz[scale=0.5]{
\draw[color=black] (0,0) -- (0,0.5) -- (1,1.5) -- (1,2);
\draw[color=black] (1,0) -- (1,0.5) -- (0,1.5) -- (0,2);
\draw[color=black, ->] (0,0.5) -- (0,0.25);
\draw[color=black, ->] (1,2) -- (1,1.75);
\draw[color=black, ->] (1,0) -- (1,0.25);
\draw[color=black, ->] (0,1.5) -- (0,1.75);
} \end{array}
~, \qquad
\begin{array}{c}
\tikz[scale=0.5]{
\draw[color=black, ->] (0,0) -- (0,0.5);
\draw[color=black, ->] (0,0.5) -- (0,1.5);
\draw[color=black] (0,1.5) -- (0,2);
\draw[thick] plot[mark=*] (0,1);
} \end{array}
\end{gather*}
subject to the following relations: 
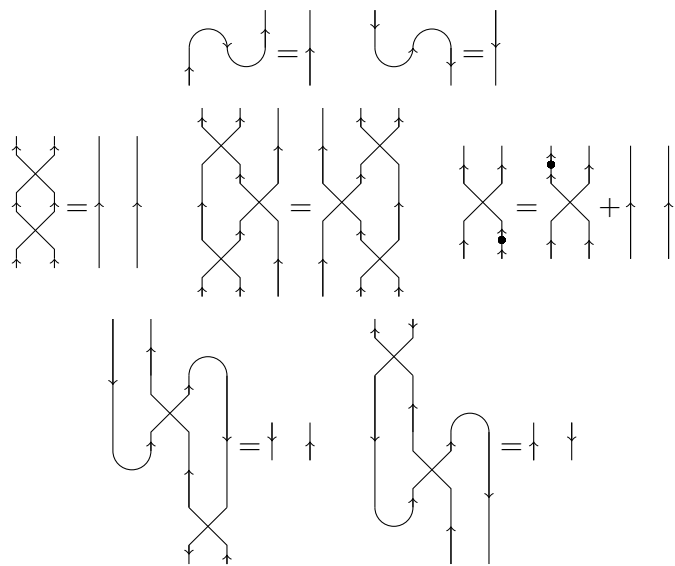
\begin{gather*}
\begin{array}{c}
\tikz[scale=0.5]{
\draw[color=black] (-1,0) -- (-1,1);
\draw[color=black] (1,1) -- (1,2);
\draw[color=black] (-1,1) arc (180:0:0.5);
\draw[color=black] (1,1) arc (360:180:0.5);
\draw[color=black, ->] (-1,0) -- (-1,0.5);
\draw[color=black, ->] (1,1) -- (1,1.5);
\draw[color=black, ->] (0,1.01) -- (0,1);
} \end{array}
=
\begin{array}{c}
\tikz[scale=0.5]{
\draw[color=black] (0,0) -- (0,2);
\draw[color=black, ->] (0,0) -- (0,1);
} \end{array}
\qquad
\begin{array}{c}
\tikz[scale=0.5]{
\draw[color=black] (-1,0) -- (-1,-1);
\draw[color=black] (1,-1) -- (1,-2);
\draw[color=black] (-1,-1) arc (180:360:0.5);
\draw[color=black] (1,-1) arc (0:180:0.5);
\draw[color=black, ->] (-1,0) -- (-1,-0.5);
\draw[color=black, ->] (1,-1) -- (1,-1.5);
\draw[color=black, ->] (0,-1.01) -- (0,-1);
} \end{array}
=
\begin{array}{c}
\tikz[scale=0.5]{
\draw[color=black] (0,0) -- (0,-2);
\draw[color=black, ->] (0,0) -- (0,-1);
} \end{array} \\
\begin{array}{c}
\tikz[scale=0.5]{
\draw[color=black, ->] (0,0) -- (0,0.25);
\draw[color=black, ->] (1,0) -- (1,0.25);
\draw[color=black, ->] (0,0.25) -- (0,0.5) -- (1,1.5) -- (1,1.75);
\draw[color=black, ->] (1,0.25) -- (1,0.5) -- (0,1.5) -- (0,1.75);
\draw[color=black, ->] (0,1.75) -- (0,2) -- (1,3) -- (1,3.25);
\draw[color=black, ->] (1,1.75) -- (1,2) -- (0,3) -- (0,3.25);
\draw[color=black] (1,3.5) -- (1,3.25);
\draw[color=black] (0,3.5) -- (0,3.25);
} \end{array}
=
\begin{array}{c}
\tikz[scale=0.5]{
\draw[color=black, ->] (0,0) -- (0,1.75);
\draw[color=black] (0,1.75) -- (0,3.5);
\draw[color=black, ->] (1,0) -- (1,1.75);
\draw[color=black] (1,1.75) -- (1,3.5);
} \end{array}
\qquad
\begin{array}{c}
\tikz[scale=0.5]{
\draw[color=black] (-1,0) -- (-1,0.5) -- (0,1.5) -- (0,2) -- (1,3) -- (1,5);
\draw[color=black, ->] (-1,0) -- (-1,0.25);
\draw[color=black, ->] (0,1.5) -- (0,1.75);
\draw[color=black, ->] (1,3) -- (1,4);
\draw[color=black] (0,0) -- (0,0.5) -- (-1,1.5) -- (-1,3.5) -- (0,4.5) -- (0,5);
\draw[color=black, ->] (0,0) -- (0,0.25);
\draw[color=black, ->] (-1,1.5) -- (-1,2.5);
\draw[color=black, ->] (0,4.5) -- (0,4.75);
\draw[color=black] (1,0) -- (1,2) -- (0,3) -- (0,3.5) -- (-1,4.5) -- (-1,5);
\draw[color=black, ->] (1,0) -- (1,1);
\draw[color=black, ->] (0,3) -- (0,3.25);
\draw[color=black, ->] (-1,4.5) -- (-1,4.75);
} \end{array}
=
\begin{array}{c}
\tikz[scale=0.5]{
\draw[color=black] (1,0) -- (1,0.5) -- (0,1.5) -- (0,2) -- (-1,3) -- (-1,5);
\draw[color=black, ->] (1,0) -- (1,0.25);
\draw[color=black, ->] (0,1.5) -- (0,1.75);
\draw[color=black, ->] (-1,3) -- (-1,4);
\draw[color=black] (0,0) -- (0,0.5) -- (1,1.5) -- (1,3.5) -- (0,4.5) -- (0,5);
\draw[color=black, ->] (0,0) -- (0,0.25);
\draw[color=black, ->] (1,1.5) -- (1,2.5);
\draw[color=black, ->] (0,4.5) -- (0,4.75);
\draw[color=black] (-1,0) -- (-1,2) -- (0,3) -- (0,3.5) -- (1,4.5) -- (1,5);
\draw[color=black, ->] (-1,0) -- (-1,1);
\draw[color=black, ->] (0,3) -- (0,3.25);
\draw[color=black, ->] (1,4.5) -- (1,4.75);
} \end{array}
\qquad
\begin{array}{c}
\tikz[scale=0.5]{
\draw[color=black] (0,0) -- (0,1) -- (1,2) -- (1,3);
\draw[color=black] (1,0) -- (1,1) -- (0,2) -- (0,3);
\draw[color=black, ->] (1,0) -- (1,0.25);
\draw[thick] plot[mark=*] (1,0.5);
\draw[color=black, ->] (0,0) -- (0,0.5);
\draw[color=black, ->] (1,0.5) -- (1,0.8);
\draw[color=black, ->] (0,2) -- (0,2.5);
\draw[color=black, ->] (1,2) -- (1,2.5);
} \end{array}
=
\begin{array}{c}
\tikz[scale=0.5]{
\draw[color=black] (0,0) -- (0,1) -- (1,2) -- (1,3);
\draw[color=black] (1,0) -- (1,1) -- (0,2) -- (0,3);
\draw[color=black, ->] (0,2.5) -- (0,2.8);
\draw[thick] plot[mark=*] (0,2.5);
\draw[color=black, ->] (0,0) -- (0,0.5);
\draw[color=black, ->] (1,0) -- (1,0.5);
\draw[color=black, ->] (0,2) -- (0,2.25);
\draw[color=black, ->] (1,2) -- (1,2.5);
} \end{array}
+
\begin{array}{c}
\tikz[scale=0.5]{
\draw[color=black] (0,0) -- (0,3);
\draw[color=black] (1,0) -- (1,3);
\draw[color=black, ->] (0,0) -- (0,1.5);
\draw[color=black, ->] (1,0) -- (1,1.5);
} \end{array} \\
\begin{array}{c}
\tikz[scale=0.5]{
\draw[color=black] (0,0) -- (0,2) -- (-1,3) -- (-1,5);
\draw[color=black] (1,0) -- (1,3.5);
\draw[color=black] (-1,1.5) -- (-1,2) -- (0,3) -- (0,3.5);
\draw[color=black] (-2,5) -- (-2,1.5);
\draw[color=black, ->] (0,0) -- (0,1);
\draw[color=black, ->] (-1,3.5) -- (-1,4.25);
\draw[color=black, ->] (-1,1.5) -- (-1,1.75);
\draw[color=black, ->] (0,3) -- (0,3.25);
\draw[color=black, ->] (-2,5) -- (-2,3.25);
\draw[color=black, ->] (1,3.5) -- (1,1.75);
\draw[color=black] (0,-1.5) -- (0,-1) -- (1,0);
\draw[color=black] (1,-1.5) -- (1,-1) -- (0,0);
\draw[color=black, ->] (0,-1) -- (0,-1.25);
\draw[color=black, ->] (1,-1.5) -- (1,-1.25);
\draw[color=black] (1,3.5) arc (0:180:0.5);
\draw[color=black] (-1,1.5) arc (360:180:0.5);
} \end{array}
=
\begin{array}{c}
\tikz[scale=0.5]{
\draw[color=black] (0,0) -- (0,1);
\draw[color=black] (1,0) -- (1,1);
\draw[color=black, ->] (0,1) -- (0,0.5);
\draw[color=black, ->] (1,0) -- (1,0.5);
} \end{array}
\qquad
\begin{array}{c}
\tikz[scale=0.5]{
\draw[color=black] (0,0) -- (0,2) -- (-1,3) -- (-1,5);
\draw[color=black] (1,0) -- (1,3.5);
\draw[color=black] (-1,1.5) -- (-1,2) -- (0,3) -- (0,3.5);
\draw[color=black] (-2,5) -- (-2,1.5);
\draw[color=black, ->] (0,0) -- (0,1);
\draw[color=black, ->] (-1,3.5) -- (-1,4.25);
\draw[color=black, ->] (-1,1.5) -- (-1,1.75);
\draw[color=black, ->] (0,3) -- (0,3.25);
\draw[color=black, ->] (-2,5) -- (-2,3.25);
\draw[color=black, ->] (1,3.5) -- (1,1.75);
\draw[color=black] (-2,6.5) -- (-2,6) -- (-1,5);
\draw[color=black] (-1,6.5) -- (-1,6) -- (-2,5);
\draw[color=black, ->] (-2,6) -- (-2,6.25);
\draw[color=black, ->] (-1,6.5) -- (-1,6.25);
\draw[color=black] (1,3.5) arc (0:180:0.5);
\draw[color=black] (-1,1.5) arc (360:180:0.5);
} \end{array}
=
\begin{array}{c}
\tikz[scale=0.5]{
\draw[color=black] (0,0) -- (0,1);
\draw[color=black] (1,0) -- (1,1);
\draw[color=black, ->] (0,0) -- (0,0.5);
\draw[color=black, ->] (1,1) -- (1,0.5);
} \end{array}
\end{gather*}

From this equational presentation of the linear~$(2,2)$-category~$\mathcal{AOB}$, we deduce a linear~$(3,2)$-polygraph presenting~$\mathcal{AOB}$.

\subsubsection{The linear~$(3,2)$-polygraph~$\mathrm{AOB}$} Let us define the linear~$(3,2)$-polygraph~$\mathrm{AOB}$. It is the $(3,2)$-polygraph with only one 0-cell and:
\begin{itemize}
\item $\mathrm{AOB}_1=\{ \wedge ,\vee \}$,
\item $\mathrm{AOB}_2=\big\{~
\begin{array}{c}
\tikz[scale=0.5]{
\draw[color=black, ->] (0,0) -- (0,0.25);
\draw[color=black] (0,0) -- (0,0.5);
\draw[color=black, ->] (1,0.5) -- (1,0.25);
\draw[color=black] (1,0) -- (1,0.5);
\draw[color=black] (1,0.5) arc (0:180:0.5);
} \end{array}
~, \qquad
\begin{array}{c}
\tikz[scale=0.5]{
\draw[color=black, ->] (0,0.5) -- (0,0.25);
\draw[color=black] (0,0) -- (0,0.5);
\draw[color=black, ->] (1,0) -- (1,0.25);
\draw[color=black] (1,0) -- (1,0.5);
\draw[color=black] (1,0) arc (360:180:0.5);
} \end{array}
~, \qquad
\begin{array}{c}
\tikz[scale=0.5]{
\draw[color=black] (0,0) -- (0,0.5) -- (1,1.5) -- (1,2);
\draw[color=black] (1,0) -- (1,0.5) -- (0,1.5) -- (0,2);
\draw[color=black, ->] (0,0) -- (0,0.25);
\draw[color=black, ->] (1,1.5) -- (1,1.75);
\draw[color=black, ->] (1,0) -- (1,0.25);
\draw[color=black, ->] (0,1.5) -- (0,1.75);
} \end{array}
~, \qquad
\begin{array}{c}
\tikz[scale=0.5]{
\draw[color=black] (0,0) -- (0,0.5) -- (1,1.5) -- (1,2);
\draw[color=black] (1,0) -- (1,0.5) -- (0,1.5) -- (0,2);
\draw[color=black, ->] (0,0.5) -- (0,0.25);
\draw[color=black, ->] (1,2) -- (1,1.75);
\draw[color=black, ->] (1,0) -- (1,0.25);
\draw[color=black, ->] (0,1.5) -- (0,1.75);
} \end{array}
~, \qquad
\begin{array}{c}
\tikz[scale=0.5]{
\draw[color=black, ->] (0,0) -- (0,0.5);
\draw[color=black, ->] (0,0.5) -- (0,1.5);
\draw[color=black] (0,1.5) -- (0,2);
\draw[thick] plot[mark=*] (0,1);
} \end{array}~ \big\}$,
\item $\mathrm{AOB}_3$ is made of the following 3-cells:
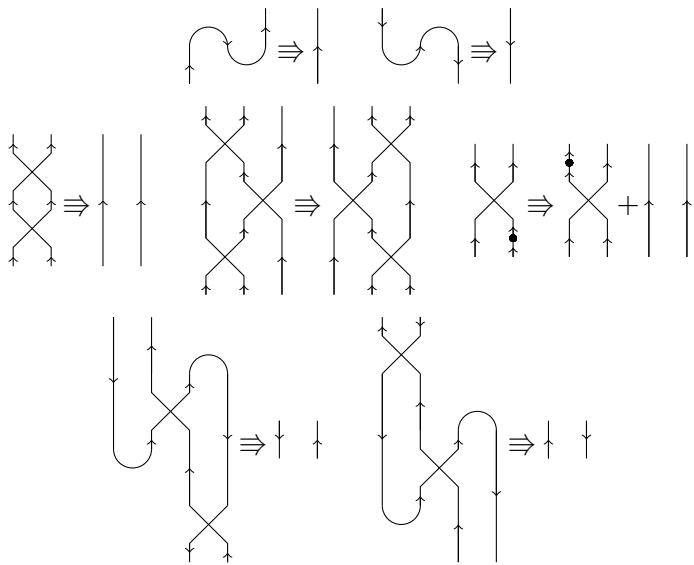
\begin{gather*}
\begin{array}{c}
\tikz[scale=0.5]{
\draw[color=black] (-1,0) -- (-1,1);
\draw[color=black] (1,1) -- (1,2);
\draw[color=black] (-1,1) arc (180:0:0.5);
\draw[color=black] (1,1) arc (360:180:0.5);
\draw[color=black, ->] (-1,0) -- (-1,0.5);
\draw[color=black, ->] (1,1) -- (1,1.5);
\draw[color=black, ->] (0,1.01) -- (0,1);
} \end{array}
\Rrightarrow
\begin{array}{c}
\tikz[scale=0.5]{
\draw[color=black] (0,0) -- (0,2);
\draw[color=black, ->] (0,0) -- (0,1);
} \end{array}
\qquad
\begin{array}{c}
\tikz[scale=0.5]{
\draw[color=black] (-1,0) -- (-1,-1);
\draw[color=black] (1,-1) -- (1,-2);
\draw[color=black] (-1,-1) arc (180:360:0.5);
\draw[color=black] (1,-1) arc (0:180:0.5);
\draw[color=black, ->] (-1,0) -- (-1,-0.5);
\draw[color=black, ->] (1,-1) -- (1,-1.5);
\draw[color=black, ->] (0,-1.01) -- (0,-1);
} \end{array}
\Rrightarrow
\begin{array}{c}
\tikz[scale=0.5]{
\draw[color=black] (0,0) -- (0,-2);
\draw[color=black, ->] (0,0) -- (0,-1);
} \end{array} \\
\begin{array}{c}
\tikz[scale=0.5]{
\draw[color=black, ->] (0,0) -- (0,0.25);
\draw[color=black, ->] (1,0) -- (1,0.25);
\draw[color=black, ->] (0,0.25) -- (0,0.5) -- (1,1.5) -- (1,1.75);
\draw[color=black, ->] (1,0.25) -- (1,0.5) -- (0,1.5) -- (0,1.75);
\draw[color=black, ->] (0,1.75) -- (0,2) -- (1,3) -- (1,3.25);
\draw[color=black, ->] (1,1.75) -- (1,2) -- (0,3) -- (0,3.25);
\draw[color=black] (1,3.5) -- (1,3.25);
\draw[color=black] (0,3.5) -- (0,3.25);
} \end{array}
\Rrightarrow
\begin{array}{c}
\tikz[scale=0.5]{
\draw[color=black, ->] (0,0) -- (0,1.75);
\draw[color=black] (0,1.75) -- (0,3.5);
\draw[color=black, ->] (1,0) -- (1,1.75);
\draw[color=black] (1,1.75) -- (1,3.5);
} \end{array}
\qquad
\begin{array}{c}
\tikz[scale=0.5]{
\draw[color=black] (-1,0) -- (-1,0.5) -- (0,1.5) -- (0,2) -- (1,3) -- (1,5);
\draw[color=black, ->] (-1,0) -- (-1,0.25);
\draw[color=black, ->] (0,1.5) -- (0,1.75);
\draw[color=black, ->] (1,3) -- (1,4);
\draw[color=black] (0,0) -- (0,0.5) -- (-1,1.5) -- (-1,3.5) -- (0,4.5) -- (0,5);
\draw[color=black, ->] (0,0) -- (0,0.25);
\draw[color=black, ->] (-1,1.5) -- (-1,2.5);
\draw[color=black, ->] (0,4.5) -- (0,4.75);
\draw[color=black] (1,0) -- (1,2) -- (0,3) -- (0,3.5) -- (-1,4.5) -- (-1,5);
\draw[color=black, ->] (1,0) -- (1,1);
\draw[color=black, ->] (0,3) -- (0,3.25);
\draw[color=black, ->] (-1,4.5) -- (-1,4.75);
} \end{array}
\Rrightarrow
\begin{array}{c}
\tikz[scale=0.5]{
\draw[color=black] (1,0) -- (1,0.5) -- (0,1.5) -- (0,2) -- (-1,3) -- (-1,5);
\draw[color=black, ->] (1,0) -- (1,0.25);
\draw[color=black, ->] (0,1.5) -- (0,1.75);
\draw[color=black, ->] (-1,3) -- (-1,4);
\draw[color=black] (0,0) -- (0,0.5) -- (1,1.5) -- (1,3.5) -- (0,4.5) -- (0,5);
\draw[color=black, ->] (0,0) -- (0,0.25);
\draw[color=black, ->] (1,1.5) -- (1,2.5);
\draw[color=black, ->] (0,4.5) -- (0,4.75);
\draw[color=black] (-1,0) -- (-1,2) -- (0,3) -- (0,3.5) -- (1,4.5) -- (1,5);
\draw[color=black, ->] (-1,0) -- (-1,1);
\draw[color=black, ->] (0,3) -- (0,3.25);
\draw[color=black, ->] (1,4.5) -- (1,4.75);
} \end{array}
\qquad
\begin{array}{c}
\tikz[scale=0.5]{
\draw[color=black] (0,0) -- (0,1) -- (1,2) -- (1,3);
\draw[color=black] (1,0) -- (1,1) -- (0,2) -- (0,3);
\draw[color=black, ->] (1,0) -- (1,0.25);
\draw[thick] plot[mark=*] (1,0.5);
\draw[color=black, ->] (0,0) -- (0,0.5);
\draw[color=black, ->] (1,0.5) -- (1,0.8);
\draw[color=black, ->] (0,2) -- (0,2.5);
\draw[color=black, ->] (1,2) -- (1,2.5);
} \end{array}
\Rrightarrow
\begin{array}{c}
\tikz[scale=0.5]{
\draw[color=black] (0,0) -- (0,1) -- (1,2) -- (1,3);
\draw[color=black] (1,0) -- (1,1) -- (0,2) -- (0,3);
\draw[color=black, ->] (0,2.5) -- (0,2.8);
\draw[thick] plot[mark=*] (0,2.5);
\draw[color=black, ->] (0,0) -- (0,0.5);
\draw[color=black, ->] (1,0) -- (1,0.5);
\draw[color=black, ->] (0,2) -- (0,2.25);
\draw[color=black, ->] (1,2) -- (1,2.5);
} \end{array}
+
\begin{array}{c}
\tikz[scale=0.5]{
\draw[color=black] (0,0) -- (0,3);
\draw[color=black] (1,0) -- (1,3);
\draw[color=black, ->] (0,0) -- (0,1.5);
\draw[color=black, ->] (1,0) -- (1,1.5);
} \end{array} \\
\begin{array}{c}
\tikz[scale=0.5]{
\draw[color=black] (0,0) -- (0,2) -- (-1,3) -- (-1,5);
\draw[color=black] (1,0) -- (1,3.5);
\draw[color=black] (-1,1.5) -- (-1,2) -- (0,3) -- (0,3.5);
\draw[color=black] (-2,5) -- (-2,1.5);
\draw[color=black, ->] (0,0) -- (0,1);
\draw[color=black, ->] (-1,3.5) -- (-1,4.25);
\draw[color=black, ->] (-1,1.5) -- (-1,1.75);
\draw[color=black, ->] (0,3) -- (0,3.25);
\draw[color=black, ->] (-2,5) -- (-2,3.25);
\draw[color=black, ->] (1,3.5) -- (1,1.75);
\draw[color=black] (0,-1.5) -- (0,-1) -- (1,0);
\draw[color=black] (1,-1.5) -- (1,-1) -- (0,0);
\draw[color=black, ->] (0,-1) -- (0,-1.25);
\draw[color=black, ->] (1,-1.5) -- (1,-1.25);
\draw[color=black] (1,3.5) arc (0:180:0.5);
\draw[color=black] (-1,1.5) arc (360:180:0.5);
} \end{array}
\Rrightarrow
\begin{array}{c}
\tikz[scale=0.5]{
\draw[color=black] (0,0) -- (0,1);
\draw[color=black] (1,0) -- (1,1);
\draw[color=black, ->] (0,1) -- (0,0.5);
\draw[color=black, ->] (1,0) -- (1,0.5);
} \end{array}
\qquad
\begin{array}{c}
\tikz[scale=0.5]{
\draw[color=black] (0,0) -- (0,2) -- (-1,3) -- (-1,5);
\draw[color=black] (1,0) -- (1,3.5);
\draw[color=black] (-1,1.5) -- (-1,2) -- (0,3) -- (0,3.5);
\draw[color=black] (-2,5) -- (-2,1.5);
\draw[color=black, ->] (0,0) -- (0,1);
\draw[color=black, ->] (-1,3.5) -- (-1,4.25);
\draw[color=black, ->] (-1,1.5) -- (-1,1.75);
\draw[color=black, ->] (0,3) -- (0,3.25);
\draw[color=black, ->] (-2,5) -- (-2,3.25);
\draw[color=black, ->] (1,3.5) -- (1,1.75);
\draw[color=black] (-2,6.5) -- (-2,6) -- (-1,5);
\draw[color=black] (-1,6.5) -- (-1,6) -- (-2,5);
\draw[color=black, ->] (-2,6) -- (-2,6.25);
\draw[color=black, ->] (-1,6.5) -- (-1,6.25);
\draw[color=black] (1,3.5) arc (0:180:0.5);
\draw[color=black] (-1,1.5) arc (360:180:0.5);
} \end{array}
\Rrightarrow
\begin{array}{c}
\tikz[scale=0.5]{
\draw[color=black] (0,0) -- (0,1);
\draw[color=black] (1,0) -- (1,1);
\draw[color=black, ->] (0,0) -- (0,0.5);
\draw[color=black, ->] (1,1) -- (1,0.5);
} \end{array}
\end{gather*}
\end{itemize}

This linear $(3,2)$-polygraph presents the linear~$(2,2)$-category $\mathcal{AOB}$. This presentation is not confluent. For example, the following  critical branching is not confluent:
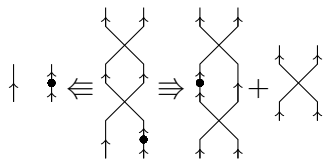
\begin{gather*}
\begin{array}{c}
\tikz[scale=0.5]{
\draw[color=black] (0,0) -- (0,1);
\draw[color=black] (1,0) -- (1,1);
\draw[color=black, ->] (1,0) -- (1,0.25);
\draw[thick] plot[mark=*] (1,0.5);
\draw[color=black, ->] (0,0) -- (0,0.5);
\draw[color=black, ->] (1,0.5) -- (1,0.8);
} \end{array}
\Lleftarrow
\begin{array}{c}
\tikz[scale=0.5]{
\draw[color=black] (0,0) -- (0,1) -- (1,2) -- (1,2.5) -- (0,3.5) -- (0,4);
\draw[color=black] (1,0) -- (1,1) -- (0,2) -- (0,2.5) -- (1,3.5) -- (1,4);
\draw[color=black, ->] (0,0) -- (0,0.5);
\draw[color=black, ->] (0,2) -- (0,2.25);
\draw[color=black, ->] (0,3.5) -- (0,3.75);
\draw[color=black, ->] (1,0) -- (1,0.25);
\draw[color=black, ->] (1,0.5) -- (1,0.8);
\draw[color=black, ->] (1,2) -- (1,2.25);
\draw[color=black, ->] (1,3.5) -- (1,3.75);
\draw[thick] plot[mark=*] (1,0.5);
} \end{array}
\Rrightarrow
\begin{array}{c}
\tikz[scale=0.5]{
\draw[color=black] (0,0) -- (0,0.5) -- (1,1.5) -- (1,2.5) -- (0,3.5) -- (0,4);
\draw[color=black] (1,0) -- (1,0.5) -- (0,1.5) -- (0,2.5) -- (1,3.5) -- (1,4);
\draw[color=black, ->] (0,0) -- (0,0.25);
\draw[color=black, ->] (1,0) -- (1,0.25);
\draw[color=black, ->] (0,1.5) -- (0,1.75);
\draw[color=black, ->] (0,2) -- (0,2.3);
\draw[color=black, ->] (1,1.5) -- (1,2);
\draw[color=black, ->] (0,3.5) -- (0,3.75);
\draw[color=black, ->] (1,3.5) -- (1,3.75);
\draw[thick] plot[mark=*] (0,2);
} \end{array}
+
\begin{array}{c}
\tikz[scale=0.5]{
\draw[color=black] (0,0) -- (0,0.5) -- (1,1.5) -- (1,2);
\draw[color=black] (1,0) -- (1,0.5) -- (0,1.5) -- (0,2);
\draw[color=black, ->] (0,0) -- (0,0.25);
\draw[color=black, ->] (1,1.5) -- (1,1.75);
\draw[color=black, ->] (1,0) -- (1,0.25);
\draw[color=black, ->] (0,1.5) -- (0,1.75);
} \end{array}
\end{gather*}
We will give a confluent presentation of the linear $(2,2)$-category $\mathcal{AOB}$ in Theorem \ref{thconf}.

\subsection{A confluent presentation of $\mathcal{AOB}$}

We will give a linear~$(3,2)$-polygraph~$\overline{\mathrm{AOB}}$ presenting the linear~$(2,2)$-category~$\mathcal{AOB}$, with more 2-cells than the linear~$(3,2)$-polygraph~$\mathrm{AOB}$. We will prove that~$\overline{\mathrm{AOB}}$ is confluent. The presence of redundant 2-cells can be seen as an addition of redundant generators to give more relations and make our presentation confluent.

\subsubsection{The linear~$(3,2)$-polygraph~$\overline{\mathrm{AOB}}$} The linear~$(3,2)$-polygraph is defined by:

\begin{itemize}
\item $\overline{\mathrm{AOB}}$ has the same 0-cells and 1-cells than the linear~$(3,2)$-polygraph~$\mathrm{AOB}$,
\item $\overline{\mathrm{AOB}}_2=\mathrm{AOB}_2 \cup \big\{~
\begin{array}{c}
\tikz[scale=0.5]{
\draw[color=black, ->] (0,0.5) -- (0,0.25);
\draw[color=black] (0,0) -- (0,0.5);
\draw[color=black, ->] (1,0) -- (1,0.25);
\draw[color=black] (1,0) -- (1,0.5);
\draw[color=black] (1,0.5) arc (0:180:0.5);
} \end{array}
~, \qquad
\begin{array}{c}
\tikz[scale=0.5]{
\draw[color=black, ->] (0,0) -- (0,0.25);
\draw[color=black] (0,0) -- (0,0.5);
\draw[color=black, ->] (1,0.5) -- (1,0.25);
\draw[color=black] (1,0) -- (1,0.5);
\draw[color=black] (1,0) arc (360:180:0.5);
} \end{array}
~, \qquad
\begin{array}{c}
\tikz[scale=0.5]{
\draw[color=black] (0,0.5) -- (0,0);
\draw[color=black, ->] (0,1.5) -- (0,0.5);
\draw[color=black, ->] (0,2) -- (0,1.5);
\draw[thick] plot[mark=*] (0,1);
} \end{array}~, \qquad
\begin{array}{c}
\tikz[scale=0.5]{
\draw[color=black] (0,0) -- (0,0.5) -- (1,1.5) -- (1,2);
\draw[color=black] (1,0) -- (1,0.5) -- (0,1.5) -- (0,2);
\draw[color=black, ->] (0,0) -- (0,0.25);
\draw[color=black, ->] (1,1.5) -- (1,1.75);
\draw[color=black, ->] (1,0.5) -- (1,0.25);
\draw[color=black, ->] (0,2) -- (0,1.75);
} \end{array}~, \qquad
\begin{array}{c}
\tikz[scale=0.5]{
\draw[color=black] (0,0) -- (0,0.5) -- (1,1.5) -- (1,2);
\draw[color=black] (1,0) -- (1,0.5) -- (0,1.5) -- (0,2);
\draw[color=black, ->] (0,0.5) -- (0,0.25);
\draw[color=black, ->] (1,2) -- (1,1.75);
\draw[color=black, ->] (1,0.5) -- (1,0.25);
\draw[color=black, ->] (0,2) -- (0,1.75);
} \end{array}~\big\}$,
\item $\overline{\mathrm{AOB}}_3$ is made of the following four families of 3-cells:

Isotopy 3-cells
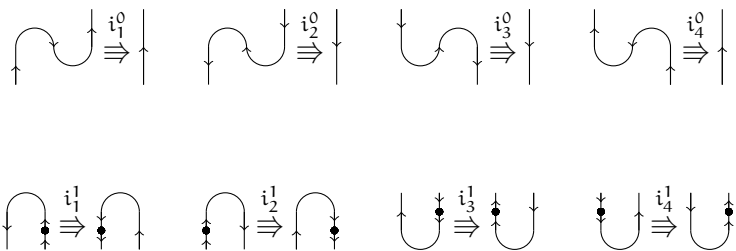
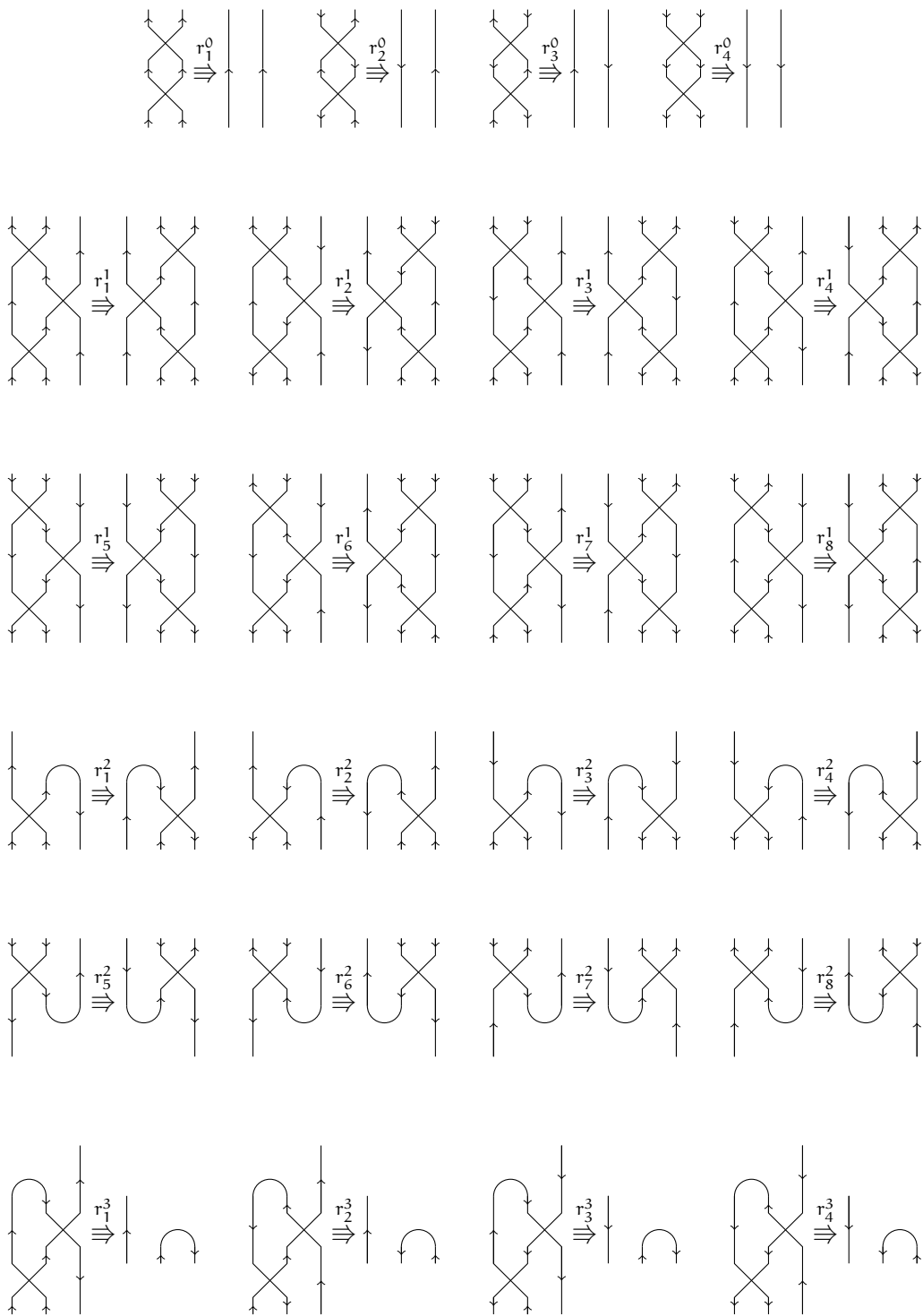
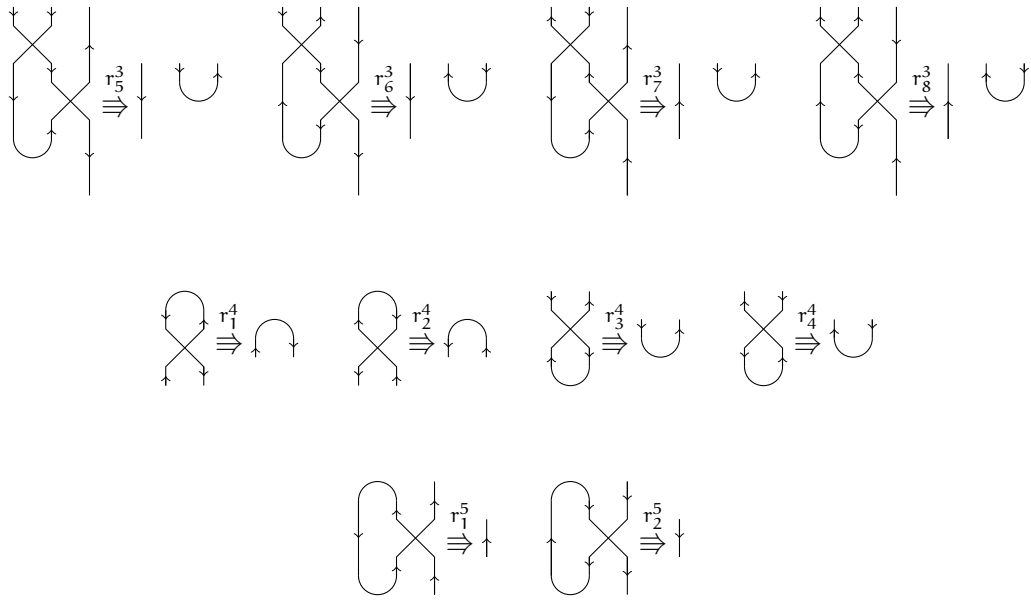
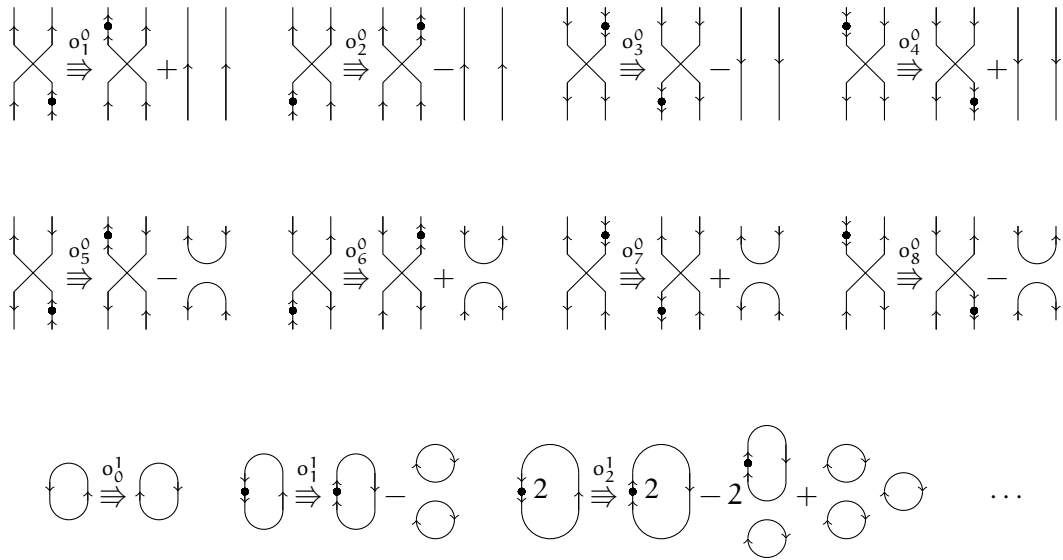
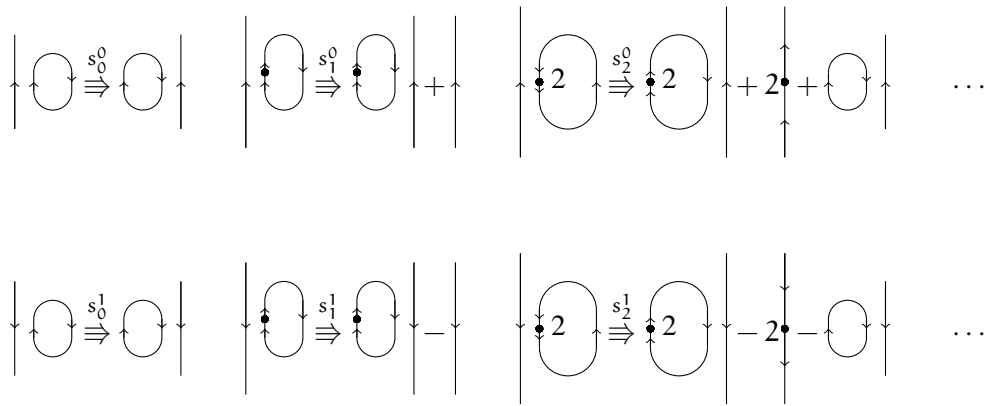
\begin{gather*}
\begin{array}{c}
\tikz[scale=0.5]{
\draw[color=black] (-1,0) -- (-1,1);
\draw[color=black] (1,1) -- (1,2);
\draw[color=black] (-1,1) arc (180:0:0.5);
\draw[color=black] (1,1) arc (360:180:0.5);
\draw[color=black, ->] (-1,0) -- (-1,0.5);
\draw[color=black, ->] (1,1) -- (1,1.5);
\draw[color=black, ->] (0,1.01) -- (0,1);
} \end{array}
\overset{i^0_1}{\Rrightarrow}
\begin{array}{c}
\tikz[scale=0.5]{
\draw[color=black] (0,0) -- (0,2);
\draw[color=black, ->] (0,0) -- (0,1);
} \end{array}
\qquad
\begin{array}{c}
\tikz[scale=0.5]{
\draw[color=black] (-1,0) -- (-1,1);
\draw[color=black] (1,1) -- (1,2);
\draw[color=black] (-1,1) arc (180:0:0.5);
\draw[color=black] (1,1) arc (360:180:0.5);
\draw[color=black, ->] (-1,1) -- (-1,0.5);
\draw[color=black, ->] (1,2) -- (1,1.5);
\draw[color=black, ->] (0,0.99) -- (0,1);
} \end{array}
\overset{i^0_2}{\Rrightarrow}
\begin{array}{c}
\tikz[scale=0.5]{
\draw[color=black] (0,0) -- (0,2);
\draw[color=black, ->] (0,2) -- (0,1);
} \end{array}
\qquad
\begin{array}{c}
\tikz[scale=0.5]{
\draw[color=black] (-1,0) -- (-1,-1);
\draw[color=black] (1,-1) -- (1,-2);
\draw[color=black] (-1,-1) arc (180:360:0.5);
\draw[color=black] (1,-1) arc (0:180:0.5);
\draw[color=black, ->] (-1,0) -- (-1,-0.5);
\draw[color=black, ->] (1,-1) -- (1,-1.5);
\draw[color=black, ->] (0,-1.01) -- (0,-1);
} \end{array}
\overset{i^0_3}{\Rrightarrow}
\begin{array}{c}
\tikz[scale=0.5]{
\draw[color=black] (0,0) -- (0,-2);
\draw[color=black, ->] (0,0) -- (0,-1);
} \end{array}
\qquad
\begin{array}{c}
\tikz[scale=0.5]{
\draw[color=black] (-1,0) -- (-1,-1);
\draw[color=black] (1,-1) -- (1,-2);
\draw[color=black] (-1,-1) arc (180:360:0.5);
\draw[color=black] (1,-1) arc (0:180:0.5);
\draw[color=black, ->] (-1,-1) -- (-1,-0.5);
\draw[color=black, ->] (1,-2) -- (1,-1.5);
\draw[color=black, ->] (0,-0.99) -- (0,-1);
} \end{array}
\overset{i^0_4}{\Rrightarrow}
\begin{array}{c}
\tikz[scale=0.5]{
\draw[color=black] (0,0) -- (0,-2);
\draw[color=black, ->] (0,-2) -- (0,-1);
} \end{array}
\end{gather*}

\begin{gather*}
\begin{array}{c}
\tikz[scale=0.5]{
\draw[color=black] (0,0) -- (0,1);
\draw[color=black] (1,0) -- (1,1);
\draw[color=black] (1,1) arc (0:180:0.5);
\draw[color=black, ->] (1,0) -- (1,0.25);
\draw[color=black, ->] (1,0.5) -- (1,0.8);
\draw[color=black, ->] (0,1) -- (0,0.5);
\draw[thick] plot[mark=*] (1,0.5);
} \end{array}
\overset{i^1_1}{\Rrightarrow}
\begin{array}{c}
\tikz[scale=0.5]{
\draw[color=black] (0,0) -- (0,1);
\draw[color=black] (-1,0) -- (-1,1);
\draw[color=black] (-1,1) arc (180:0:0.5);
\draw[color=black, ->] (-1,0.5) -- (-1,0.2);
\draw[color=black, ->] (-1,1) -- (-1,0.75);
\draw[color=black, ->] (0,0) -- (0,0.5);
\draw[thick] plot[mark=*] (-1,0.5);
} \end{array}
\qquad
\begin{array}{c}
\tikz[scale=0.5]{
\draw[color=black] (0,0) -- (0,1);
\draw[color=black] (-1,0) -- (-1,1);
\draw[color=black] (-1,1) arc (180:0:0.5);
\draw[color=black, ->] (-1,0) -- (-1,0.25);
\draw[color=black, ->] (-1,0.5) -- (-1,0.8);
\draw[color=black, ->] (0,1) -- (0,0.5);
\draw[thick] plot[mark=*] (-1,0.5);
} \end{array}
\overset{i^1_2}{\Rrightarrow}
\begin{array}{c}
\tikz[scale=0.5]{
\draw[color=black] (0,0) -- (0,1);
\draw[color=black] (1,0) -- (1,1);
\draw[color=black] (1,1) arc (0:180:0.5);
\draw[color=black, ->] (1,0.5) -- (1,0.2);
\draw[color=black, ->] (1,1) -- (1,0.75);
\draw[color=black, ->] (0,0) -- (0,0.5);
\draw[thick] plot[mark=*] (1,0.5);
} \end{array}
\qquad
\begin{array}{c}
\tikz[scale=0.5]{
\draw[color=black] (0,0) -- (0,-1);
\draw[color=black] (1,0) -- (1,-1);
\draw[color=black] (1,-1) arc (360:180:0.5);
\draw[color=black, ->] (1,0) -- (1,-0.25);
\draw[color=black, ->] (1,-0.5) -- (1,-0.8);
\draw[color=black, ->] (0,-1) -- (0,-0.5);
\draw[thick] plot[mark=*] (1,-0.5);
} \end{array}
\overset{i^1_3}{\Rrightarrow}
\begin{array}{c}
\tikz[scale=0.5]{
\draw[color=black] (0,0) -- (0,-1);
\draw[color=black] (-1,0) -- (-1,-1);
\draw[color=black] (-1,-1) arc (180:360:0.5);
\draw[color=black, ->] (-1,-0.5) -- (-1,-0.2);
\draw[color=black, ->] (-1,-1) -- (-1,-0.75);
\draw[color=black, ->] (0,0) -- (0,-0.5);
\draw[thick] plot[mark=*] (-1,-0.5);
} \end{array}
\qquad
\begin{array}{c}
\tikz[scale=0.5]{
\draw[color=black] (0,0) -- (0,-1);
\draw[color=black] (-1,0) -- (-1,-1);
\draw[color=black] (-1,-1) arc (180:360:0.5);
\draw[color=black, ->] (-1,0) -- (-1,-0.25);
\draw[color=black, ->] (-1,-0.5) -- (-1,-0.8);
\draw[color=black, ->] (0,-1) -- (0,-0.5);
\draw[thick] plot[mark=*] (-1,-0.5);
} \end{array}
\overset{i^1_4}{\Rrightarrow}
\begin{array}{c}
\tikz[scale=0.5]{
\draw[color=black] (0,0) -- (0,-1);
\draw[color=black] (1,0) -- (1,-1);
\draw[color=black] (1,-1) arc (360:180:0.5);
\draw[color=black, ->] (1,-0.5) -- (1,-0.2);
\draw[color=black, ->] (1,-1) -- (1,-0.75);
\draw[color=black, ->] (0,0) -- (0,-0.5);
\draw[thick] plot[mark=*] (1,-0.5);
} \end{array}
\end{gather*}
Reidemeister 3-cells
\begin{gather*}
\begin{array}{c}
\tikz[scale=0.5]{
\draw[color=black, ->] (0,0) -- (0,0.25);
\draw[color=black, ->] (1,0) -- (1,0.25);
\draw[color=black, ->] (0,0.25) -- (0,0.5) -- (1,1.5) -- (1,1.75);
\draw[color=black, ->] (1,0.25) -- (1,0.5) -- (0,1.5) -- (0,1.75);
\draw[color=black, ->] (0,1.75) -- (0,2) -- (1,3) -- (1,3.25);
\draw[color=black, ->] (1,1.75) -- (1,2) -- (0,3) -- (0,3.25);
\draw[color=black] (1,3.5) -- (1,3.25);
\draw[color=black] (0,3.5) -- (0,3.25);
} \end{array}
\overset{r^0_1}{\Rrightarrow}
\begin{array}{c}
\tikz[scale=0.5]{
\draw[color=black, ->] (0,0) -- (0,1.75);
\draw[color=black] (0,1.75) -- (0,3.5);
\draw[color=black, ->] (1,0) -- (1,1.75);
\draw[color=black] (1,1.75) -- (1,3.5);
} \end{array}
\qquad
\begin{array}{c}
\tikz[scale=0.5]{
\draw[color=black, ->] (0,3.5) -- (0,3.25);
\draw[color=black, ->] (1,0) -- (1,0.25);
\draw[color=black, ->] (0,3.25) -- (0,3) -- (1,2) -- (1,1.75);
\draw[color=black, ->] (1,0.25) -- (1,0.5) -- (0,1.5) -- (0,1.75);
\draw[color=black, ->] (0,1.75) -- (0,2) -- (1,3) -- (1,3.25);
\draw[color=black, ->] (1,1.75) -- (1,1.5) -- (0,0.5) -- (0,0.25);
\draw[color=black] (1,3.5) -- (1,3.25);
\draw[color=black] (0,0) -- (0,0.25);
} \end{array}
\overset{r^0_2}{\Rrightarrow}
\begin{array}{c}
\tikz[scale=0.5]{
\draw[color=black] (0,1.75) -- (0,0);
\draw[color=black, ->] (0,3.5) -- (0,1.75);
\draw[color=black, ->] (1,0) -- (1,1.75);
\draw[color=black] (1,1.75) -- (1,3.5);
} \end{array}
\qquad
\begin{array}{c}
\tikz[scale=0.5]{
\draw[color=black, ->] (0,0) -- (0,0.25);
\draw[color=black, ->] (1,3.5) -- (1,3.25);
\draw[color=black, ->] (0,0.25) -- (0,0.5) -- (1,1.5) -- (1,1.75);
\draw[color=black, ->] (1,3.25) -- (1,3) -- (0,2) -- (0,1.75);
\draw[color=black, ->] (0,1.75) -- (0,1.5) -- (1,0.5) -- (1,0.25);
\draw[color=black, ->] (1,1.75) -- (1,2) -- (0,3) -- (0,3.25);
\draw[color=black] (1,0) -- (1,0.25);
\draw[color=black] (0,3.5) -- (0,3.25);
} \end{array}
\overset{r^0_3}{\Rrightarrow}
\begin{array}{c}
\tikz[scale=0.5]{
\draw[color=black] (1,1.75) -- (1,0);
\draw[color=black, ->] (1,3.5) -- (1,1.75);
\draw[color=black, ->] (0,0) -- (0,1.75);
\draw[color=black] (0,1.75) -- (0,3.5);
} \end{array}
\qquad
\begin{array}{c}
\tikz[scale=0.5]{
\draw[color=black, ->] (0,3.5) -- (0,3.25);
\draw[color=black, ->] (1,3.5) -- (1,3.25);
\draw[color=black, ->] (0,3.25) -- (0,3) -- (1,2) -- (1,1.75);
\draw[color=black, ->] (1,3.25) -- (1,3) -- (0,2) -- (0,1.75);
\draw[color=black, ->] (0,1.75) -- (0,1.5) -- (1,0.5) -- (1,0.25);
\draw[color=black, ->] (1,1.75) -- (1,1.5) -- (0,0.5) -- (0,0.25);
\draw[color=black] (1,0) -- (1,0.25);
\draw[color=black] (0,0) -- (0,0.25);
} \end{array}
\overset{r^0_4}{\Rrightarrow}
\begin{array}{c}
\tikz[scale=0.5]{
\draw[color=black] (1,1.75) -- (1,0);
\draw[color=black, ->] (1,3.5) -- (1,1.75);
\draw[color=black] (0,0) -- (0,1.75);
\draw[color=black, ->] (0,3.5) -- (0,1.75);
} \end{array}
\end{gather*}

\begin{gather*}
\begin{array}{c}
\tikz[scale=0.5]{
\draw[color=black] (-1,0) -- (-1,0.5) -- (0,1.5) -- (0,2) -- (1,3) -- (1,5);
\draw[color=black, ->] (-1,0) -- (-1,0.25);
\draw[color=black, ->] (0,1.5) -- (0,1.75);
\draw[color=black, ->] (1,3) -- (1,4);
\draw[color=black] (0,0) -- (0,0.5) -- (-1,1.5) -- (-1,3.5) -- (0,4.5) -- (0,5);
\draw[color=black, ->] (0,0) -- (0,0.25);
\draw[color=black, ->] (-1,1.5) -- (-1,2.5);
\draw[color=black, ->] (0,4.5) -- (0,4.75);
\draw[color=black] (1,0) -- (1,2) -- (0,3) -- (0,3.5) -- (-1,4.5) -- (-1,5);
\draw[color=black, ->] (1,0) -- (1,1);
\draw[color=black, ->] (0,3) -- (0,3.25);
\draw[color=black, ->] (-1,4.5) -- (-1,4.75);
} \end{array}
\overset{r^1_1}{\Rrightarrow}
\begin{array}{c}
\tikz[scale=0.5]{
\draw[color=black] (1,0) -- (1,0.5) -- (0,1.5) -- (0,2) -- (-1,3) -- (-1,5);
\draw[color=black, ->] (1,0) -- (1,0.25);
\draw[color=black, ->] (0,1.5) -- (0,1.75);
\draw[color=black, ->] (-1,3) -- (-1,4);
\draw[color=black] (0,0) -- (0,0.5) -- (1,1.5) -- (1,3.5) -- (0,4.5) -- (0,5);
\draw[color=black, ->] (0,0) -- (0,0.25);
\draw[color=black, ->] (1,1.5) -- (1,2.5);
\draw[color=black, ->] (0,4.5) -- (0,4.75);
\draw[color=black] (-1,0) -- (-1,2) -- (0,3) -- (0,3.5) -- (1,4.5) -- (1,5);
\draw[color=black, ->] (-1,0) -- (-1,1);
\draw[color=black, ->] (0,3) -- (0,3.25);
\draw[color=black, ->] (1,4.5) -- (1,4.75);
} \end{array}
\qquad
\begin{array}{c}
\tikz[scale=0.5]{
\draw[color=black] (-1,0) -- (-1,0.5) -- (0,1.5) -- (0,2) -- (1,3) -- (1,5);
\draw[color=black, ->] (-1,0.5) -- (-1,0.25);
\draw[color=black, ->] (0,2) -- (0,1.75);
\draw[color=black, ->] (1,5) -- (1,4);
\draw[color=black] (0,0) -- (0,0.5) -- (-1,1.5) -- (-1,3.5) -- (0,4.5) -- (0,5);
\draw[color=black, ->] (0,0) -- (0,0.25);
\draw[color=black, ->] (-1,1.5) -- (-1,2.5);
\draw[color=black, ->] (0,4.5) -- (0,4.75);
\draw[color=black] (1,0) -- (1,2) -- (0,3) -- (0,3.5) -- (-1,4.5) -- (-1,5);
\draw[color=black, ->] (1,0) -- (1,1);
\draw[color=black, ->] (0,3) -- (0,3.25);
\draw[color=black, ->] (-1,4.5) -- (-1,4.75);
} \end{array}
\overset{r^1_2}{\Rrightarrow}
\begin{array}{c}
\tikz[scale=0.5]{
\draw[color=black] (1,0) -- (1,0.5) -- (0,1.5) -- (0,2) -- (-1,3) -- (-1,5);
\draw[color=black, ->] (1,0) -- (1,0.25);
\draw[color=black, ->] (0,1.5) -- (0,1.75);
\draw[color=black, ->] (-1,3) -- (-1,4);
\draw[color=black] (0,0) -- (0,0.5) -- (1,1.5) -- (1,3.5) -- (0,4.5) -- (0,5);
\draw[color=black, ->] (0,0) -- (0,0.25);
\draw[color=black, ->] (1,1.5) -- (1,2.5);
\draw[color=black, ->] (0,4.5) -- (0,4.75);
\draw[color=black] (-1,0) -- (-1,2) -- (0,3) -- (0,3.5) -- (1,4.5) -- (1,5);
\draw[color=black, ->] (-1,2) -- (-1,1);
\draw[color=black, ->] (0,3.5) -- (0,3.25);
\draw[color=black, ->] (1,5) -- (1,4.75);
} \end{array}
\qquad
\begin{array}{c}
\tikz[scale=0.5]{
\draw[color=black] (-1,0) -- (-1,0.5) -- (0,1.5) -- (0,2) -- (1,3) -- (1,5);
\draw[color=black, ->] (-1,0) -- (-1,0.25);
\draw[color=black, ->] (0,1.5) -- (0,1.75);
\draw[color=black, ->] (1,3) -- (1,4);
\draw[color=black] (0,0) -- (0,0.5) -- (-1,1.5) -- (-1,3.5) -- (0,4.5) -- (0,5);
\draw[color=black, ->] (0,0.5) -- (0,0.25);
\draw[color=black, ->] (-1,3.5) -- (-1,2.5);
\draw[color=black, ->] (0,5) -- (0,4.75);
\draw[color=black] (1,0) -- (1,2) -- (0,3) -- (0,3.5) -- (-1,4.5) -- (-1,5);
\draw[color=black, ->] (1,0) -- (1,1);
\draw[color=black, ->] (0,3) -- (0,3.25);
\draw[color=black, ->] (-1,4.5) -- (-1,4.75);
} \end{array}
\overset{r^1_3}{\Rrightarrow}
\begin{array}{c}
\tikz[scale=0.5]{
\draw[color=black] (1,0) -- (1,0.5) -- (0,1.5) -- (0,2) -- (-1,3) -- (-1,5);
\draw[color=black, ->] (1,0) -- (1,0.25);
\draw[color=black, ->] (0,1.5) -- (0,1.75);
\draw[color=black, ->] (-1,3) -- (-1,4);
\draw[color=black] (0,0) -- (0,0.5) -- (1,1.5) -- (1,3.5) -- (0,4.5) -- (0,5);
\draw[color=black, ->] (0,0.5) -- (0,0.25);
\draw[color=black, ->] (1,3.5) -- (1,2.5);
\draw[color=black, ->] (0,5) -- (0,4.75);
\draw[color=black] (-1,0) -- (-1,2) -- (0,3) -- (0,3.5) -- (1,4.5) -- (1,5);
\draw[color=black, ->] (-1,0) -- (-1,1);
\draw[color=black, ->] (0,3) -- (0,3.25);
\draw[color=black, ->] (1,4.5) -- (1,4.75);
} \end{array}
\qquad
\begin{array}{c}
\tikz[scale=0.5]{
\draw[color=black] (-1,0) -- (-1,0.5) -- (0,1.5) -- (0,2) -- (1,3) -- (1,5);
\draw[color=black, ->] (-1,0) -- (-1,0.25);
\draw[color=black, ->] (0,1.5) -- (0,1.75);
\draw[color=black, ->] (1,3) -- (1,4);
\draw[color=black] (0,0) -- (0,0.5) -- (-1,1.5) -- (-1,3.5) -- (0,4.5) -- (0,5);
\draw[color=black, ->] (0,0) -- (0,0.25);
\draw[color=black, ->] (-1,1.5) -- (-1,2.5);
\draw[color=black, ->] (0,4.5) -- (0,4.75);
\draw[color=black] (1,0) -- (1,2) -- (0,3) -- (0,3.5) -- (-1,4.5) -- (-1,5);
\draw[color=black, ->] (1,2) -- (1,1);
\draw[color=black, ->] (0,3.5) -- (0,3.25);
\draw[color=black, ->] (-1,5) -- (-1,4.75);
} \end{array}
\overset{r^1_4}{\Rrightarrow}
\begin{array}{c}
\tikz[scale=0.5]{
\draw[color=black] (1,0) -- (1,0.5) -- (0,1.5) -- (0,2) -- (-1,3) -- (-1,5);
\draw[color=black, ->] (1,0.5) -- (1,0.25);
\draw[color=black, ->] (0,2) -- (0,1.75);
\draw[color=black, ->] (-1,5) -- (-1,4);
\draw[color=black] (0,0) -- (0,0.5) -- (1,1.5) -- (1,3.5) -- (0,4.5) -- (0,5);
\draw[color=black, ->] (0,0) -- (0,0.25);
\draw[color=black, ->] (1,1.5) -- (1,2.5);
\draw[color=black, ->] (0,4.5) -- (0,4.75);
\draw[color=black] (-1,0) -- (-1,2) -- (0,3) -- (0,3.5) -- (1,4.5) -- (1,5);
\draw[color=black, ->] (-1,0) -- (-1,1);
\draw[color=black, ->] (0,3) -- (0,3.25);
\draw[color=black, ->] (1,4.5) -- (1,4.75);
} \end{array}
\end{gather*}

\begin{gather*}
\begin{array}{c}
\tikz[scale=0.5]{
\draw[color=black] (-1,0) -- (-1,0.5) -- (0,1.5) -- (0,2) -- (1,3) -- (1,5);
\draw[color=black, ->] (-1,0.5) -- (-1,0.25);
\draw[color=black, ->] (0,2) -- (0,1.75);
\draw[color=black, ->] (1,5) -- (1,4);
\draw[color=black] (0,0) -- (0,0.5) -- (-1,1.5) -- (-1,3.5) -- (0,4.5) -- (0,5);
\draw[color=black, ->] (0,0.5) -- (0,0.25);
\draw[color=black, ->] (-1,3.5) -- (-1,2.5);
\draw[color=black, ->] (0,5) -- (0,4.75);
\draw[color=black] (1,0) -- (1,2) -- (0,3) -- (0,3.5) -- (-1,4.5) -- (-1,5);
\draw[color=black, ->] (1,2) -- (1,1);
\draw[color=black, ->] (0,3.5) -- (0,3.25);
\draw[color=black, ->] (-1,5) -- (-1,4.75);
} \end{array}
\overset{r^1_5}{\Rrightarrow}
\begin{array}{c}
\tikz[scale=0.5]{
\draw[color=black] (1,0) -- (1,0.5) -- (0,1.5) -- (0,2) -- (-1,3) -- (-1,5);
\draw[color=black, ->] (1,0.5) -- (1,0.25);
\draw[color=black, ->] (0,2) -- (0,1.75);
\draw[color=black, ->] (-1,5) -- (-1,4);
\draw[color=black] (0,0) -- (0,0.5) -- (1,1.5) -- (1,3.5) -- (0,4.5) -- (0,5);
\draw[color=black, ->] (0,0.5) -- (0,0.25);
\draw[color=black, ->] (1,3.5) -- (1,2.5);
\draw[color=black, ->] (0,5) -- (0,4.75);
\draw[color=black] (-1,0) -- (-1,2) -- (0,3) -- (0,3.5) -- (1,4.5) -- (1,5);
\draw[color=black, ->] (-1,2) -- (-1,1);
\draw[color=black, ->] (0,3.5) -- (0,3.25);
\draw[color=black, ->] (1,5) -- (1,4.75);
} \end{array}
\qquad
\begin{array}{c}
\tikz[scale=0.5]{
\draw[color=black] (-1,0) -- (-1,0.5) -- (0,1.5) -- (0,2) -- (1,3) -- (1,5);
\draw[color=black, ->] (-1,0.5) -- (-1,0.25);
\draw[color=black, ->] (0,2) -- (0,1.75);
\draw[color=black, ->] (1,5) -- (1,4);
\draw[color=black] (0,0) -- (0,0.5) -- (-1,1.5) -- (-1,3.5) -- (0,4.5) -- (0,5);
\draw[color=black, ->] (0,0.5) -- (0,0.25);
\draw[color=black, ->] (-1,3.5) -- (-1,2.5);
\draw[color=black, ->] (0,5) -- (0,4.75);
\draw[color=black] (1,0) -- (1,2) -- (0,3) -- (0,3.5) -- (-1,4.5) -- (-1,5);
\draw[color=black, ->] (1,0) -- (1,1);
\draw[color=black, ->] (0,3) -- (0,3.25);
\draw[color=black, ->] (-1,4.5) -- (-1,4.75);
} \end{array}
\overset{r^1_6}{\Rrightarrow}
\begin{array}{c}
\tikz[scale=0.5]{
\draw[color=black] (1,0) -- (1,0.5) -- (0,1.5) -- (0,2) -- (-1,3) -- (-1,5);
\draw[color=black, ->] (1,0) -- (1,0.25);
\draw[color=black, ->] (0,1.5) -- (0,1.75);
\draw[color=black, ->] (-1,3) -- (-1,4);
\draw[color=black] (0,0) -- (0,0.5) -- (1,1.5) -- (1,3.5) -- (0,4.5) -- (0,5);
\draw[color=black, ->] (0,0.5) -- (0,0.25);
\draw[color=black, ->] (1,3.5) -- (1,2.5);
\draw[color=black, ->] (0,5) -- (0,4.75);
\draw[color=black] (-1,0) -- (-1,2) -- (0,3) -- (0,3.5) -- (1,4.5) -- (1,5);
\draw[color=black, ->] (-1,2) -- (-1,1);
\draw[color=black, ->] (0,3.5) -- (0,3.25);
\draw[color=black, ->] (1,5) -- (1,4.75);
} \end{array}
\qquad
\begin{array}{c}
\tikz[scale=0.5]{
\draw[color=black] (-1,0) -- (-1,0.5) -- (0,1.5) -- (0,2) -- (1,3) -- (1,5);
\draw[color=black, ->] (-1,0) -- (-1,0.25);
\draw[color=black, ->] (0,1.5) -- (0,1.75);
\draw[color=black, ->] (1,3) -- (1,4);
\draw[color=black] (0,0) -- (0,0.5) -- (-1,1.5) -- (-1,3.5) -- (0,4.5) -- (0,5);
\draw[color=black, ->] (0,0.5) -- (0,0.25);
\draw[color=black, ->] (-1,3.5) -- (-1,2.5);
\draw[color=black, ->] (0,5) -- (0,4.75);
\draw[color=black] (1,0) -- (1,2) -- (0,3) -- (0,3.5) -- (-1,4.5) -- (-1,5);
\draw[color=black, ->] (1,2) -- (1,1);
\draw[color=black, ->] (0,3.5) -- (0,3.25);
\draw[color=black, ->] (-1,5) -- (-1,4.75);
} \end{array}
\overset{r^1_7}{\Rrightarrow}
\begin{array}{c}
\tikz[scale=0.5]{
\draw[color=black] (1,0) -- (1,0.5) -- (0,1.5) -- (0,2) -- (-1,3) -- (-1,5);
\draw[color=black, ->] (1,0.5) -- (1,0.25);
\draw[color=black, ->] (0,2) -- (0,1.75);
\draw[color=black, ->] (-1,5) -- (-1,4);
\draw[color=black] (0,0) -- (0,0.5) -- (1,1.5) -- (1,3.5) -- (0,4.5) -- (0,5);
\draw[color=black, ->] (0,0.5) -- (0,0.25);
\draw[color=black, ->] (1,3.5) -- (1,2.5);
\draw[color=black, ->] (0,5) -- (0,4.75);
\draw[color=black] (-1,0) -- (-1,2) -- (0,3) -- (0,3.5) -- (1,4.5) -- (1,5);
\draw[color=black, ->] (-1,0) -- (-1,1);
\draw[color=black, ->] (0,3) -- (0,3.25);
\draw[color=black, ->] (1,4.5) -- (1,4.75);
} \end{array}
\qquad
\begin{array}{c}
\tikz[scale=0.5]{
\draw[color=black] (-1,0) -- (-1,0.5) -- (0,1.5) -- (0,2) -- (1,3) -- (1,5);
\draw[color=black, ->] (-1,0.5) -- (-1,0.25);
\draw[color=black, ->] (0,2) -- (0,1.75);
\draw[color=black, ->] (1,5) -- (1,4);
\draw[color=black] (0,0) -- (0,0.5) -- (-1,1.5) -- (-1,3.5) -- (0,4.5) -- (0,5);
\draw[color=black, ->] (0,0) -- (0,0.25);
\draw[color=black, ->] (-1,1.5) -- (-1,2.5);
\draw[color=black, ->] (0,4.5) -- (0,4.75);
\draw[color=black] (1,0) -- (1,2) -- (0,3) -- (0,3.5) -- (-1,4.5) -- (-1,5);
\draw[color=black, ->] (1,2) -- (1,1);
\draw[color=black, ->] (0,3.5) -- (0,3.25);
\draw[color=black, ->] (-1,5) -- (-1,4.75);
} \end{array}
\overset{r^1_8}{\Rrightarrow}
\begin{array}{c}
\tikz[scale=0.5]{
\draw[color=black] (1,0) -- (1,0.5) -- (0,1.5) -- (0,2) -- (-1,3) -- (-1,5);
\draw[color=black, ->] (1,0.5) -- (1,0.25);
\draw[color=black, ->] (0,2) -- (0,1.75);
\draw[color=black, ->] (-1,5) -- (-1,4);
\draw[color=black] (0,0) -- (0,0.5) -- (1,1.5) -- (1,3.5) -- (0,4.5) -- (0,5);
\draw[color=black, ->] (0,0) -- (0,0.25);
\draw[color=black, ->] (1,1.5) -- (1,2.5);
\draw[color=black, ->] (0,4.5) -- (0,4.75);
\draw[color=black] (-1,0) -- (-1,2) -- (0,3) -- (0,3.5) -- (1,4.5) -- (1,5);
\draw[color=black, ->] (-1,2) -- (-1,1);
\draw[color=black, ->] (0,3.5) -- (0,3.25);
\draw[color=black, ->] (1,5) -- (1,4.75);
} \end{array}
\end{gather*}

\begin{gather*}
\begin{array}{c}
\tikz[scale=0.5]{
\draw[color=black] (-1,0) -- (-1,0.5) -- (0,1.5) -- (0,2);
\draw[color=black, ->] (-1,0) -- (-1,0.25);
\draw[color=black, ->] (0,1.5) -- (0,1.75);
\draw[color=black] (0,0) -- (0,0.5) -- (-1,1.5) -- (-1,3.5);
\draw[color=black, ->] (0,0) -- (0,0.25);
\draw[color=black, ->] (-1,1.5) -- (-1,2.5);
\draw[color=black] (1,0) -- (1,2);
\draw[color=black] (1,2) arc (0:180:0.5);
\draw[color=black, ->] (1,2) -- (1,1);
} \end{array}
\overset{r^2_1}{\Rrightarrow}
\begin{array}{c}
\tikz[scale=0.5]{
\draw[color=black] (1,0) -- (1,0.5) -- (0,1.5) -- (0,2);
\draw[color=black, ->] (1,0.5) -- (1,0.25);
\draw[color=black, ->] (0,2) -- (0,1.75);
\draw[color=black] (0,0) -- (0,0.5) -- (1,1.5) -- (1,3.5);
\draw[color=black, ->] (0,0) -- (0,0.25);
\draw[color=black, ->] (1,1.5) -- (1,2.5);
\draw[color=black] (-1,0) -- (-1,2);
\draw[color=black] (-1,2) arc (180:0:0.5);
\draw[color=black, ->] (-1,0) -- (-1,1);
} \end{array}
\qquad
\begin{array}{c}
\tikz[scale=0.5]{
\draw[color=black] (-1,0) -- (-1,0.5) -- (0,1.5) -- (0,2);
\draw[color=black, ->] (-1,0.5) -- (-1,0.25);
\draw[color=black, ->] (0,2) -- (0,1.75);
\draw[color=black] (0,0) -- (0,0.5) -- (-1,1.5) -- (-1,3.5);
\draw[color=black, ->] (0,0) -- (0,0.25);
\draw[color=black, ->] (-1,1.5) -- (-1,2.5);
\draw[color=black] (1,0) -- (1,2);
\draw[color=black] (1,2) arc (0:180:0.5);
\draw[color=black, ->] (1,0) -- (1,1);
} \end{array}
\overset{r^2_2}{\Rrightarrow}
\begin{array}{c}
\tikz[scale=0.5]{
\draw[color=black] (1,0) -- (1,0.5) -- (0,1.5) -- (0,2);
\draw[color=black, ->] (1,0) -- (1,0.25);
\draw[color=black, ->] (0,1.5) -- (0,1.75);
\draw[color=black] (0,0) -- (0,0.5) -- (1,1.5) -- (1,3.5);
\draw[color=black, ->] (0,0) -- (0,0.25);
\draw[color=black, ->] (1,1.5) -- (1,2.5);
\draw[color=black] (-1,0) -- (-1,2);
\draw[color=black] (-1,2) arc (180:0:0.5);
\draw[color=black, ->] (-1,2) -- (-1,1);
} \end{array}
\qquad
\begin{array}{c}
\tikz[scale=0.5]{
\draw[color=black] (-1,0) -- (-1,0.5) -- (0,1.5) -- (0,2);
\draw[color=black, ->] (-1,0) -- (-1,0.25);
\draw[color=black, ->] (0,1.5) -- (0,1.75);
\draw[color=black] (0,0) -- (0,0.5) -- (-1,1.5) -- (-1,3.5);
\draw[color=black, ->] (0,0.5) -- (0,0.25);
\draw[color=black, ->] (-1,3.5) -- (-1,2.5);
\draw[color=black] (1,0) -- (1,2);
\draw[color=black] (1,2) arc (0:180:0.5);
\draw[color=black, ->] (1,2) -- (1,1);
} \end{array}
\overset{r^2_3}{\Rrightarrow}
\begin{array}{c}
\tikz[scale=0.5]{
\draw[color=black] (1,0) -- (1,0.5) -- (0,1.5) -- (0,2);
\draw[color=black, ->] (1,0.5) -- (1,0.25);
\draw[color=black, ->] (0,2) -- (0,1.75);
\draw[color=black] (0,0) -- (0,0.5) -- (1,1.5) -- (1,3.5);
\draw[color=black, ->] (0,0.5) -- (0,0.25);
\draw[color=black, ->] (1,3.5) -- (1,2.5);
\draw[color=black] (-1,0) -- (-1,2);
\draw[color=black] (-1,2) arc (180:0:0.5);
\draw[color=black, ->] (-1,0) -- (-1,1);
} \end{array}
\qquad
\begin{array}{c}
\tikz[scale=0.5]{
\draw[color=black] (-1,0) -- (-1,0.5) -- (0,1.5) -- (0,2);
\draw[color=black, ->] (-1,0.5) -- (-1,0.25);
\draw[color=black, ->] (0,2) -- (0,1.75);
\draw[color=black] (0,0) -- (0,0.5) -- (-1,1.5) -- (-1,3.5);
\draw[color=black, ->] (0,0.5) -- (0,0.25);
\draw[color=black, ->] (-1,3.5) -- (-1,2.5);
\draw[color=black] (1,0) -- (1,2);
\draw[color=black] (1,2) arc (0:180:0.5);
\draw[color=black, ->] (1,0) -- (1,1);
} \end{array}
\overset{r^2_4}{\Rrightarrow}
\begin{array}{c}
\tikz[scale=0.5]{
\draw[color=black] (1,0) -- (1,0.5) -- (0,1.5) -- (0,2);
\draw[color=black, ->] (1,0) -- (1,0.25);
\draw[color=black, ->] (0,1.5) -- (0,1.75);
\draw[color=black] (0,0) -- (0,0.5) -- (1,1.5) -- (1,3.5);
\draw[color=black, ->] (0,0.5) -- (0,0.25);
\draw[color=black, ->] (1,3.5) -- (1,2.5);
\draw[color=black] (-1,0) -- (-1,2);
\draw[color=black] (-1,2) arc (180:0:0.5);
\draw[color=black, ->] (-1,2) -- (-1,1);
} \end{array}
\end{gather*}

\begin{gather*}
\begin{array}{c}
\tikz[scale=0.5]{
\draw[color=black] (-1,0) -- (-1,-0.5) -- (0,-1.5) -- (0,-2);
\draw[color=black, ->] (-1,0) -- (-1,-0.25);
\draw[color=black, ->] (0,-1.5) -- (0,-1.75);
\draw[color=black] (0,0) -- (0,-0.5) -- (-1,-1.5) -- (-1,-3.5);
\draw[color=black, ->] (0,0) -- (0,-0.25);
\draw[color=black, ->] (-1,-1.5) -- (-1,-2.5);
\draw[color=black] (1,0) -- (1,-2);
\draw[color=black] (1,-2) arc (360:180:0.5);
\draw[color=black, ->] (1,-2) -- (1,-1);
} \end{array}
\overset{r^2_5}{\Rrightarrow}
\begin{array}{c}
\tikz[scale=0.5]{
\draw[color=black] (1,0) -- (1,-0.5) -- (0,-1.5) -- (0,-2);
\draw[color=black, ->] (1,-0.5) -- (1,-0.25);
\draw[color=black, ->] (0,-2) -- (0,-1.75);
\draw[color=black] (0,0) -- (0,-0.5) -- (1,-1.5) -- (1,-3.5);
\draw[color=black, ->] (0,0) -- (0,-0.25);
\draw[color=black, ->] (1,-1.5) -- (1,-2.5);
\draw[color=black] (-1,0) -- (-1,-2);
\draw[color=black] (-1,-2) arc (180:360:0.5);
\draw[color=black, ->] (-1,0) -- (-1,-1);
} \end{array}
\qquad
\begin{array}{c}
\tikz[scale=0.5]{
\draw[color=black] (-1,0) -- (-1,-0.5) -- (0,-1.5) -- (0,-2);
\draw[color=black, ->] (-1,-0.5) -- (-1,-0.25);
\draw[color=black, ->] (0,-2) -- (0,-1.75);
\draw[color=black] (0,0) -- (0,-0.5) -- (-1,-1.5) -- (-1,-3.5);
\draw[color=black, ->] (0,0) -- (0,-0.25);
\draw[color=black, ->] (-1,-1.5) -- (-1,-2.5);
\draw[color=black] (1,0) -- (1,-2);
\draw[color=black] (1,-2) arc (360:180:0.5);
\draw[color=black, ->] (1,0) -- (1,-1);
} \end{array}
\overset{r^2_6}{\Rrightarrow}
\begin{array}{c}
\tikz[scale=0.5]{
\draw[color=black] (1,0) -- (1,-0.5) -- (0,-1.5) -- (0,-2);
\draw[color=black, ->] (1,0) -- (1,-0.25);
\draw[color=black, ->] (0,-1.5) -- (0,-1.75);
\draw[color=black] (0,0) -- (0,-0.5) -- (1,-1.5) -- (1,-3.5);
\draw[color=black, ->] (0,0) -- (0,-0.25);
\draw[color=black, ->] (1,-1.5) -- (1,-2.5);
\draw[color=black] (-1,0) -- (-1,-2);
\draw[color=black] (-1,-2) arc (180:360:0.5);
\draw[color=black, ->] (-1,-2) -- (-1,-1);
} \end{array}
\qquad
\begin{array}{c}
\tikz[scale=0.5]{
\draw[color=black] (-1,0) -- (-1,-0.5) -- (0,-1.5) -- (0,-2);
\draw[color=black, ->] (-1,0) -- (-1,-0.25);
\draw[color=black, ->] (0,-1.5) -- (0,-1.75);
\draw[color=black] (0,0) -- (0,-0.5) -- (-1,-1.5) -- (-1,-3.5);
\draw[color=black, ->] (0,-0.5) -- (0,-0.25);
\draw[color=black, ->] (-1,-3.5) -- (-1,-2.5);
\draw[color=black] (1,0) -- (1,-2);
\draw[color=black] (1,-2) arc (360:180:0.5);
\draw[color=black, ->] (1,-2) -- (1,-1);
} \end{array}
\overset{r^2_7}{\Rrightarrow}
\begin{array}{c}
\tikz[scale=0.5]{
\draw[color=black] (1,0) -- (1,-0.5) -- (0,-1.5) -- (0,-2);
\draw[color=black, ->] (1,-0.5) -- (1,-0.25);
\draw[color=black, ->] (0,-2) -- (0,-1.75);
\draw[color=black] (0,0) -- (0,-0.5) -- (1,-1.5) -- (1,-3.5);
\draw[color=black, ->] (0,-0.5) -- (0,-0.25);
\draw[color=black, ->] (1,-3.5) -- (1,-2.5);
\draw[color=black] (-1,0) -- (-1,-2);
\draw[color=black] (-1,-2) arc (180:360:0.5);
\draw[color=black, ->] (-1,0) -- (-1,-1);
} \end{array}
\qquad
\begin{array}{c}
\tikz[scale=0.5]{
\draw[color=black] (-1,0) -- (-1,-0.5) -- (0,-1.5) -- (0,-2);
\draw[color=black, ->] (-1,-0.5) -- (-1,-0.25);
\draw[color=black, ->] (0,-2) -- (0,-1.75);
\draw[color=black] (0,0) -- (0,-0.5) -- (-1,-1.5) -- (-1,-3.5);
\draw[color=black, ->] (0,-0.5) -- (0,-0.25);
\draw[color=black, ->] (-1,-3.5) -- (-1,-2.5);
\draw[color=black] (1,0) -- (1,-2);
\draw[color=black] (1,-2) arc (360:180:0.5);
\draw[color=black, ->] (1,0) -- (1,-1);
} \end{array}
\overset{r^2_8}{\Rrightarrow}
\begin{array}{c}
\tikz[scale=0.5]{
\draw[color=black] (1,0) -- (1,-0.5) -- (0,-1.5) -- (0,-2);
\draw[color=black, ->] (1,0) -- (1,-0.25);
\draw[color=black, ->] (0,-1.5) -- (0,-1.75);
\draw[color=black] (0,0) -- (0,-0.5) -- (1,-1.5) -- (1,-3.5);
\draw[color=black, ->] (0,-0.5) -- (0,-0.25);
\draw[color=black, ->] (1,-3.5) -- (1,-2.5);
\draw[color=black] (-1,0) -- (-1,-2);
\draw[color=black] (-1,-2) arc (180:360:0.5);
\draw[color=black, ->] (-1,-2) -- (-1,-1);
} \end{array}
\end{gather*}

\begin{gather*}
\begin{array}{c}
\tikz[scale=0.5]{
\draw[color=black] (-1,0) -- (-1,0.5) -- (0,1.5) -- (0,2) -- (1,3) -- (1,5);
\draw[color=black, ->] (-1,0) -- (-1,0.25);
\draw[color=black, ->] (0,1.5) -- (0,1.75);
\draw[color=black, ->] (1,3) -- (1,4);
\draw[color=black] (0,0) -- (0,0.5) -- (-1,1.5) -- (-1,3.5);
\draw[color=black, ->] (0,0) -- (0,0.25);
\draw[color=black, ->] (-1,1.5) -- (-1,2.5);
\draw[color=black] (1,0) -- (1,2) -- (0,3) -- (0,3.5);
\draw[color=black] (0,3.5) arc (0:180:0.5);
\draw[color=black, ->] (1,2) -- (1,1);
\draw[color=black, ->] (0,3.5) -- (0,3.25);
} \end{array}
\overset{r^3_1}{\Rrightarrow}
\begin{array}{c}
\tikz[scale=0.5]{
\draw[color=black] (-1,0) -- (-1,2);
\draw[color=black, ->] (-1,0) -- (-1,1);
\draw[color=black] (0,0) -- (0,0.5);
\draw[color=black, ->] (0,0) -- (0,0.25);
\draw[color=black] (1,0) -- (1,0.5);
\draw[color=black] (0,0.5) arc (180:0:0.5);
\draw[color=black, ->] (1,0.5) -- (1,0.25);
} \end{array}
\qquad
\begin{array}{c}
\tikz[scale=0.5]{
\draw[color=black] (-1,0) -- (-1,0.5) -- (0,1.5) -- (0,2) -- (1,3) -- (1,5);
\draw[color=black, ->] (-1,0) -- (-1,0.25);
\draw[color=black, ->] (0,1.5) -- (0,1.75);
\draw[color=black, ->] (1,3) -- (1,4);
\draw[color=black] (0,0) -- (0,0.5) -- (-1,1.5) -- (-1,3.5);
\draw[color=black, ->] (0,0.5) -- (0,0.25);
\draw[color=black, ->] (-1,3.5) -- (-1,2.5);
\draw[color=black] (1,0) -- (1,2) -- (0,3) -- (0,3.5);
\draw[color=black] (0,3.5) arc (0:180:0.5);
\draw[color=black, ->] (1,0) -- (1,1);
\draw[color=black, ->] (0,3) -- (0,3.25);
} \end{array}
\overset{r^3_2}{\Rrightarrow}
\begin{array}{c}
\tikz[scale=0.5]{
\draw[color=black] (-1,0) -- (-1,2);
\draw[color=black, ->] (-1,0) -- (-1,1);
\draw[color=black] (0,0) -- (0,0.5);
\draw[color=black, ->] (0,0.5) -- (0,0.25);
\draw[color=black] (1,0) -- (1,0.5);
\draw[color=black] (0,0.5) arc (180:0:0.5);
\draw[color=black, ->] (1,0) -- (1,0.25);
} \end{array}
\qquad
\begin{array}{c}
\tikz[scale=0.5]{
\draw[color=black] (-1,0) -- (-1,0.5) -- (0,1.5) -- (0,2) -- (1,3) -- (1,5);
\draw[color=black, ->] (-1,0.5) -- (-1,0.25);
\draw[color=black, ->] (0,2) -- (0,1.75);
\draw[color=black, ->] (1,5) -- (1,4);
\draw[color=black] (0,0) -- (0,0.5) -- (-1,1.5) -- (-1,3.5);
\draw[color=black, ->] (0,0) -- (0,0.25);
\draw[color=black, ->] (-1,1.5) -- (-1,2.5);
\draw[color=black] (1,0) -- (1,2) -- (0,3) -- (0,3.5);
\draw[color=black] (0,3.5) arc (0:180:0.5);
\draw[color=black, ->] (1,2) -- (1,1);
\draw[color=black, ->] (0,3.5) -- (0,3.25);
} \end{array}
\overset{r^3_3}{\Rrightarrow}
\begin{array}{c}
\tikz[scale=0.5]{
\draw[color=black] (-1,0) -- (-1,2);
\draw[color=black, ->] (-1,2) -- (-1,1);
\draw[color=black] (0,0) -- (0,0.5);
\draw[color=black, ->] (0,0) -- (0,0.25);
\draw[color=black] (1,0) -- (1,0.5);
\draw[color=black] (0,0.5) arc (180:0:0.5);
\draw[color=black, ->] (1,0.5) -- (1,0.25);
} \end{array}
\qquad
\begin{array}{c}
\tikz[scale=0.5]{
\draw[color=black] (-1,0) -- (-1,0.5) -- (0,1.5) -- (0,2) -- (1,3) -- (1,5);
\draw[color=black, ->] (-1,0.5) -- (-1,0.25);
\draw[color=black, ->] (0,2) -- (0,1.75);
\draw[color=black, ->] (1,5) -- (1,4);
\draw[color=black] (0,0) -- (0,0.5) -- (-1,1.5) -- (-1,3.5);
\draw[color=black, ->] (0,0.5) -- (0,0.25);
\draw[color=black, ->] (-1,3.5) -- (-1,2.5);
\draw[color=black] (1,0) -- (1,2) -- (0,3) -- (0,3.5);
\draw[color=black] (0,3.5) arc (0:180:0.5);
\draw[color=black, ->] (1,0) -- (1,1);
\draw[color=black, ->] (0,3) -- (0,3.25);
} \end{array}
\overset{r^3_4}{\Rrightarrow}
\begin{array}{c}
\tikz[scale=0.5]{
\draw[color=black] (-1,0) -- (-1,2);
\draw[color=black, ->] (-1,2) -- (-1,1);
\draw[color=black] (0,0) -- (0,0.5);
\draw[color=black, ->] (0,0.5) -- (0,0.25);
\draw[color=black] (1,0) -- (1,0.5);
\draw[color=black] (0,0.5) arc (180:0:0.5);
\draw[color=black, ->] (1,0) -- (1,0.25);
} \end{array}
\end{gather*}

\begin{gather*}
\begin{array}{c}
\tikz[scale=0.5]{
\draw[color=black] (-1,0) -- (-1,-0.5) -- (0,-1.5) -- (0,-2) -- (1,-3) -- (1,-5);
\draw[color=black, ->] (-1,0) -- (-1,-0.25);
\draw[color=black, ->] (0,-1.5) -- (0,-1.75);
\draw[color=black, ->] (1,-3) -- (1,-4);
\draw[color=black] (0,0) -- (0,-0.5) -- (-1,-1.5) -- (-1,-3.5);
\draw[color=black, ->] (0,0) -- (0,-0.25);
\draw[color=black, ->] (-1,-1.5) -- (-1,-2.5);
\draw[color=black] (1,0) -- (1,-2) -- (0,-3) -- (0,-3.5);
\draw[color=black] (0,-3.5) arc (360:180:0.5);
\draw[color=black, ->] (1,-2) -- (1,-1);
\draw[color=black, ->] (0,-3.5) -- (0,-3.25);
} \end{array}
\overset{r^3_5}{\Rrightarrow}
\begin{array}{c}
\tikz[scale=0.5]{
\draw[color=black] (-1,0) -- (-1,-2);
\draw[color=black, ->] (-1,0) -- (-1,-1);
\draw[color=black] (0,0) -- (0,-0.5);
\draw[color=black, ->] (0,0) -- (0,-0.25);
\draw[color=black] (1,0) -- (1,-0.5);
\draw[color=black] (0,-0.5) arc (180:360:0.5);
\draw[color=black, ->] (1,-0.5) -- (1,-0.25);
} \end{array}
\qquad
\begin{array}{c}
\tikz[scale=0.5]{
\draw[color=black] (-1,0) -- (-1,-0.5) -- (0,-1.5) -- (0,-2) -- (1,-3) -- (1,-5);
\draw[color=black, ->] (-1,0) -- (-1,-0.25);
\draw[color=black, ->] (0,-1.5) -- (0,-1.75);
\draw[color=black, ->] (1,-3) -- (1,-4);
\draw[color=black] (0,0) -- (0,-0.5) -- (-1,-1.5) -- (-1,-3.5);
\draw[color=black, ->] (0,-0.5) -- (0,-0.25);
\draw[color=black, ->] (-1,-3.5) -- (-1,-2.5);
\draw[color=black] (1,0) -- (1,-2) -- (0,-3) -- (0,-3.5);
\draw[color=black] (0,-3.5) arc (360:180:0.5);
\draw[color=black, ->] (1,0) -- (1,-1);
\draw[color=black, ->] (0,-3) -- (0,-3.25);
} \end{array}
\overset{r^3_6}{\Rrightarrow}
\begin{array}{c}
\tikz[scale=0.5]{
\draw[color=black] (-1,0) -- (-1,-2);
\draw[color=black, ->] (-1,0) -- (-1,-1);
\draw[color=black] (0,0) -- (0,-0.5);
\draw[color=black, ->] (0,-0.5) -- (0,-0.25);
\draw[color=black] (1,0) -- (1,-0.5);
\draw[color=black] (0,-0.5) arc (180:360:0.5);
\draw[color=black, ->] (1,0) -- (1,-0.25);
} \end{array}
\qquad
\begin{array}{c}
\tikz[scale=0.5]{
\draw[color=black] (-1,0) -- (-1,-0.5) -- (0,-1.5) -- (0,-2) -- (1,-3) -- (1,-5);
\draw[color=black, ->] (-1,-0.5) -- (-1,-0.25);
\draw[color=black, ->] (0,-2) -- (0,-1.75);
\draw[color=black, ->] (1,-5) -- (1,-4);
\draw[color=black] (0,0) -- (0,-0.5) -- (-1,-1.5) -- (-1,-3.5);
\draw[color=black, ->] (0,0) -- (0,-0.25);
\draw[color=black, ->] (-1,-1.5) -- (-1,-2.5);
\draw[color=black] (1,0) -- (1,-2) -- (0,-3) -- (0,-3.5);
\draw[color=black] (0,-3.5) arc (360:180:0.5);
\draw[color=black, ->] (1,-2) -- (1,-1);
\draw[color=black, ->] (0,-3.5) -- (0,-3.25);
} \end{array}
\overset{r^3_7}{\Rrightarrow}
\begin{array}{c}
\tikz[scale=0.5]{
\draw[color=black] (-1,0) -- (-1,-2);
\draw[color=black, ->] (-1,-2) -- (-1,-1);
\draw[color=black] (0,0) -- (0,-0.5);
\draw[color=black, ->] (0,0) -- (0,-0.25);
\draw[color=black] (1,0) -- (1,-0.5);
\draw[color=black] (0,-0.5) arc (180:360:0.5);
\draw[color=black, ->] (1,-0.5) -- (1,-0.25);
} \end{array}
\qquad
\begin{array}{c}
\tikz[scale=0.5]{
\draw[color=black] (-1,0) -- (-1,-0.5) -- (0,-1.5) -- (0,-2) -- (1,-3) -- (1,-5);
\draw[color=black, ->] (-1,-0.5) -- (-1,-0.25);
\draw[color=black, ->] (0,-2) -- (0,-1.75);
\draw[color=black, ->] (1,-5) -- (1,-4);
\draw[color=black] (0,0) -- (0,-0.5) -- (-1,-1.5) -- (-1,-3.5);
\draw[color=black, ->] (0,-0.5) -- (0,-0.25);
\draw[color=black, ->] (-1,-3.5) -- (-1,-2.5);
\draw[color=black] (1,0) -- (1,-2) -- (0,-3) -- (0,-3.5);
\draw[color=black] (0,-3.5) arc (360:180:0.5);
\draw[color=black, ->] (1,0) -- (1,-1);
\draw[color=black, ->] (0,-3) -- (0,-3.25);
} \end{array}
\overset{r^3_8}{\Rrightarrow}
\begin{array}{c}
\tikz[scale=0.5]{
\draw[color=black] (-1,0) -- (-1,-2);
\draw[color=black, ->] (-1,-2) -- (-1,-1);
\draw[color=black] (0,0) -- (0,-0.5);
\draw[color=black, ->] (0,-0.5) -- (0,-0.25);
\draw[color=black] (1,0) -- (1,-0.5);
\draw[color=black] (0,-0.5) arc (180:360:0.5);
\draw[color=black, ->] (1,0) -- (1,-0.25);
} \end{array}
\end{gather*}

\begin{gather*}
\begin{array}{c}
\tikz[scale=0.5]{
\draw[color=black] (0,0) -- (0,0.5) -- (1,1.5) -- (1,2);
\draw[color=black] (1,0) -- (1,0.5) -- (0,1.5) -- (0,2);
\draw[color=black] (0,2) arc (180:0:0.5);
\draw[color=black, ->] (0,0) -- (0,0.25);
\draw[color=black, ->] (1,1.5) -- (1,1.75);
\draw[color=black, ->] (1,0.5) -- (1,0.25);
\draw[color=black, ->] (0,2) -- (0,1.75);
} \end{array}
\overset{r^4_1}{\Rrightarrow}
\begin{array}{c}
\tikz[scale=0.5]{
\draw[color=black] (0,0) -- (0,0.5);
\draw[color=black] (1,0) -- (1,0.5);
\draw[color=black] (0,0.5) arc (180:0:0.5);
\draw[color=black, ->] (0,0) -- (0,0.25);
\draw[color=black, ->] (1,0.5) -- (1,0.25);
} \end{array}
\qquad
\begin{array}{c}
\tikz[scale=0.5]{
\draw[color=black] (0,0) -- (0,0.5) -- (1,1.5) -- (1,2);
\draw[color=black] (1,0) -- (1,0.5) -- (0,1.5) -- (0,2);
\draw[color=black] (0,2) arc (180:0:0.5);
\draw[color=black, ->] (0,0.5) -- (0,0.25);
\draw[color=black, ->] (1,2) -- (1,1.75);
\draw[color=black, ->] (1,0) -- (1,0.25);
\draw[color=black, ->] (0,1.5) -- (0,1.75);
} \end{array}
\overset{r^4_2}{\Rrightarrow}
\begin{array}{c}
\tikz[scale=0.5]{
\draw[color=black] (0,0) -- (0,0.5);
\draw[color=black] (1,0) -- (1,0.5);
\draw[color=black] (0,0.5) arc (180:0:0.5);
\draw[color=black, ->] (0,0.5) -- (0,0.25);
\draw[color=black, ->] (1,0) -- (1,0.25);
} \end{array}
\qquad
\begin{array}{c}
\tikz[scale=0.5]{
\draw[color=black] (0,0) -- (0,-0.5) -- (1,-1.5) -- (1,-2);
\draw[color=black] (1,0) -- (1,-0.5) -- (0,-1.5) -- (0,-2);
\draw[color=black] (0,-2) arc (180:360:0.5);
\draw[color=black, ->] (0,0) -- (0,-0.25);
\draw[color=black, ->] (1,-1.5) -- (1,-1.75);
\draw[color=black, ->] (1,-0.5) -- (1,-0.25);
\draw[color=black, ->] (0,-2) -- (0,-1.75);
} \end{array}
\overset{r^4_3}{\Rrightarrow}
\begin{array}{c}
\tikz[scale=0.5]{
\draw[color=black] (0,0) -- (0,-0.5);
\draw[color=black] (1,0) -- (1,-0.5);
\draw[color=black] (0,-0.5) arc (180:360:0.5);
\draw[color=black, ->] (0,0) -- (0,-0.25);
\draw[color=black, ->] (1,-0.5) -- (1,-0.25);
} \end{array}
\qquad
\begin{array}{c}
\tikz[scale=0.5]{
\draw[color=black] (0,0) -- (0,-0.5) -- (1,-1.5) -- (1,-2);
\draw[color=black] (1,0) -- (1,-0.5) -- (0,-1.5) -- (0,-2);
\draw[color=black] (0,-2) arc (180:360:0.5);
\draw[color=black, ->] (0,-0.5) -- (0,-0.25);
\draw[color=black, ->] (1,-2) -- (1,-1.75);
\draw[color=black, ->] (1,0) -- (1,-0.25);
\draw[color=black, ->] (0,-1.5) -- (0,-1.75);
} \end{array}
\overset{r^4_4}{\Rrightarrow}
\begin{array}{c}
\tikz[scale=0.5]{
\draw[color=black] (0,0) -- (0,-0.5);
\draw[color=black] (1,0) -- (1,-0.5);
\draw[color=black] (0,-0.5) arc (180:360:0.5);
\draw[color=black, ->] (0,-0.5) -- (0,-0.25);
\draw[color=black, ->] (1,0) -- (1,-0.25);
} \end{array}
\end{gather*}

\begin{gather*}
\begin{array}{c}
\tikz[scale=0.5]{
\draw[color=black] (0,0) -- (0,1) -- (-1,2) -- (-1,2.5);
\draw[color=black] (0,3) -- (0,2) -- (-1,1) -- (-1,0.5);
\draw[color=black] (-2,0.5) -- (-2,2.5);
\draw[color=black] (-1,2.5) arc (0:180:0.5);
\draw[color=black] (-1,0.5) arc (360:180:0.5);
\draw[color=black, ->] (0,0) -- (0,0.5);
\draw[color=black, ->] (0,2) -- (0,2.5);
\draw[color=black, ->] (-1,2) -- (-1,2.25);
\draw[color=black, ->] (-1,0.5) -- (-1,0.75);
\draw[color=black, ->] (-2,2.5) -- (-2,1.5);
} \end{array}
\overset{r^5_1}{\Rrightarrow}
\begin{array}{c}
\tikz[scale=0.5]{
\draw[color=black, ->] (0,0) -- (0,0.5);
\draw[color=black] (0,0) -- (0,1);
} \end{array}
\qquad
\begin{array}{c}
\tikz[scale=0.5]{
\draw[color=black] (0,0) -- (0,1) -- (-1,2) -- (-1,2.5);
\draw[color=black] (0,3) -- (0,2) -- (-1,1) -- (-1,0.5);
\draw[color=black] (-2,0.5) -- (-2,2.5);
\draw[color=black] (-1,2.5) arc (0:180:0.5);
\draw[color=black] (-1,0.5) arc (360:180:0.5);
\draw[color=black, ->] (0,1) -- (0,0.5);
\draw[color=black, ->] (0,3) -- (0,2.5);
\draw[color=black, ->] (-1,2.5) -- (-1,2.25);
\draw[color=black, ->] (-1,1) -- (-1,0.75);
\draw[color=black, ->] (-2,0.5) -- (-2,1.5);
} \end{array}
\overset{r^5_2}{\Rrightarrow}
\begin{array}{c}
\tikz[scale=0.5]{
\draw[color=black, ->] (0,1) -- (0,0.5);
\draw[color=black] (0,0) -- (0,1);
} \end{array}
\end{gather*}

Ordering 3-cells

\begin{gather*}
\begin{array}{c}
\tikz[scale=0.5]{
\draw[color=black] (0,0) -- (0,1) -- (1,2) -- (1,3);
\draw[color=black] (1,0) -- (1,1) -- (0,2) -- (0,3);
\draw[color=black, ->] (1,0) -- (1,0.25);
\draw[thick] plot[mark=*] (1,0.5);
\draw[color=black, ->] (0,0) -- (0,0.5);
\draw[color=black, ->] (1,0.5) -- (1,0.8);
\draw[color=black, ->] (0,2) -- (0,2.5);
\draw[color=black, ->] (1,2) -- (1,2.5);
} \end{array}
\overset{o^0_1}{\Rrightarrow}
\begin{array}{c}
\tikz[scale=0.5]{
\draw[color=black] (0,0) -- (0,1) -- (1,2) -- (1,3);
\draw[color=black] (1,0) -- (1,1) -- (0,2) -- (0,3);
\draw[color=black, ->] (0,2.5) -- (0,2.8);
\draw[thick] plot[mark=*] (0,2.5);
\draw[color=black, ->] (0,0) -- (0,0.5);
\draw[color=black, ->] (1,0) -- (1,0.5);
\draw[color=black, ->] (0,2) -- (0,2.25);
\draw[color=black, ->] (1,2) -- (1,2.5);
} \end{array}
+
\begin{array}{c}
\tikz[scale=0.5]{
\draw[color=black] (0,0) -- (0,3);
\draw[color=black] (1,0) -- (1,3);
\draw[color=black, ->] (0,0) -- (0,1.5);
\draw[color=black, ->] (1,0) -- (1,1.5);
} \end{array}
\qquad
\begin{array}{c}
\tikz[scale=0.5]{
\draw[color=black] (0,0) -- (0,1) -- (1,2) -- (1,3);
\draw[color=black] (1,0) -- (1,1) -- (0,2) -- (0,3);
\draw[color=black, ->] (0,0) -- (0,0.25);
\draw[thick] plot[mark=*] (0,0.5);
\draw[color=black, ->] (0,0.5) -- (0,0.8);
\draw[color=black, ->] (1,0) -- (1,0.5);
\draw[color=black, ->] (0,2) -- (0,2.5);
\draw[color=black, ->] (1,2) -- (1,2.5);
} \end{array}
\overset{o^0_2}{\Rrightarrow}
\begin{array}{c}
\tikz[scale=0.5]{
\draw[color=black] (0,0) -- (0,1) -- (1,2) -- (1,3);
\draw[color=black] (1,0) -- (1,1) -- (0,2) -- (0,3);
\draw[color=black, ->] (1,2.5) -- (1,2.8);
\draw[thick] plot[mark=*] (1,2.5);
\draw[color=black, ->] (0,0) -- (0,0.5);
\draw[color=black, ->] (1,0) -- (1,0.5);
\draw[color=black, ->] (0,2) -- (0,2.5);
\draw[color=black, ->] (1,2) -- (1,2.25);
} \end{array}
-
\begin{array}{c}
\tikz[scale=0.5]{
\draw[color=black] (0,0) -- (0,3);
\draw[color=black] (1,0) -- (1,3);
\draw[color=black, ->] (0,0) -- (0,1.5);
\draw[color=black, ->] (1,0) -- (1,1.5);
} \end{array}
\qquad
\begin{array}{c}
\tikz[scale=0.5]{
\draw[color=black] (0,0) -- (0,-1) -- (1,-2) -- (1,-3);
\draw[color=black] (1,0) -- (1,-1) -- (0,-2) -- (0,-3);
\draw[color=black, ->] (1,0) -- (1,-0.25);
\draw[thick] plot[mark=*] (1,-0.5);
\draw[color=black, ->] (0,0) -- (0,-0.5);
\draw[color=black, ->] (1,-0.5) -- (1,-0.8);
\draw[color=black, ->] (0,-2) -- (0,-2.5);
\draw[color=black, ->] (1,-2) -- (1,-2.5);
} \end{array}
\overset{o^0_3}{\Rrightarrow}
\begin{array}{c}
\tikz[scale=0.5]{
\draw[color=black] (0,0) -- (0,-1) -- (1,-2) -- (1,-3);
\draw[color=black] (1,0) -- (1,-1) -- (0,-2) -- (0,-3);
\draw[color=black, ->] (0,-2.5) -- (0,-2.8);
\draw[thick] plot[mark=*] (0,-2.5);
\draw[color=black, ->] (0,0) -- (0,-0.5);
\draw[color=black, ->] (1,0) -- (1,-0.5);
\draw[color=black, ->] (0,-2) -- (0,-2.25);
\draw[color=black, ->] (1,-2) -- (1,-2.5);
} \end{array}
-
\begin{array}{c}
\tikz[scale=0.5]{
\draw[color=black] (0,0) -- (0,-3);
\draw[color=black] (1,0) -- (1,-3);
\draw[color=black, ->] (0,0) -- (0,-1.5);
\draw[color=black, ->] (1,0) -- (1,-1.5);
} \end{array}
\qquad
\begin{array}{c}
\tikz[scale=0.5]{
\draw[color=black] (0,0) -- (0,-1) -- (1,-2) -- (1,-3);
\draw[color=black] (1,0) -- (1,-1) -- (0,-2) -- (0,-3);
\draw[color=black, ->] (0,0) -- (0,-0.25);
\draw[thick] plot[mark=*] (0,-0.5);
\draw[color=black, ->] (0,-0.5) -- (0,-0.8);
\draw[color=black, ->] (1,0) -- (1,-0.5);
\draw[color=black, ->] (0,-2) -- (0,-2.5);
\draw[color=black, ->] (1,-2) -- (1,-2.5);
} \end{array}
\overset{o^0_4}{\Rrightarrow}
\begin{array}{c}
\tikz[scale=0.5]{
\draw[color=black] (0,0) -- (0,-1) -- (1,-2) -- (1,-3);
\draw[color=black] (1,0) -- (1,-1) -- (0,-2) -- (0,-3);
\draw[color=black, ->] (1,-2.5) -- (1,-2.8);
\draw[thick] plot[mark=*] (1,-2.5);
\draw[color=black, ->] (0,0) -- (0,-0.5);
\draw[color=black, ->] (1,0) -- (1,-0.5);
\draw[color=black, ->] (0,-2) -- (0,-2.5);
\draw[color=black, ->] (1,-2) -- (1,-2.25);
} \end{array}
+
\begin{array}{c}
\tikz[scale=0.5]{
\draw[color=black] (0,0) -- (0,-3);
\draw[color=black] (1,0) -- (1,-3);
\draw[color=black, ->] (0,0) -- (0,-1.5);
\draw[color=black, ->] (1,0) -- (1,-1.5);
} \end{array}
\end{gather*}

\begin{gather*}
\begin{array}{c}
\tikz[scale=0.5]{
\draw[color=black] (0,0) -- (0,1) -- (1,2) -- (1,3);
\draw[color=black] (1,0) -- (1,1) -- (0,2) -- (0,3);
\draw[color=black, ->] (1,0) -- (1,0.25);
\draw[thick] plot[mark=*] (1,0.5);
\draw[color=black, ->] (0,1) -- (0,0.5);
\draw[color=black, ->] (1,0.5) -- (1,0.8);
\draw[color=black, ->] (0,2) -- (0,2.5);
\draw[color=black, ->] (1,3) -- (1,2.5);
} \end{array}
\overset{o^0_5}{\Rrightarrow}
\begin{array}{c}
\tikz[scale=0.5]{
\draw[color=black] (0,0) -- (0,1) -- (1,2) -- (1,3);
\draw[color=black] (1,0) -- (1,1) -- (0,2) -- (0,3);
\draw[color=black, ->] (0,2.5) -- (0,2.8);
\draw[thick] plot[mark=*] (0,2.5);
\draw[color=black, ->] (0,1) -- (0,0.5);
\draw[color=black, ->] (1,0) -- (1,0.5);
\draw[color=black, ->] (0,2) -- (0,2.25);
\draw[color=black, ->] (1,3) -- (1,2.5);
} \end{array}
-
\begin{array}{c}
\tikz[scale=0.5]{
\draw[color=black] (0,0) -- (0,0.5);
\draw[color=black] (1,0) -- (1,0.5);
\draw[color=black] (1,0.5) arc (0:180:0.5);
\draw[color=black, ->] (0,0.5) -- (0,0.25);
\draw[color=black, ->] (1,0) -- (1,0.25);
\draw[color=black] (0,2) -- (0,2.5);
\draw[color=black] (1,2) -- (1,2.5);
\draw[color=black] (1,2) arc (360:180:0.5);
\draw[color=black, ->] (0,2) -- (0,2.25);
\draw[color=black, ->] (1,2.5) -- (1,2.25);
} \end{array}
\qquad
\begin{array}{c}
\tikz[scale=0.5]{
\draw[color=black] (0,0) -- (0,1) -- (-1,2) -- (-1,3);
\draw[color=black] (-1,0) -- (-1,1) -- (0,2) -- (0,3);
\draw[color=black, ->] (-1,0) -- (-1,0.25);
\draw[thick] plot[mark=*] (-1,0.5);
\draw[color=black, ->] (0,1) -- (0,0.5);
\draw[color=black, ->] (-1,0.5) -- (-1,0.8);
\draw[color=black, ->] (0,2) -- (0,2.5);
\draw[color=black, ->] (-1,3) -- (-1,2.5);
} \end{array}
\overset{o^0_6}{\Rrightarrow}
\begin{array}{c}
\tikz[scale=0.5]{
\draw[color=black] (0,0) -- (0,1) -- (-1,2) -- (-1,3);
\draw[color=black] (-1,0) -- (-1,1) -- (0,2) -- (0,3);
\draw[color=black, ->] (0,2.5) -- (0,2.8);
\draw[thick] plot[mark=*] (0,2.5);
\draw[color=black, ->] (0,1) -- (0,0.5);
\draw[color=black, ->] (-1,0) -- (-1,0.5);
\draw[color=black, ->] (0,2) -- (0,2.25);
\draw[color=black, ->] (-1,3) -- (-1,2.5);
} \end{array}
+
\begin{array}{c}
\tikz[scale=0.5]{
\draw[color=black] (0,0) -- (0,0.5);
\draw[color=black] (-1,0) -- (-1,0.5);
\draw[color=black] (0,0.5) arc (0:180:0.5);
\draw[color=black, ->] (0,0.5) -- (0,0.25);
\draw[color=black, ->] (-1,0) -- (-1,0.25);
\draw[color=black] (0,2) -- (0,2.5);
\draw[color=black] (-1,2) -- (-1,2.5);
\draw[color=black] (0,2) arc (360:180:0.5);
\draw[color=black, ->] (0,2) -- (0,2.25);
\draw[color=black, ->] (-1,2.5) -- (-1,2.25);
} \end{array}
\qquad
\begin{array}{c}
\tikz[scale=0.5]{
\draw[color=black] (0,0) -- (0,-1) -- (1,-2) -- (1,-3);
\draw[color=black] (1,0) -- (1,-1) -- (0,-2) -- (0,-3);
\draw[color=black, ->] (1,0) -- (1,-0.25);
\draw[thick] plot[mark=*] (1,-0.5);
\draw[color=black, ->] (0,-1) -- (0,-0.5);
\draw[color=black, ->] (1,-0.5) -- (1,-0.8);
\draw[color=black, ->] (0,-2) -- (0,-2.5);
\draw[color=black, ->] (1,-3) -- (1,-2.5);
} \end{array}
\overset{o^0_7}{\Rrightarrow}
\begin{array}{c}
\tikz[scale=0.5]{
\draw[color=black] (0,0) -- (0,-1) -- (1,-2) -- (1,-3);
\draw[color=black] (1,0) -- (1,-1) -- (0,-2) -- (0,-3);
\draw[color=black, ->] (0,-2.5) -- (0,-2.8);
\draw[thick] plot[mark=*] (0,-2.5);
\draw[color=black, ->] (0,-1) -- (0,-0.5);
\draw[color=black, ->] (1,0) -- (1,-0.5);
\draw[color=black, ->] (0,-2) -- (0,-2.25);
\draw[color=black, ->] (1,-3) -- (1,-2.5);
} \end{array}
+
\begin{array}{c}
\tikz[scale=0.5]{
\draw[color=black] (0,0) -- (0,-0.5);
\draw[color=black] (1,0) -- (1,-0.5);
\draw[color=black] (1,-0.5) arc (360:180:0.5);
\draw[color=black, ->] (0,-0.5) -- (0,-0.25);
\draw[color=black, ->] (1,0) -- (1,-0.25);
\draw[color=black] (0,-2) -- (0,-2.5);
\draw[color=black] (1,-2) -- (1,-2.5);
\draw[color=black] (1,-2) arc (0:180:0.5);
\draw[color=black, ->] (0,-2) -- (0,-2.25);
\draw[color=black, ->] (1,-2.5) -- (1,-2.25);
} \end{array}
\qquad
\begin{array}{c}
\tikz[scale=0.5]{
\draw[color=black] (0,0) -- (0,-1) -- (-1,-2) -- (-1,-3);
\draw[color=black] (-1,0) -- (-1,-1) -- (0,-2) -- (0,-3);
\draw[color=black, ->] (-1,0) -- (-1,-0.25);
\draw[thick] plot[mark=*] (-1,-0.5);
\draw[color=black, ->] (0,-1) -- (0,-0.5);
\draw[color=black, ->] (-1,-0.5) -- (-1,-0.8);
\draw[color=black, ->] (0,-2) -- (0,-2.5);
\draw[color=black, ->] (-1,-3) -- (-1,-2.5);
} \end{array}
\overset{o^0_8}{\Rrightarrow}
\begin{array}{c}
\tikz[scale=0.5]{
\draw[color=black] (0,0) -- (0,-1) -- (-1,-2) -- (-1,-3);
\draw[color=black] (-1,0) -- (-1,-1) -- (0,-2) -- (0,-3);
\draw[color=black, ->] (0,-2.5) -- (0,-2.75);
\draw[thick] plot[mark=*] (0,-2.5);
\draw[color=black, ->] (0,-1) -- (0,-0.5);
\draw[color=black, ->] (-1,0) -- (-1,-0.5);
\draw[color=black, ->] (0,-2) -- (0,-2.2);
\draw[color=black, ->] (-1,-3) -- (-1,-2.5);
} \end{array}
-
\begin{array}{c}
\tikz[scale=0.5]{
\draw[color=black] (0,0) -- (0,-0.5);
\draw[color=black] (-1,0) -- (-1,-0.5);
\draw[color=black] (0,-0.5) arc (360:180:0.5);
\draw[color=black, ->] (0,-0.5) -- (0,-0.25);
\draw[color=black, ->] (-1,0) -- (-1,-0.25);
\draw[color=black] (0,-2) -- (0,-2.5);
\draw[color=black] (-1,-2) -- (-1,-2.5);
\draw[color=black] (0,-2) arc (0:180:0.5);
\draw[color=black, ->] (0,-2) -- (0,-2.25);
\draw[color=black, ->] (-1,-2.5) -- (-1,-2.25);
} \end{array}
\end{gather*}

\begin{gather*}
\begin{array}{c}
\tikz[scale=0.5]{
\draw[color=black] (0,0) -- (0,0.5);
\draw[color=black] (1,0) -- (1,0.5);
\draw[color=black] (0,0.5) arc (180:0:0.5);
\draw[color=black] (0,0) arc (180:360:0.5);
\draw[color=black, ->] (0,0.5) -- (0,0.25);
\draw[color=black, ->] (1,0) -- (1,0.25);
} \end{array}%
\qquad
\cdots 
\end{gather*}
\end{itemize}

\subsubsection{Remark}

The last three families of 3-cells correspond to infinite families of relations which can be calculated by induction. Those relations are noted as in Convention (\ref{convention}). The first induction formula is given in \cite{BCNR}. If we denote the counterclockwise bubble with~$n$ dots
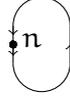
\begin{gather*}
%
\end{gather*}
as $u_{n,0}$ and by expressing the clockwise bubbles with~$m$ dots
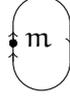
\begin{gather*}
%
\end{gather*}
as $u_{0,m}$, we have for any $n$ and~$m$ of $\mathbb{N}$:
$$u_{n+1,m}=u_{n,m+1}-u_{n,0}\star_1u_{0,m}.$$
Which can be used to rewrite $u_{n,0}$ as a linear combination of the family~$(u_{0,i}^j)_{0 \leqslant i \leqslant n, j\in \mathbb{N}}$. The others induction forumulas are:
$$v_{n+1,m}=v_{n,m+1}+w_{n+m}+\sum^{n-1}_{k=0}w_{m+n-k}\star_1v_{k,0},$$
$$v'_{n+1,m}=v'_{n,m+1}+w'_{n+m}-\sum^{n-1}_{k=0}v'_{0,k}\star_1w'_{m+n-k}$$
where $v_{n,0}$, $v_{0,m}$, $v'_{n,0}$ and $v'_{0,m}$ respectively denote:
\begin{gather*}
%
\end{gather*}

Let us now expand on the first chapter of \cite{Turaev}. This chapter defines first ribbon categories as braided monoidal categories with duals and twist. The twist is a natural transformation $\theta$ from the identity functor to itself satisfying:
$$\theta_{V \otimes W}=b_{W,V} \circ (\theta_V \otimes \theta_W) \circ b_{V,W}$$
for each objects $V$ and $W$, where $b$ denotes the braiding. A ribbon category satisfies the axiom $\theta^*=\theta$.

An example of ribbon category is the category of ribbon tangles on a set $S$ of colors. This strict monoidal category $\mathcal{RIB}_S$ is defined by:
\begin{itemize}
\item the objects of $\mathcal{RIB}_S$ are the words of the free monoid on $S \times \{\vee,\wedge\}$,
\item the morphisms of $\mathcal{RIB}_S$ from a word $u$ to a word $v$ are the oriented tangles of colored ribbons such that the word $u$ is on the upper boundary and the word $v$ is on the lower boundary,
\item two isotopic tangles are equal.
\end{itemize}

In this category, each object is its own dual. The twist corresponds to transversally twisting a ribbon by 360 degrees. Turaev then gives a presentation by generators and relations of $\mathcal{RIB}_S$. The generators are the cups, caps, crossings and twistings of each ribbon colors and directions. The relations are invariance by multiple moves Turaev describes as elementary isotopies and Reidemeister moves.

From the presentation of $\mathcal{RIB}_S$ by generators and relations, we can present the category $\mathcal{RIB}$ of ribbon tangles with only one color by generators and relations. $\mathcal{RIB}$ being a monoidal category, we can describe it as a 2-category with only one 0-cell. The linearization of $\mathcal{RIB}$ is thus a linear $(2,2)$-category with the same 0-cell and 1-cells than $\mathcal{RIB}$. We will call $\mathcal{RIB}_{\mathbb{K}}$ this linear $(2,2)$-category, where $\mathbb{K}$ is our fixed field. Let $\mathcal{OB}$ be the subcategory of $\mathcal{AOB}$ defined by:
\begin{itemize}
\item $\mathcal{OB}_0=\mathcal{AOB}_0$ and $\mathcal{OB}_1=\mathcal{AOB}_1$,
\item $\mathcal{OB}_2$ is made of the oriented Brauer diagrams with bubbles (without dots)
\end{itemize}
The linear 2-functor from $\mathcal{RIB}_{\mathbb{K}}$ to $\mathcal{OB}$ sending each ribbon to a string with the same direction is full as noted in \cite{BCNR}. We derive from this fact a description of elementary isotopies and Reidemeister moves first described in \cite{Turaev} for the linear $(2,2)$-category $\mathcal{AOB}$

\begin{proposition} The linear~$(3,2)$-polygraph~$\overline{\mathrm{AOB}}$ is a presentation of the linear~$(2,2)$-category~$\mathcal{AOB}$. \end{proposition}

\begin{proof} All 3-cells of~$\overline{\mathrm{AOB}}$ correspond to a relation verified in the linear~$(2,2)$-category~$\mathcal{AOB}$ by Definition \ref{defAOB}. Moreover, the set of 3-cells of type $i^0$, $r^0$, $r^1$, $r^2$, $r^4$, and $r^5$ contains all elementary isotopies and Reidemeister moves given in \cite{Turaev} (Chapter 1, Section 4), and thus generates the equivalence of dotted oriented Brauer diagram with bubbles. As a consequence, the 3-cells of~$\overline{\mathrm{AOB}}$ are sufficient to find any relation verified in~$\overline{\mathrm{AOB}}$. \end{proof}

Then, we define some particular 2-cells of $\overline{\mathrm{AOB}}_2^\ell$.

\begin{definition} We call a monomial of $\overline{\mathrm{AOB}}$ quasi-reduced if the only 3-cells we can apply to it are of the form:
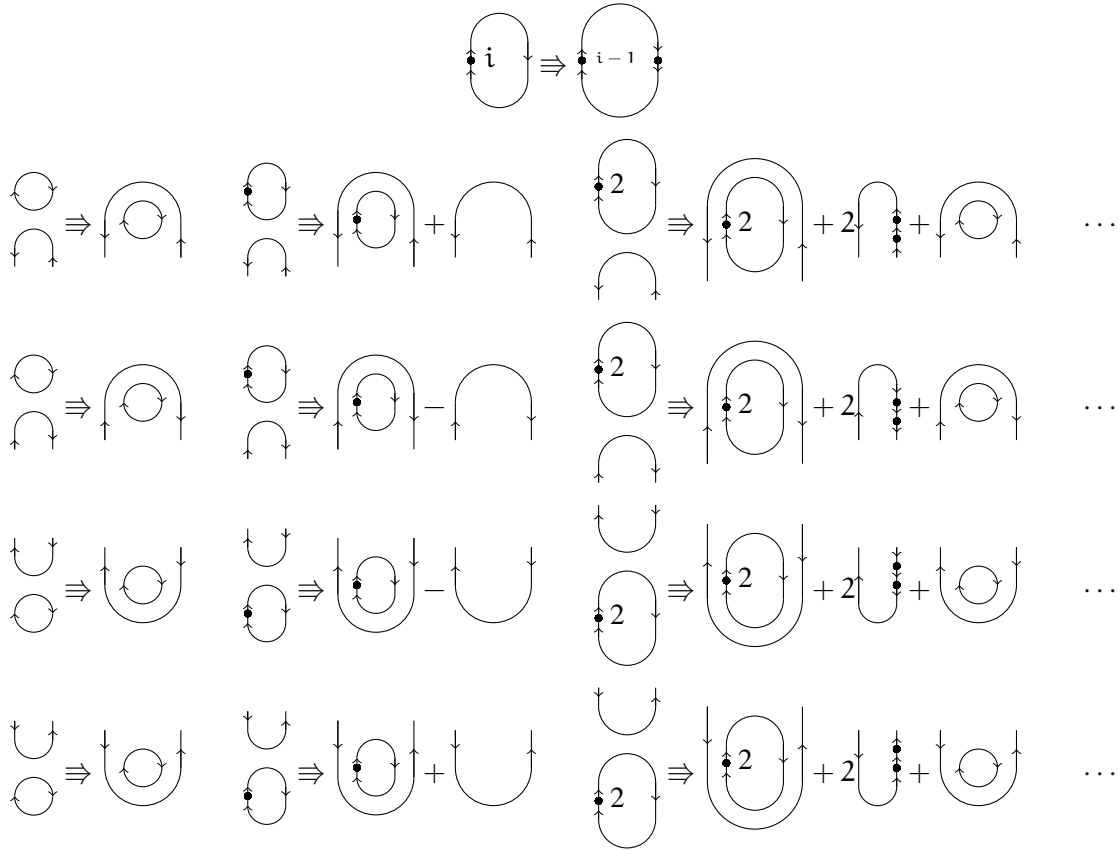
\begin{gather*}
\begin{array}{c}
\tikz[scale=0.5]{
\draw[color=black] (0,0) -- (0,1);
\draw[color=black] (1.5,0) -- (1.5,1);
\draw[color=black] (0,1) arc (180:0:0.75);
\draw[color=black] (0,0) arc (180:360:0.75);
\draw[thick] plot[mark=*] (0,0.5);
\draw[color=black, ->] (0,0) -- (0,0.25);
\draw[color=black, ->] (0,0.5) -- (0,0.8);
\draw[color=black, ->] (1.5,1) -- (1.5,0.5);
\node at (0.5,0.6) {$i$};
} \end{array}
\Rrightarrow
\begin{array}{c}
\tikz[scale=0.5]{
\draw[color=black] (0,0) -- (0,1);
\draw[color=black] (2,0) -- (2,1);
\draw[color=black] (0,1) arc (180:0:1);
\draw[color=black] (0,0) arc (180:360:1);
\draw[thick] plot[mark=*] (0,0.5);
\draw[thick] plot[mark=*] (2,0.5);
\draw[color=black, ->] (0,0) -- (0,0.25);
\draw[color=black, ->] (0,0.5) -- (0,0.8);
\draw[color=black, ->] (2,1) -- (2,0.75);
\draw[color=black, ->] (2,0.5) -- (2,0.2);
\node at (0.9,0.56) {\tiny{$i-1$}};
} \end{array}\\
\begin{array}{c}
\tikz[scale=0.5]{
\draw[color=black] (0,0) -- (0,0.5);
\draw[color=black] (1,0) -- (1,0.5);
\draw[color=black] (1,0.5) arc (0:180:0.5);
\draw[color=black, ->] (0,0.5) -- (0,0.25);
\draw[color=black, ->] (1,0) -- (1,0.25);
\draw[color=black, ->] (0,2) arc (180:0:0.5);
\draw[color=black, ->] (1,2) arc (360:180:0.5);
} \end{array}
\Rrightarrow
\begin{array}{c}
\tikz[scale=0.5]{
\draw[color=black] (0,0) -- (0,1);
\draw[color=black] (2,0) -- (2,1);
\draw[color=black] (2,1) arc (0:180:1);
\draw[color=black, ->] (0,1) -- (0,0.5);
\draw[color=black, ->] (2,0) -- (2,0.5);
\draw[color=black, ->] (0.5,1) arc (180:0:0.5);
\draw[color=black, ->] (1.5,1) arc (360:180:0.5);
} \end{array}
\qquad
\begin{array}{c}
\tikz[scale=0.5]{
\draw[color=black] (0,0) -- (0,0.5);
\draw[color=black] (1,0) -- (1,0.5);
\draw[color=black] (1,0.5) arc (0:180:0.5);
\draw[color=black, ->] (0,0.5) -- (0,0.25);
\draw[color=black, ->] (1,0) -- (1,0.25);
\draw[color=black] (0,2.5) arc (180:0:0.5);
\draw[color=black, ->] (1,2) arc (360:180:0.5);
\draw[color=black, ->] (0,2) -- (0,2.5);
\draw[color=black] (1,2) -- (1,2.5);
\draw[color=black, ->] (1,2.5) -- (1,2.25);
\draw[thick] plot[mark=*] (0,2.25);
} \end{array}
\Rrightarrow
\begin{array}{c}
\tikz[scale=0.5]{
\draw[color=black] (0,0) -- (0,1.5);
\draw[color=black] (2,0) -- (2,1.5);
\draw[color=black] (2,1.5) arc (0:180:1);
\draw[color=black, ->] (0,1.5) -- (0,0.75);
\draw[color=black, ->] (2,0) -- (2,0.75);
\draw[color=black] (0.5,1.5) arc (180:0:0.5);
\draw[color=black, ->] (1.5,1) arc (360:180:0.5);
\draw[color=black, ->] (0.5,1) -- (0.5,1.5);
\draw[color=black] (1.5,1) -- (1.5,1.5);
\draw[color=black, ->] (1.5,1.5) -- (1.5,1.25);
\draw[thick] plot[mark=*] (0.5,1.25);
} \end{array}
+
\begin{array}{c}
\tikz[scale=0.5]{
\draw[color=black] (0,0) -- (0,1);
\draw[color=black] (2,0) -- (2,1);
\draw[color=black] (2,1) arc (0:180:1);
\draw[color=black, ->] (0,1) -- (0,0.5);
\draw[color=black, ->] (2,0) -- (2,0.5);
} \end{array}
\qquad
\begin{array}{c}
\tikz[scale=0.5]{
\draw[color=black] (0,0) -- (0,1);
\draw[color=black] (1.5,0) -- (1.5,1);
\draw[color=black] (0,1) arc (180:0:0.75);
\draw[color=black] (0,0) arc (180:360:0.75);
\draw[thick] plot[mark=*] (0,0.5);
\draw[color=black, ->] (0,0) -- (0,0.25);
\draw[color=black, ->] (0,0.5) -- (0,0.8);
\draw[color=black, ->] (1.5,1) -- (1.5,0.5);
\node at (0.5,0.6) {2};
\draw[color=black] (0,-2.5) -- (0,-2);
\draw[color=black] (1.5,-2.5) -- (1.5,-2);
\draw[color=black, ->] (0,-2) -- (0,-2.25);
\draw[color=black, ->] (1.5,-2.5) -- (1.5,-2.25);
\draw[color=black] (0,-2) arc (180:0:0.75);
} \end{array}
\Rrightarrow
\begin{array}{c}
\tikz[scale=0.5]{
\draw[color=black] (0,0) -- (0,1);
\draw[color=black] (1.5,0) -- (1.5,1);
\draw[color=black] (0,1) arc (180:0:0.75);
\draw[color=black] (0,0) arc (180:360:0.75);
\draw[thick] plot[mark=*] (0,0.5);
\draw[color=black, ->] (0,0) -- (0,0.25);
\draw[color=black, ->] (0,0.5) -- (0,0.8);
\draw[color=black, ->] (1.5,1) -- (1.5,0.5);
\node at (0.5,0.6) {2};
\draw[color=black] (-0.5,-1) -- (-0.5,1);
\draw[color=black] (2,-1) -- (2,1);
\draw[color=black, ->] (-0.5,1) -- (-0.5,0);
\draw[color=black, ->] (2,-1) -- (2,0);
\draw[color=black] (-0.5,1) arc (180:0:1.25);
} \end{array}
+2
\begin{array}{c}
\tikz[scale=0.5]{
\draw[color=black] (1,0) -- (1,1.5);
\draw[color=black] (0,0) -- (0,1.5);
\draw[color=black] (1,1.5) arc (0:180:0.5);
\draw[thick] plot[mark=*] (1,0.5);
\draw[thick] plot[mark=*] (1,1);
\draw[color=black, ->] (1,0) -- (1,0.25);
\draw[color=black, ->] (1,0.5) -- (1,0.8);
\draw[color=black, ->] (1,1) -- (1,1.3);
\draw[color=black, ->] (0,1.5) -- (0,0.75);
} \end{array}
+
\begin{array}{c}
\tikz[scale=0.5]{
\draw[color=black] (0,0) -- (0,1);
\draw[color=black] (2,0) -- (2,1);
\draw[color=black] (2,1) arc (0:180:1);
\draw[color=black, ->] (0,1) -- (0,0.5);
\draw[color=black, ->] (2,0) -- (2,0.5);
\draw[color=black, ->] (0.5,1) arc (180:0:0.5);
\draw[color=black, ->] (1.5,1) arc (360:180:0.5);
} \end{array}
\qquad \cdots  \\
\begin{array}{c}
\tikz[scale=0.5]{
\draw[color=black] (0,0) -- (0,0.5);
\draw[color=black] (1,0) -- (1,0.5);
\draw[color=black] (1,0.5) arc (0:180:0.5);
\draw[color=black, ->] (0,0) -- (0,0.25);
\draw[color=black, ->] (1,0.5) -- (1,0.25);
\draw[color=black, ->] (0,2) arc (180:0:0.5);
\draw[color=black, ->] (1,2) arc (360:180:0.5);
} \end{array}
\Rrightarrow
\begin{array}{c}
\tikz[scale=0.5]{
\draw[color=black] (0,0) -- (0,1);
\draw[color=black] (2,0) -- (2,1);
\draw[color=black] (2,1) arc (0:180:1);
\draw[color=black, ->] (0,0) -- (0,0.5);
\draw[color=black, ->] (2,1) -- (2,0.5);
\draw[color=black, ->] (0.5,1) arc (180:0:0.5);
\draw[color=black, ->] (1.5,1) arc (360:180:0.5);
} \end{array}
\qquad
\begin{array}{c}
\tikz[scale=0.5]{
\draw[color=black] (0,0) -- (0,0.5);
\draw[color=black] (1,0) -- (1,0.5);
\draw[color=black] (1,0.5) arc (0:180:0.5);
\draw[color=black, ->] (0,0) -- (0,0.25);
\draw[color=black, ->] (1,0.5) -- (1,0.25);
\draw[color=black] (0,2.5) arc (180:0:0.5);
\draw[color=black, ->] (1,2) arc (360:180:0.5);
\draw[color=black, ->] (0,2) -- (0,2.5);
\draw[color=black] (1,2) -- (1,2.5);
\draw[color=black, ->] (1,2.5) -- (1,2.25);
\draw[thick] plot[mark=*] (0,2.25);
} \end{array}
\Rrightarrow
\begin{array}{c}
\tikz[scale=0.5]{
\draw[color=black] (0,0) -- (0,1.5);
\draw[color=black] (2,0) -- (2,1.5);
\draw[color=black] (2,1.5) arc (0:180:1);
\draw[color=black, ->] (0,0) -- (0,0.75);
\draw[color=black, ->] (2,1.5) -- (2,0.75);
\draw[color=black] (0.5,1.5) arc (180:0:0.5);
\draw[color=black, ->] (1.5,1) arc (360:180:0.5);
\draw[color=black, ->] (0.5,1) -- (0.5,1.5);
\draw[color=black] (1.5,1) -- (1.5,1.5);
\draw[color=black, ->] (1.5,1.5) -- (1.5,1.25);
\draw[thick] plot[mark=*] (0.5,1.25);
} \end{array}
-
\begin{array}{c}
\tikz[scale=0.5]{
\draw[color=black] (0,0) -- (0,1);
\draw[color=black] (2,0) -- (2,1);
\draw[color=black] (2,1) arc (0:180:1);
\draw[color=black, ->] (0,0) -- (0,0.5);
\draw[color=black, ->] (2,1) -- (2,0.5);
} \end{array}
\qquad
\begin{array}{c}
\tikz[scale=0.5]{
\draw[color=black] (0,0) -- (0,1);
\draw[color=black] (1.5,0) -- (1.5,1);
\draw[color=black] (0,1) arc (180:0:0.75);
\draw[color=black] (0,0) arc (180:360:0.75);
\draw[thick] plot[mark=*] (0,0.5);
\draw[color=black, ->] (0,0) -- (0,0.25);
\draw[color=black, ->] (0,0.5) -- (0,0.8);
\draw[color=black, ->] (1.5,1) -- (1.5,0.5);
\node at (0.5,0.6) {2};
\draw[color=black] (0,-2.5) -- (0,-2);
\draw[color=black] (1.5,-2.5) -- (1.5,-2);
\draw[color=black, ->] (0,-2.5) -- (0,-2.25);
\draw[color=black, ->] (1.5,-2) -- (1.5,-2.25);
\draw[color=black] (0,-2) arc (180:0:0.75);
} \end{array}
\Rrightarrow
\begin{array}{c}
\tikz[scale=0.5]{
\draw[color=black] (0,0) -- (0,1);
\draw[color=black] (1.5,0) -- (1.5,1);
\draw[color=black] (0,1) arc (180:0:0.75);
\draw[color=black] (0,0) arc (180:360:0.75);
\draw[thick] plot[mark=*] (0,0.5);
\draw[color=black, ->] (0,0) -- (0,0.25);
\draw[color=black, ->] (0,0.5) -- (0,0.8);
\draw[color=black, ->] (1.5,1) -- (1.5,0.5);
\node at (0.5,0.6) {2};
\draw[color=black] (-0.5,-1) -- (-0.5,1);
\draw[color=black] (2,-1) -- (2,1);
\draw[color=black, ->] (-0.5,-1) -- (-0.5,0);
\draw[color=black, ->] (2,1) -- (2,0);
\draw[color=black] (-0.5,1) arc (180:0:1.25);
} \end{array}
+2
\begin{array}{c}
\tikz[scale=0.5]{
\draw[color=black] (1,0) -- (1,1.5);
\draw[color=black] (0,0) -- (0,1.5);
\draw[color=black] (1,1.5) arc (0:180:0.5);
\draw[thick] plot[mark=*] (1,0.5);
\draw[thick] plot[mark=*] (1,1);
\draw[color=black, ->] (1,0.5) -- (1,0.2);
\draw[color=black, ->] (1,1) -- (1,0.7);
\draw[color=black, ->] (1,1.5) -- (1,1.25);
\draw[color=black, ->] (0,0) -- (0,0.75);
} \end{array}
+
\begin{array}{c}
\tikz[scale=0.5]{
\draw[color=black] (0,0) -- (0,1);
\draw[color=black] (2,0) -- (2,1);
\draw[color=black] (2,1) arc (0:180:1);
\draw[color=black, ->] (0,0) -- (0,0.5);
\draw[color=black, ->] (2,1) -- (2,0.5);
\draw[color=black, ->] (0.5,1) arc (180:0:0.5);
\draw[color=black, ->] (1.5,1) arc (360:180:0.5);
} \end{array}
\qquad \cdots  \\
\begin{array}{c}
\tikz[scale=0.5]{
\draw[color=black] (0,4) -- (0,3.5);
\draw[color=black] (1,4) -- (1,3.5);
\draw[color=black] (1,3.5) arc (360:180:0.5);
\draw[color=black, ->] (0,3.5) -- (0,3.75);
\draw[color=black, ->] (1,4) -- (1,3.75);
\draw[color=black, ->] (0,2) arc (180:0:0.5);
\draw[color=black, ->] (1,2) arc (360:180:0.5);
} \end{array}
\Rrightarrow
\begin{array}{c}
\tikz[scale=0.5]{
\draw[color=black] (0,2) -- (0,1);
\draw[color=black] (2,2) -- (2,1);
\draw[color=black] (2,1) arc (360:180:1);
\draw[color=black, ->] (0,1) -- (0,1.5);
\draw[color=black, ->] (2,2) -- (2,1.5);
\draw[color=black, ->] (0.5,1) arc (180:0:0.5);
\draw[color=black, ->] (1.5,1) arc (360:180:0.5);
} \end{array}
\qquad
\begin{array}{c}
\tikz[scale=0.5]{
\draw[color=black] (0,4.5) -- (0,4);
\draw[color=black] (1,4.5) -- (1,4);
\draw[color=black] (1,4) arc (360:180:0.5);
\draw[color=black, ->] (0,4) -- (0,4.25);
\draw[color=black, ->] (1,4.5) -- (1,4.25);
\draw[color=black] (0,2.5) arc (180:0:0.5);
\draw[color=black, ->] (1,2) arc (360:180:0.5);
\draw[color=black, ->] (0,2) -- (0,2.5);
\draw[color=black] (1,2) -- (1,2.5);
\draw[color=black, ->] (1,2.5) -- (1,2.25);
\draw[thick] plot[mark=*] (0,2.25);
} \end{array}
\Rrightarrow
\begin{array}{c}
\tikz[scale=0.5]{
\draw[color=black] (0,2.5) -- (0,1);
\draw[color=black] (2,2.5) -- (2,1);
\draw[color=black] (2,1) arc (360:180:1);
\draw[color=black, ->] (0,1) -- (0,1.75);
\draw[color=black, ->] (2,2.5) -- (2,1.75);
\draw[color=black] (0.5,1.5) arc (180:0:0.5);
\draw[color=black, ->] (1.5,1) arc (360:180:0.5);
\draw[color=black, ->] (0.5,1) -- (0.5,1.5);
\draw[color=black] (1.5,1) -- (1.5,1.5);
\draw[color=black, ->] (1.5,1.5) -- (1.5,1.25);
\draw[thick] plot[mark=*] (0.5,1.25);
} \end{array}
-
\begin{array}{c}
\tikz[scale=0.5]{
\draw[color=black] (0,2) -- (0,1);
\draw[color=black] (2,2) -- (2,1);
\draw[color=black] (2,1) arc (360:180:1);
\draw[color=black, ->] (0,1) -- (0,1.5);
\draw[color=black, ->] (2,2) -- (2,1.5);
} \end{array}
\qquad
\begin{array}{c}
\tikz[scale=0.5]{
\draw[color=black] (0,0) -- (0,1);
\draw[color=black] (1.5,0) -- (1.5,1);
\draw[color=black] (0,1) arc (180:0:0.75);
\draw[color=black] (0,0) arc (180:360:0.75);
\draw[thick] plot[mark=*] (0,0.5);
\draw[color=black, ->] (0,0) -- (0,0.25);
\draw[color=black, ->] (0,0.5) -- (0,0.8);
\draw[color=black, ->] (1.5,1) -- (1.5,0.5);
\node at (0.5,0.6) {2};
\draw[color=black] (0,3.5) -- (0,3);
\draw[color=black] (1.5,3.5) -- (1.5,3);
\draw[color=black, ->] (0,3) -- (0,3.25);
\draw[color=black, ->] (1.5,3.5) -- (1.5,3.25);
\draw[color=black] (0,3) arc (180:360:0.75);
} \end{array}
\Rrightarrow
\begin{array}{c}
\tikz[scale=0.5]{
\draw[color=black] (0,0) -- (0,1);
\draw[color=black] (1.5,0) -- (1.5,1);
\draw[color=black] (0,1) arc (180:0:0.75);
\draw[color=black] (0,0) arc (180:360:0.75);
\draw[thick] plot[mark=*] (0,0.5);
\draw[color=black, ->] (0,0) -- (0,0.25);
\draw[color=black, ->] (0,0.5) -- (0,0.8);
\draw[color=black, ->] (1.5,1) -- (1.5,0.5);
\node at (0.5,0.6) {2};
\draw[color=black] (-0.5,2) -- (-0.5,0);
\draw[color=black] (2,2) -- (2,0);
\draw[color=black, ->] (-0.5,0) -- (-0.5,1);
\draw[color=black, ->] (2,2) -- (2,1);
\draw[color=black] (-0.5,0) arc (180:360:1.25);
} \end{array}
+2
\begin{array}{c}
\tikz[scale=0.5]{
\draw[color=black] (1,0) -- (1,-1.5);
\draw[color=black] (0,0) -- (0,-1.5);
\draw[color=black] (1,-1.5) arc (360:180:0.5);
\draw[thick] plot[mark=*] (1,-0.5);
\draw[thick] plot[mark=*] (1,-1);
\draw[color=black, ->] (1,0) -- (1,-0.25);
\draw[color=black, ->] (1,-0.5) -- (1,-0.8);
\draw[color=black, ->] (1,-1) -- (1,-1.3);
\draw[color=black, ->] (0,-1.5) -- (0,-0.75);
} \end{array}
+
\begin{array}{c}
\tikz[scale=0.5]{
\draw[color=black] (0,2) -- (0,1);
\draw[color=black] (2,2) -- (2,1);
\draw[color=black] (2,1) arc (360:180:1);
\draw[color=black, ->] (0,1) -- (0,1.5);
\draw[color=black, ->] (2,2) -- (2,1.5);
\draw[color=black, ->] (0.5,1) arc (180:0:0.5);
\draw[color=black, ->] (1.5,1) arc (360:180:0.5);
} \end{array}
\qquad \cdots  \\
\begin{array}{c}
\tikz[scale=0.5]{
\draw[color=black] (0,4) -- (0,3.5);
\draw[color=black] (1,4) -- (1,3.5);
\draw[color=black] (1,3.5) arc (360:180:0.5);
\draw[color=black, ->] (0,4) -- (0,3.75);
\draw[color=black, ->] (1,3.5) -- (1,3.75);
\draw[color=black, ->] (0,2) arc (180:0:0.5);
\draw[color=black, ->] (1,2) arc (360:180:0.5);
} \end{array}
\Rrightarrow
\begin{array}{c}
\tikz[scale=0.5]{
\draw[color=black] (0,2) -- (0,1);
\draw[color=black] (2,2) -- (2,1);
\draw[color=black] (2,1) arc (360:180:1);
\draw[color=black, ->] (0,2) -- (0,1.5);
\draw[color=black, ->] (2,1) -- (2,1.5);
\draw[color=black, ->] (0.5,1) arc (180:0:0.5);
\draw[color=black, ->] (1.5,1) arc (360:180:0.5);
} \end{array}
\qquad
\begin{array}{c}
\tikz[scale=0.5]{
\draw[color=black] (0,4.5) -- (0,4);
\draw[color=black] (1,4.5) -- (1,4);
\draw[color=black] (1,4) arc (360:180:0.5);
\draw[color=black, ->] (0,4.5) -- (0,4.25);
\draw[color=black, ->] (1,4) -- (1,4.25);
\draw[color=black] (0,2.5) arc (180:0:0.5);
\draw[color=black, ->] (1,2) arc (360:180:0.5);
\draw[color=black, ->] (0,2) -- (0,2.5);
\draw[color=black] (1,2) -- (1,2.5);
\draw[color=black, ->] (1,2.5) -- (1,2.25);
\draw[thick] plot[mark=*] (0,2.25);
} \end{array}
\Rrightarrow
\begin{array}{c}
\tikz[scale=0.5]{
\draw[color=black] (0,2.5) -- (0,1);
\draw[color=black] (2,2.5) -- (2,1);
\draw[color=black] (2,1) arc (360:180:1);
\draw[color=black, ->] (0,2.5) -- (0,1.75);
\draw[color=black, ->] (2,1) -- (2,1.75);
\draw[color=black] (0.5,1.5) arc (180:0:0.5);
\draw[color=black, ->] (1.5,1) arc (360:180:0.5);
\draw[color=black, ->] (0.5,1) -- (0.5,1.5);
\draw[color=black] (1.5,1) -- (1.5,1.5);
\draw[color=black, ->] (1.5,1.5) -- (1.5,1.25);
\draw[thick] plot[mark=*] (0.5,1.25);
} \end{array}
+
\begin{array}{c}
\tikz[scale=0.5]{
\draw[color=black] (0,2) -- (0,1);
\draw[color=black] (2,2) -- (2,1);
\draw[color=black] (2,1) arc (360:180:1);
\draw[color=black, ->] (0,2) -- (0,1.5);
\draw[color=black, ->] (2,1) -- (2,1.5);
} \end{array}
\qquad
\begin{array}{c}
\tikz[scale=0.5]{
\draw[color=black] (0,0) -- (0,1);
\draw[color=black] (1.5,0) -- (1.5,1);
\draw[color=black] (0,1) arc (180:0:0.75);
\draw[color=black] (0,0) arc (180:360:0.75);
\draw[thick] plot[mark=*] (0,0.5);
\draw[color=black, ->] (0,0) -- (0,0.25);
\draw[color=black, ->] (0,0.5) -- (0,0.8);
\draw[color=black, ->] (1.5,1) -- (1.5,0.5);
\node at (0.5,0.6) {2};
\draw[color=black] (0,3.5) -- (0,3);
\draw[color=black] (1.5,3.5) -- (1.5,3);
\draw[color=black, ->] (0,3.5) -- (0,3.25);
\draw[color=black, ->] (1.5,3) -- (1.5,3.25);
\draw[color=black] (0,3) arc (180:360:0.75);
} \end{array}
\Rrightarrow
\begin{array}{c}
\tikz[scale=0.5]{
\draw[color=black] (0,0) -- (0,1);
\draw[color=black] (1.5,0) -- (1.5,1);
\draw[color=black] (0,1) arc (180:0:0.75);
\draw[color=black] (0,0) arc (180:360:0.75);
\draw[thick] plot[mark=*] (0,0.5);
\draw[color=black, ->] (0,0) -- (0,0.25);
\draw[color=black, ->] (0,0.5) -- (0,0.8);
\draw[color=black, ->] (1.5,1) -- (1.5,0.5);
\node at (0.5,0.6) {2};
\draw[color=black] (-0.5,2) -- (-0.5,0);
\draw[color=black] (2,2) -- (2,0);
\draw[color=black, ->] (-0.5,2) -- (-0.5,1);
\draw[color=black, ->] (2,0) -- (2,1);
\draw[color=black] (-0.5,0) arc (180:360:1.25);
} \end{array}
+2
\begin{array}{c}
\tikz[scale=0.5]{
\draw[color=black] (1,0) -- (1,-1.5);
\draw[color=black] (0,0) -- (0,-1.5);
\draw[color=black] (1,-1.5) arc (360:180:0.5);
\draw[thick] plot[mark=*] (1,-0.5);
\draw[thick] plot[mark=*] (1,-1);
\draw[color=black, ->] (1,-0.5) -- (1,-0.2);
\draw[color=black, ->] (1,-1) -- (1,-0.7);
\draw[color=black, ->] (1,-1.5) -- (1,-1.25);
\draw[color=black, ->] (0,0) -- (0,-0.75);
} \end{array}
+
\begin{array}{c}
\tikz[scale=0.5]{
\draw[color=black] (0,2) -- (0,1);
\draw[color=black] (2,2) -- (2,1);
\draw[color=black] (2,1) arc (360:180:1);
\draw[color=black, ->] (0,2) -- (0,1.5);
\draw[color=black, ->] (2,1) -- (2,1.5);
\draw[color=black, ->] (0.5,1) arc (180:0:0.5);
\draw[color=black, ->] (1.5,1) arc (360:180:0.5);
} \end{array}
\qquad \cdots 
\end{gather*}
We call a 2-cell of $\overline{\mathrm{AOB}}_2^\ell$ quasi-reduced if all monomials in its monomial decomposition are quasi-reduced. 
\end{definition}

\begin{lemma}\label{part1} The linear~$(3,2)$-polygraph~$\overline{\mathrm{AOB}}$ is locally confluent. \end{lemma}

\begin{proof} We will prove all critical branchings of~$\overline{\mathrm{AOB}}$ are confluent, which will give local confluence of~$\overline{\mathrm{AOB}}$ by \ref{CP} because its strict canonical rewrite order is well-founded.

We enumerate first the sources of the critical branchings which do not involve dotted bubbles, starting with the overlapping of the isotopy 3-cells and continuing in the order in which the relations are given. The first given 3-cell of $\mathrm{AOB}$ is $i^0_1$. We search for all the cells having a source overlapping with the source of $i^1_0$. Those 3-cells are $i^0_3$, $i^1_2$, $r^2_1$ and a family of sliding 3-cells given in the end denoted by $s_*$. We continue our enumeration by searching for all the overlaps of the source of $i^0_2$ and so on. The final enumeration is:
$$(i^0_1,i^0_3),(i^0_1,i^1_2),(i^0_1,r^2_1),(i^0_1,r^2_3),(i^0_2,i^0_4),(i^0_2,i^1_3),(i^0_2,r^2_2),(i^0_2,r^2_4),(i^0_3,i^1_4),(i^0_3,r^2_5),$$
$$(i^0_3,r^2_7),(i^0_4,i^1_1),(i^0_4,r^2_6),(i^0_4,r^2_8),(i^1_1,r^2_2),(i^1_1,r^2_4),(i^1_3,r^2_6),(i^1_3,r^2_8),(r^0_1,r^0_1),(r^0_1,r^1_1),$$
$$(r^0_1,r^1_2),(r^0_1,r^1_4),(r^0_1,r^2_1),(r^0_1,r^3_1),(r^0_2,r^0_2),(r^0_2,r^1_2),(r^0_2,r^1_4),(r^0_2,r^1_8),(r^0_2,r^2_2),$$
$$(r^0_2,r^2_7),(r^0_2,r^3_3),(r^0_2,r^3_6),(r^0_2,r^4_2),(r^0_2,r^4_3),(r^0_2,o^0_6),(r^0_2,o^0_7),(r^0_3,r^0_3),(r^0_3,r^1_3),$$
$$(r^0_3,r^1_6),(r^0_3,r^1_7),(r^0_3,r^2_3),(r^0_3,r^2_6),(r^0_3,r^3_2),(r^0_3,r^3_7),(r^0_3,r^4_1),(r^0_3,r^4_4),(r^0_3,o^0_5),$$
$$(r^0_3,o^0_8),(r^0_4,r^0_4),(r^0_4,r^1_5),(r^0_4,r^1_6),(r^0_4,r^1_7),(r^0_4,r^2_4),(r^0_4,r^2_5),(r^0_4,r^3_4),(r^0_4,r^3_5),$$
$$(r^0_4,o^0_4),(r^0_4,o^0_7),(r^1_1,r^1_1),(r^1_1,o^0_1),(r^1_1,o^0_2),(r^1_2,r^1_4),(r^1_2,r^3_7),(r^1_2,o^0_5),(r^1_2,o^0_7),$$
$$(r^1_3,r^1_3),(r^1_3,r^3_2),(r^1_3,r^4_1),(r^1_3,o^0_5),(r^1_3,o^0_6),(r^1_3,o^0_7),(r^1_4,r^2_8),(r^1_4,r^2_8),(r^1_4,r^3_8),$$
$$(r^1_4,r^4_2),(r^1_4,o^0_1),(r^1_4,o^0_2),(r^1_4,o^0_6),(r^1_4,o^0_8),(r^1_5,r^1_5),(r^1_5,o^0_3),(r^1_5,o^0_4),(r^1_6,r^1_7),$$
$$(r^1_6,r^2_3),(r^1_6,r^2_5),(r^1_6,r^3_5),(r^1_6,r^4_1),(r^1_6,o^0_5),(r^1_6,o^0_7),(r^1_7,r^2_4),(r^1_7,r^2_6),(r^1_7,r^3_4),$$
$$(r^1_7,r^4_4),(r^1_7,o^0_4),(r^1_7,o^0_5),(r^1_7,o^0_6),(r^1_8,r^1_8),(r^1_8,r^3_3),(r^1_8,r^3_6),(r^1_8,r^4_2),(r^1_8,r^4_3),$$
$$(r^1_8,o^0_3),(r^1_8,o^0_5),(r^1_8,o^0_8),(r^2_1,r^2_8),(r^2_1,r^3_7),(r^2_1,r^3_8),(r^2_1,r^5_1),(r^2_1,o^0_1),(r^2_1,o^0_2),$$
$$(r^2_2,r^3_4),(r^2_2,r^3_8),(r^2_2,r^4_3),(r^2_2,o^0_5),(r^2_3,r^3_1),(r^2_3,r^3_6),(r^2_3,r^4_4),(r^2_3,o^0_6),(r^2_4,r^2_5),$$
$$(r^2_4,r^3_3),(r^2_4,r^3_5),(r^2_4,r^5_2),(r^2_4,o^0_4),(r^2_5,r^3_4),(r^2_5,r^3_6),(r^2_5,r^5_2),(r^2_5,o^0_3),(r^2_5,o^0_4),$$
$$(r^2_6,r^3_8),(r^2_6,r^4_1),(r^2_6,o^0_7),(r^2_7,r^3_3),(r^2_7,r^3_5),(r^2_7,r^4_2),(r^2_7,o^0_6),(r^2_7,o^0_8),(r^2_8,r^3_2),$$
$$(r^2_8,r^5_1),(r^2_8,o^0_2),(r^3_1,r^3_8),(r^3_1,o^0_1),(r^3_1,o^0_2),(r^3_2,r^4_4),(r^3_2,o^0_5),(r^3_2,o^0_6),(r^3_3,r^4_3),$$
$$(r^3_3,o^0_3),(r^3_3,o^0_5),(r^3_4,r^3_5),(r^3_4,o^0_5),(r^3_4,o^0_7),(r^3_5,o^0_3),(r^3_5,o^0_4),(r^3_6,r^4_2),(r^3_6,r^4_2),$$
$$(r^3_6,o^0_7),(r^3_6,o^0_8),(r^3_7,r^4_1),(r^3_7,o^0_1),(r^3_7,o^0_7),(r^3_8,o^0_5),(r^3_8,o^0_7),(r^4_1,r^4_4),(r^4_1,o^0_6),$$
$$(r^4_2,r^4_3),(r^4_2,o^0_5),(r^4_3,o^0_8),(r^4_4,o^0_7),(r^5_1,o^0_1),(r^5_2,o^0_3),$$
$$(i^0_1,s_*),(i^0_2,s_*),(i^0_3,s_*),(i^0_4,s_*),(r^0_1,s_*),(r^0_2,s_*),(r^0_3,s_*),(r^0_4,s_*),$$
$$(r^1_1,s_*),(r^1_2,s_*),(r^1_3,s_*),(r^1_4,s_*),(r^1_5,s_*),(r^1_6,s_*),(r^1_7,s_*),(r^1_8,s_*),$$
$$(r^2_1,s_*),(r^2_2,s_*),(r^2_3,s_*),(r^2_4,s_*),(r^2_5,s_*),(r^2_6,s_*),(r^2_7,s_*),(r^2_8,s_*),$$
$$(r^3_1,s_*),(r^3_2,s_*),(r^3_3,s_*),(r^3_4,s_*),(r^3_5,s_*),(r^3_6,s_*),(r^3_7,s_*),(r^3_8,s_*),$$
$$(r^4_1,s_*),(r^4_2,s_*),(r^4_3,s_*),(r^4_4,s_*),(r^5_1,s_*),(r^5_2,s_*),(o^1_*,s_*),(s_*,s_*).$$
By doing this, we remark we can eliminate multiple critical branchings, similar to others. First, we can ignore the directions of the up and down for each of the following interactions:
\begin{itemize}
\item all critical branchings involving $i^1_k$ and $i^1_l$ for some $1\leqslant k\leqslant 4$ and $1\leqslant l\leqslant 4$,
\item all critical branchings involving $i^1_k$ for some $1\leqslant k\leqslant 4$ and a Reidemeister 3-cell,
\item all critical branchings involving two Reidemeister 3-cells.
\end{itemize}
Indeed, the above 3-cells correspond to the relations verified by the category $\mathcal{OB}$ defined in \cite{BCNR}. Those relations do not depend on the up or down directions. Those relations are also invariant by axial symmetry. Those facts allow us to withdraw multiple critical pairs from our enumeration. In the same way, any critical branchings involving two isotopy 3-cells can be studied up to axial symmetry.

The interactions between an ordering 3-cell and a  Reidemester 3-cell are not invariant by symmetry. But, when a critical branching involving an ordering 3-cell and a Reidemeister 3-cell $r^k_l$ is confluent, all critical branchings involving an ordering 3-cell and a 3-cell of the form $r^k_*$ are confluent. This fact allows us to treat multiple cases of critical branching simultaneously.

We now verify each critical branching is confluent. For each critical branching we are treating, we will use notation \ref{notaCP}. Thus, we will give the source of each critical branching and a target attained once a confluence diagram is given for this critical pair.
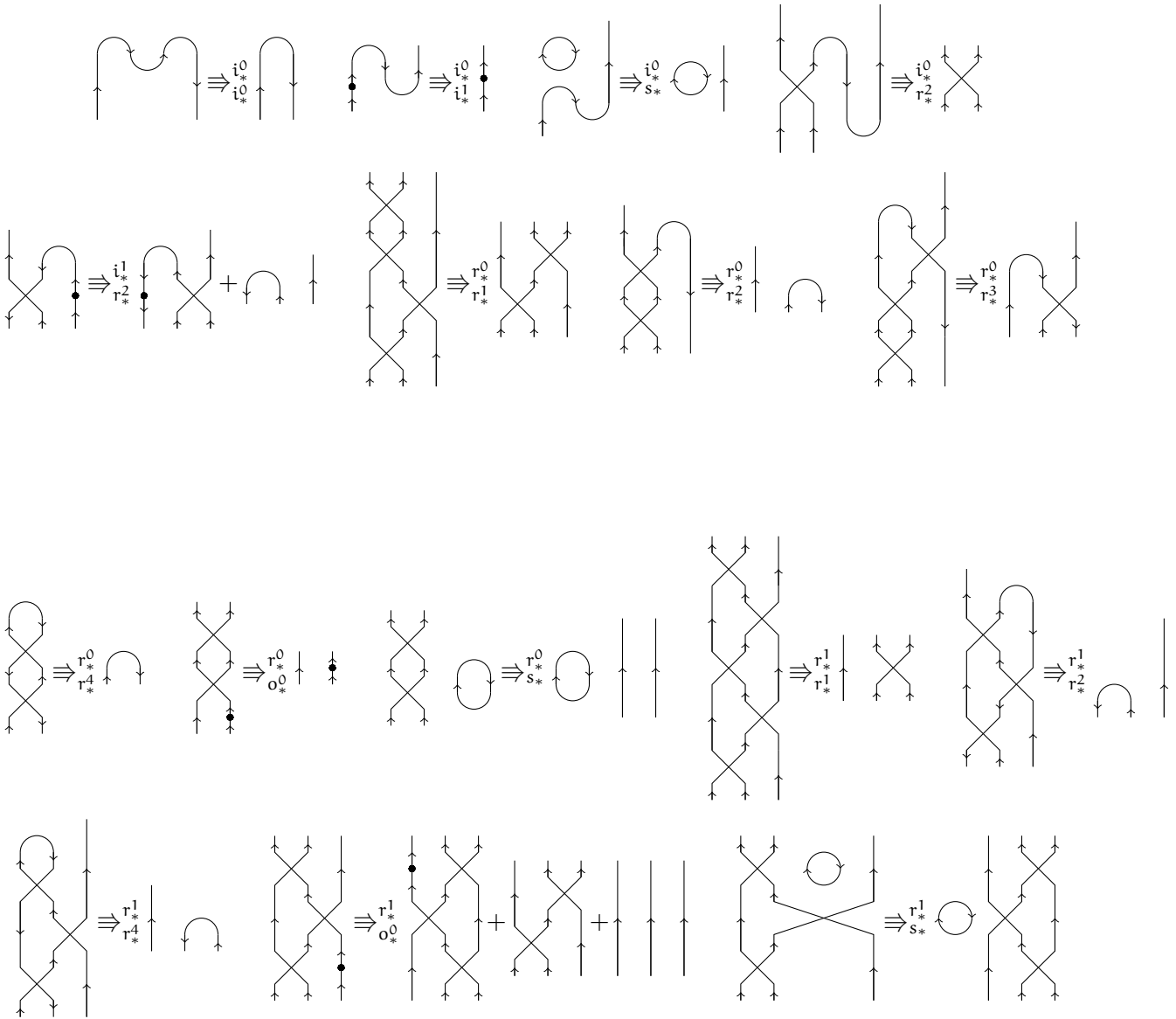
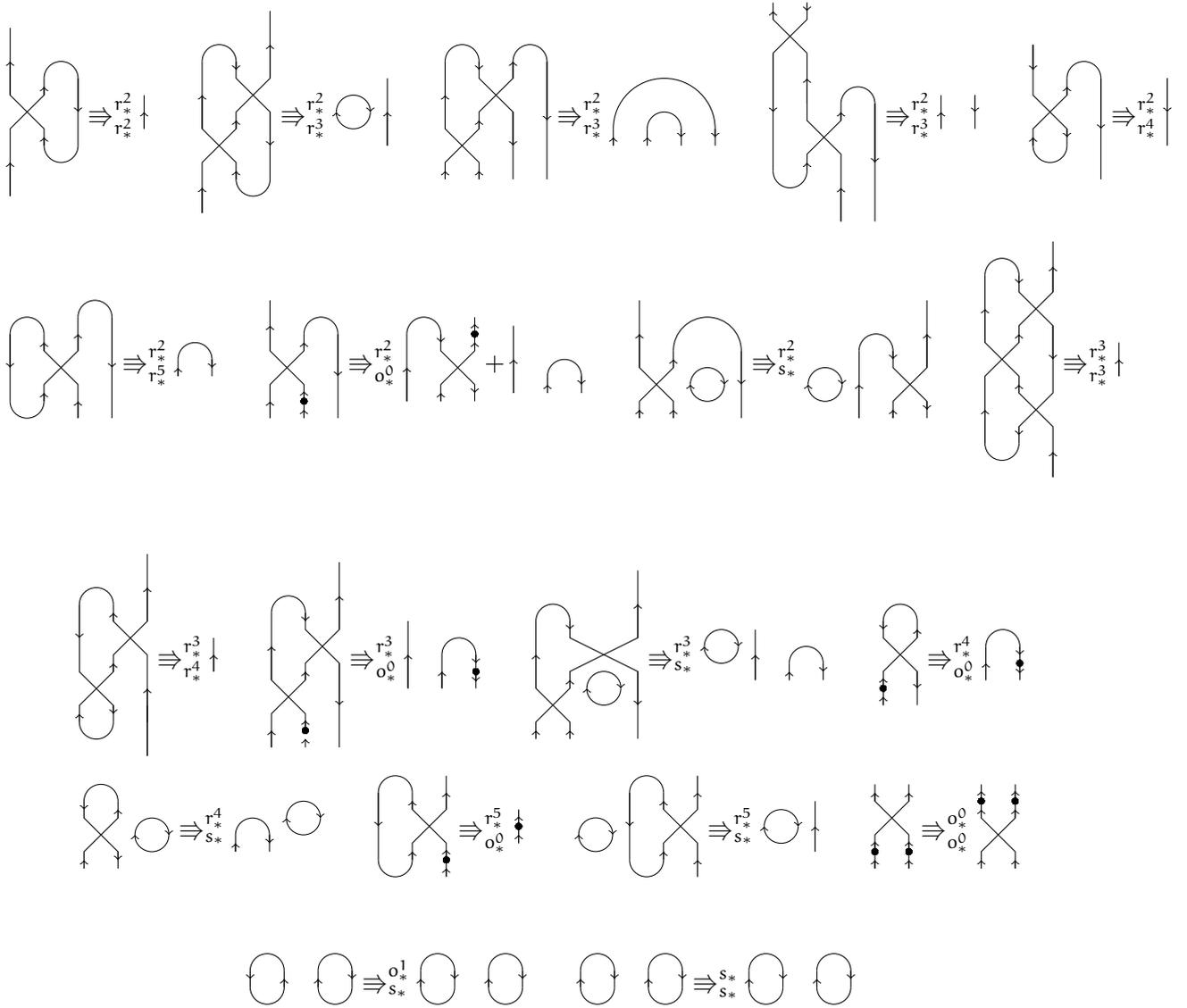
\begin{gather*}

\end{gather*}

The sliding 3-cells can be seen as making a bubble going through an identity and creating some additive residues. For each $n$ of $\mathbb{N}$, the only bubbles created by the 3-cells $s^0_n$ and $s^1_n$ have less than $n$ dots. Thus, treating the confluence of the critical branchings involving no sliding relation other than $s^0_0$ and $s^1_0$ shows by induction on $n$ all critical branchings involving sliding relations are confluent. This deals with the totality of the critical branchings. \end{proof}

\begin{lemma}\label{algo} Each 2-cell of the linear~$(2,2)$-polygraph $\overline{\mathrm{AOB}}_2^\ell$ can be rewritten into a quasi-reduced 2-cell. \end{lemma}

Before proving this claim, we introduce the notion of weight function on a linear~$(2,2)$-category.

\begin{definition} Let $\mathcal{C}$ be a linear~$(2,2)$-category. A weight function on~$\mathcal{C}$ is a function~$\tau$ from $\mathcal{C}_2$ to~$\mathbb{N}$ such that:
\begin{itemize}
\item for each 0-composable monomials $a$ and~$b$ we have~$\tau(a \star_0 b)=\tau(a)+\tau(b)$,
\item for each 1-composable monomials $a$ and~$b$ we have~$\tau(a \star_1 b)=\tau(a)+\tau(b)$,
\item for each 2-cell~$x$, we have~$\tau(x)=max\{\tau(a)|\text{$a$ appears in the monomial decomposition of x}\}$.
\end{itemize}
\end{definition}

\begin{proof} To prove that each 2-cell of $\overline{\mathrm{AOB}}_2^\ell$ rewrites into a normally dotted oriented Brauer diagram, we give first a weight function~$\tau$ on~$\overline{\mathrm{AOB}}_2^\ell$ defined by:
\begin{gather*}
\tau

=0
\end{gather*}

Then, we give the following algorithm taking a monomial~$m$ as an entry:
\begin{itemize}
\item While $m$ can be rewritten into a 2-cell~$m'$ such that~$\tau (m')<\tau (m)$ do:
\begin{itemize}
\item assign $m$ to~$m'$
\end{itemize}
\item While $m$ can be rewritten into a 2-cell~$m'$  without applying a 3-cell of the form:
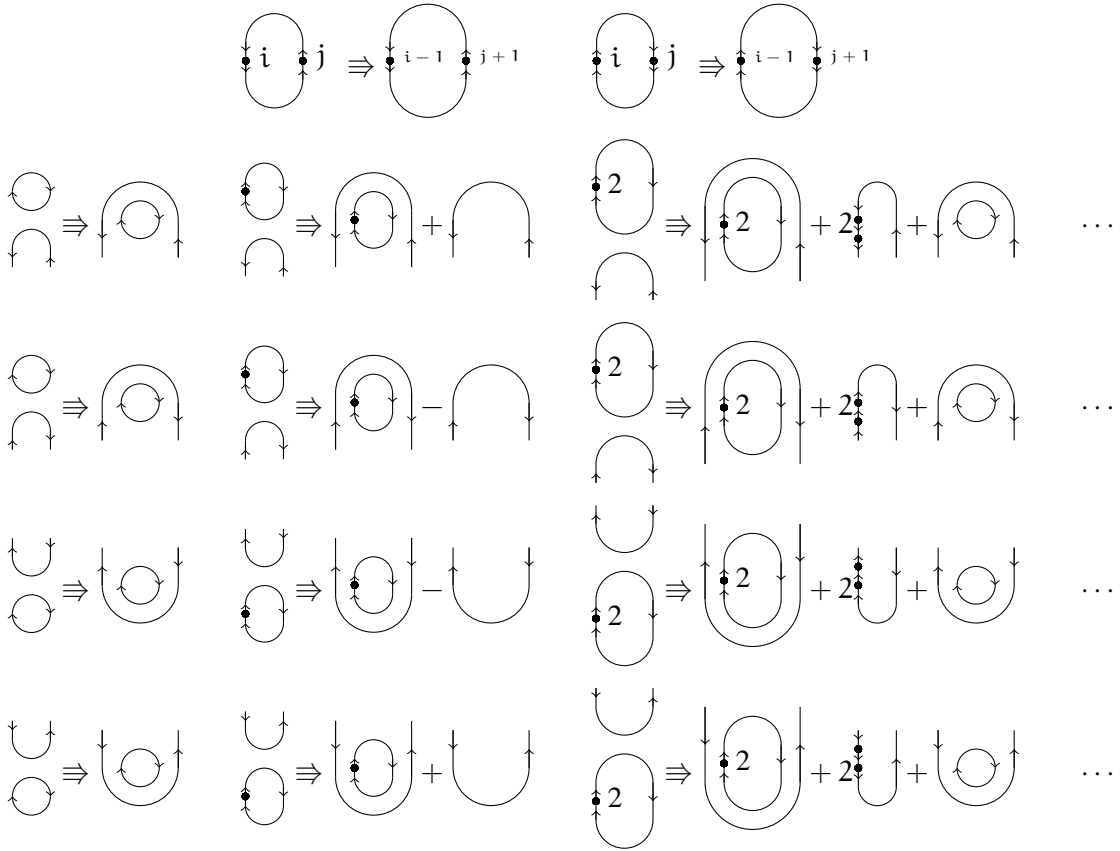
\begin{gather*}
%
\qquad \cdots 
\end{gather*}
do:
\begin{itemize}
\item assign $m$ to~$m'$
\end{itemize}
\end{itemize}
This algorithm terminates because $\Sigma$ is terminating without the sliding 3-cells and the isotopy 3-cells that do not decrease the weight. This results in rewriting any 2-cell into a linear combination of normally ordered Brauer diagrams. For each  equivalence class of dotted oriented Brauer diagram with bubbles, exactly one representative of this class can appear in a 2-cell attained by the given algorithm because there is no isotopy or Reidemeister 3-cell to apply anymore. Those are the quasi-reduced monomials. \end{proof}

\begin{lemma}\label{part2} The linear~$(3,2)$-polygraph~$\overline{\mathrm{AOB}}$ is decreasing. \end{lemma}

\begin{proof} We define the set~$\mathcal{P}=\{P_n|n \in \mathbb{N} \}$ where for each rewriting step~$\alpha$ of~$\overline{\mathrm{AOB}}$ and each $n$ in $\mathbb{N}$, we have $\alpha \in P_n$ if and only if the rewriting sequence of minimal length from $t_2(\alpha)$ to a quasi-reduced 2-cell is of length~$n$. Because of the fact that every 2-cell can be rewritten into a quasi-reduced one \ref{algo},~$\mathcal{P}$ is a partition of the set of rewriting steps of~$\overline{\mathrm{AOB}}$. We give on this partition the order~$\prec$ such that~$P_n \prec P_m$ if and only if~$n<m$. This well-founded order respects all the properties of \ref{decr} for the critical branchings of $\overline{\mathrm{AOB}}$. The decreasingness critical branchings is explicited in the appendix. Then, the order~$\prec$ meets the properties of \ref{decr}, ending the proof by \ref{VO}. \end{proof}

\begin{theorem}\label{thconf} The linear~$(3,2)$-polygraph~$\overline{\mathrm{AOB}}$ is a confluent presentation of $\mathcal{AOB}$. \end{theorem}

\begin{proof} The linear~$(3,2)$-polygraph~$\overline{\mathrm{AOB}}$ is locally confluent by lemma \ref{part1}. Lemma \ref{part2} allows us to use \ref{VO} and conclude the proof. \end{proof}

This theorem implies the following result from \cite{BCNR}:
\begin{corollary}\label{Bru} Let $a$ and~$b$ be two 1-cells of the linear~$(2,2)$-category~$\mathcal{AOB}$. Then, the vector space $\mathcal{AOB}_2(a,b)$ has basis given by equivalence classes of normally ordered dotted oriented Brauer diagrams with bubbles with source $a$ and target $b$. \end{corollary}


\begin{proof} We need to give a confluent left-monomial linear~$(3,2)$-polygraph presenting~the linear~$(2,2)$-category $\mathcal{AOB}$ such that:
\begin{itemize}
\item There is a bijection between the 2-cells of the free linear~$(2,2)$-category and dotted oriented Brauer diagrams with bubbles.
\item A normally ordered dotted oriented Brauer diagram does not rewrite into a linear combination of non-equivalent others.
\item Each 2-cell that is not in normal form rewrites into a linear combination of normally ordered dotted oriented Brauer diagrams.
\end{itemize}

The linear~$(3,2)$-polygraph~$\overline{\mathrm{AOB}}$ from \ref{thconf} verifies those properties. To prove that a normally ordered Brauer diagram does not rewrite into a linear combination of non-equivalent others, it is sufficient to show that a quasi-reduced 2-cell does not rewrite into another quasi-reduced 2-cell. The only way to rewrite a quasi-reduced monomial is to apply sliding 3-cells or isototpy relations that do not decrease the weight. So, the only way to rewrite a quasi-reduced monomial into a linear combination of others is to use rewriting paths of the forms:

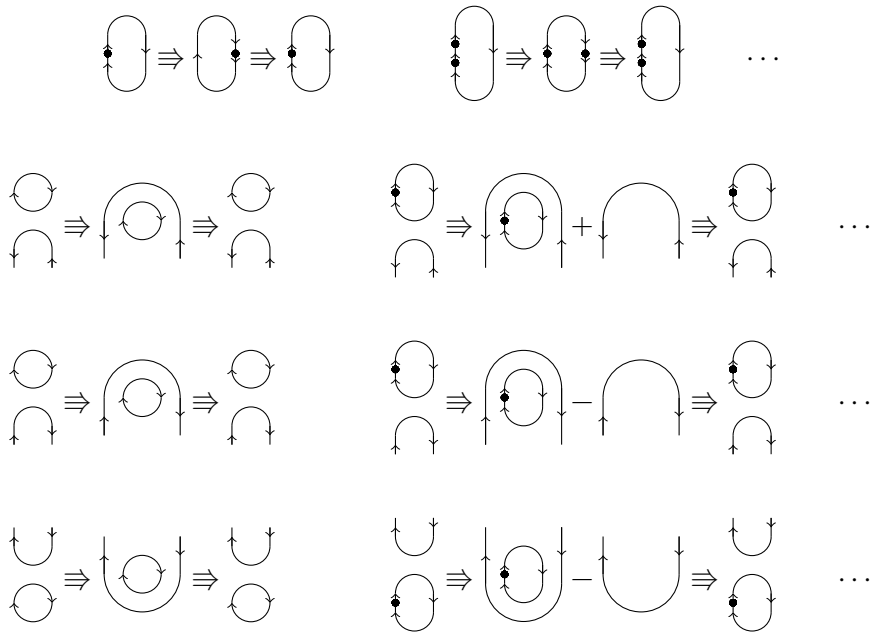
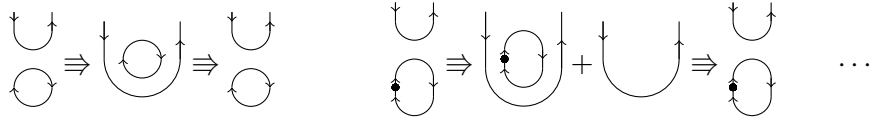
\begin{gather*}
\begin{array}{c}
\tikz[scale=0.5]{
\draw[color=black] (0,0) -- (0,1);
\draw[color=black] (1,0) -- (1,1);
\draw[color=black] (0,1) arc (180:0:0.5);
\draw[color=black] (0,0) arc (180:360:0.5);
\draw[color=black, ->] (0,0) -- (0,0.25);
\draw[color=black, ->] (0,0.5) -- (0,0.75);
\draw[color=black, ->] (1,1) -- (1,0.5);
\draw[thick] plot[mark=*] (0,0.5);
} \end{array}
\Rrightarrow
\begin{array}{c}
\tikz[scale=0.5]{
\draw[color=black] (0,0) -- (0,1);
\draw[color=black] (1,0) -- (1,1);
\draw[color=black] (0,1) arc (180:0:0.5);
\draw[color=black] (0,0) arc (180:360:0.5);
\draw[color=black, ->] (0,0) -- (0,0.5);
\draw[color=black, ->] (1,0.5) -- (1,0.25);
\draw[color=black, ->] (1,1) -- (1,0.75);
\draw[thick] plot[mark=*] (1,0.5);
} \end{array}
\Rrightarrow
\begin{array}{c}
\tikz[scale=0.5]{
\draw[color=black] (0,0) -- (0,1);
\draw[color=black] (1,0) -- (1,1);
\draw[color=black] (0,1) arc (180:0:0.5);
\draw[color=black] (0,0) arc (180:360:0.5);
\draw[color=black, ->] (0,0) -- (0,0.25);
\draw[color=black, ->] (0,0.5) -- (0,0.75);
\draw[color=black, ->] (1,1) -- (1,0.5);
\draw[thick] plot[mark=*] (0,0.5);
} \end{array}
\qquad
\qquad
\begin{array}{c}
\tikz[scale=0.5]{
\draw[color=black] (0,0) -- (0,1.5);
\draw[color=black] (1,0) -- (1,1.5);
\draw[color=black] (0,1.5) arc (180:0:0.5);
\draw[color=black] (0,0) arc (180:360:0.5);
\draw[color=black, ->] (0,0) -- (0,0.25);
\draw[color=black, ->] (0,0.5) -- (0,0.75);
\draw[color=black, ->] (0,1) -- (0,1.25);
\draw[color=black, ->] (1,1.5) -- (1,0.75);
\draw[thick] plot[mark=*] (0,0.5);
\draw[thick] plot[mark=*] (0,1);
} \end{array}
\Rrightarrow
\begin{array}{c}
\tikz[scale=0.5]{
\draw[color=black] (0,0) -- (0,1);
\draw[color=black] (1,0) -- (1,1);
\draw[color=black] (0,1) arc (180:0:0.5);
\draw[color=black] (0,0) arc (180:360:0.5);
\draw[color=black, ->] (0,0) -- (0,0.25);
\draw[color=black, ->] (0,0.5) -- (0,0.75);
\draw[color=black, ->] (1,0.5) -- (1,0.25);
\draw[color=black, ->] (1,1) -- (1,0.75);
\draw[thick] plot[mark=*] (1,0.5);
\draw[thick] plot[mark=*] (0,0.5);
} \end{array}
\Rrightarrow
\begin{array}{c}
\tikz[scale=0.5]{
\draw[color=black] (0,0) -- (0,1.5);
\draw[color=black] (1,0) -- (1,1.5);
\draw[color=black] (0,1.5) arc (180:0:0.5);
\draw[color=black] (0,0) arc (180:360:0.5);
\draw[color=black, ->] (0,0) -- (0,0.25);
\draw[color=black, ->] (0,0.5) -- (0,0.75);
\draw[color=black, ->] (0,1) -- (0,1.25);
\draw[color=black, ->] (1,1.5) -- (1,0.75);
\draw[thick] plot[mark=*] (0,0.5);
\draw[thick] plot[mark=*] (0,1);
} \end{array}
\qquad
\cdots  
\end{gather*}

\begin{gather*}
\begin{array}{c}
\tikz[scale=0.5]{
\draw[color=black] (0,0) -- (0,0.5);
\draw[color=black] (1,0) -- (1,0.5);
\draw[color=black] (1,0.5) arc (0:180:0.5);
\draw[color=black, ->] (0,0.5) -- (0,0.25);
\draw[color=black, ->] (1,0) -- (1,0.25);
\draw[color=black, ->] (0,2) arc (180:0:0.5);
\draw[color=black, ->] (1,2) arc (360:180:0.5);
} \end{array}
\Rrightarrow
\begin{array}{c}
\tikz[scale=0.5]{
\draw[color=black] (0,0) -- (0,1);
\draw[color=black] (2,0) -- (2,1);
\draw[color=black] (2,1) arc (0:180:1);
\draw[color=black, ->] (0,1) -- (0,0.5);
\draw[color=black, ->] (2,0) -- (2,0.5);
\draw[color=black, ->] (0.5,1) arc (180:0:0.5);
\draw[color=black, ->] (1.5,1) arc (360:180:0.5);
} \end{array}
\Rrightarrow
\begin{array}{c}
\tikz[scale=0.5]{
\draw[color=black] (0,0) -- (0,0.5);
\draw[color=black] (1,0) -- (1,0.5);
\draw[color=black] (1,0.5) arc (0:180:0.5);
\draw[color=black, ->] (0,0.5) -- (0,0.25);
\draw[color=black, ->] (1,0) -- (1,0.25);
\draw[color=black, ->] (0,2) arc (180:0:0.5);
\draw[color=black, ->] (1,2) arc (360:180:0.5);
} \end{array}
\qquad
\qquad
\begin{array}{c}
\tikz[scale=0.5]{
\draw[color=black] (0,0) -- (0,0.5);
\draw[color=black] (1,0) -- (1,0.5);
\draw[color=black] (1,0.5) arc (0:180:0.5);
\draw[color=black, ->] (0,0.5) -- (0,0.25);
\draw[color=black, ->] (1,0) -- (1,0.25);
\draw[color=black] (0,2.5) arc (180:0:0.5);
\draw[color=black, ->] (1,2) arc (360:180:0.5);
\draw[color=black, ->] (0,2) -- (0,2.5);
\draw[color=black] (1,2) -- (1,2.5);
\draw[color=black, ->] (1,2.5) -- (1,2.25);
\draw[thick] plot[mark=*] (0,2.25);
} \end{array}
\Rrightarrow
\begin{array}{c}
\tikz[scale=0.5]{
\draw[color=black] (0,0) -- (0,1.5);
\draw[color=black] (2,0) -- (2,1.5);
\draw[color=black] (2,1.5) arc (0:180:1);
\draw[color=black, ->] (0,1.5) -- (0,0.75);
\draw[color=black, ->] (2,0) -- (2,0.75);
\draw[color=black] (0.5,1.5) arc (180:0:0.5);
\draw[color=black, ->] (1.5,1) arc (360:180:0.5);
\draw[color=black, ->] (0.5,1) -- (0.5,1.5);
\draw[color=black] (1.5,1) -- (1.5,1.5);
\draw[color=black, ->] (1.5,1.5) -- (1.5,1.25);
\draw[thick] plot[mark=*] (0.5,1.25);
} \end{array}
+
\begin{array}{c}
\tikz[scale=0.5]{
\draw[color=black] (0,0) -- (0,1);
\draw[color=black] (2,0) -- (2,1);
\draw[color=black] (2,1) arc (0:180:1);
\draw[color=black, ->] (0,1) -- (0,0.5);
\draw[color=black, ->] (2,0) -- (2,0.5);
} \end{array}
\Rrightarrow
\begin{array}{c}
\tikz[scale=0.5]{
\draw[color=black] (0,0) -- (0,0.5);
\draw[color=black] (1,0) -- (1,0.5);
\draw[color=black] (1,0.5) arc (0:180:0.5);
\draw[color=black, ->] (0,0.5) -- (0,0.25);
\draw[color=black, ->] (1,0) -- (1,0.25);
\draw[color=black] (0,2.5) arc (180:0:0.5);
\draw[color=black, ->] (1,2) arc (360:180:0.5);
\draw[color=black, ->] (0,2) -- (0,2.5);
\draw[color=black] (1,2) -- (1,2.5);
\draw[color=black, ->] (1,2.5) -- (1,2.25);
\draw[thick] plot[mark=*] (0,2.25);
} \end{array}
\qquad \cdots 
\end{gather*}

\begin{gather*}
\begin{array}{c}
\tikz[scale=0.5]{
\draw[color=black] (0,0) -- (0,0.5);
\draw[color=black] (1,0) -- (1,0.5);
\draw[color=black] (1,0.5) arc (0:180:0.5);
\draw[color=black, ->] (0,0) -- (0,0.25);
\draw[color=black, ->] (1,0.5) -- (1,0.25);
\draw[color=black, ->] (0,2) arc (180:0:0.5);
\draw[color=black, ->] (1,2) arc (360:180:0.5);
} \end{array}
\Rrightarrow
\begin{array}{c}
\tikz[scale=0.5]{
\draw[color=black] (0,0) -- (0,1);
\draw[color=black] (2,0) -- (2,1);
\draw[color=black] (2,1) arc (0:180:1);
\draw[color=black, ->] (0,0) -- (0,0.5);
\draw[color=black, ->] (2,1) -- (2,0.5);
\draw[color=black, ->] (0.5,1) arc (180:0:0.5);
\draw[color=black, ->] (1.5,1) arc (360:180:0.5);
} \end{array}
\Rrightarrow
\begin{array}{c}
\tikz[scale=0.5]{
\draw[color=black] (0,0) -- (0,0.5);
\draw[color=black] (1,0) -- (1,0.5);
\draw[color=black] (1,0.5) arc (0:180:0.5);
\draw[color=black, ->] (0,0) -- (0,0.25);
\draw[color=black, ->] (1,0.5) -- (1,0.25);
\draw[color=black, ->] (0,2) arc (180:0:0.5);
\draw[color=black, ->] (1,2) arc (360:180:0.5);
} \end{array}
\qquad
\qquad
\begin{array}{c}
\tikz[scale=0.5]{
\draw[color=black] (0,0) -- (0,0.5);
\draw[color=black] (1,0) -- (1,0.5);
\draw[color=black] (1,0.5) arc (0:180:0.5);
\draw[color=black, ->] (0,0) -- (0,0.25);
\draw[color=black, ->] (1,0.5) -- (1,0.25);
\draw[color=black] (0,2.5) arc (180:0:0.5);
\draw[color=black, ->] (1,2) arc (360:180:0.5);
\draw[color=black, ->] (0,2) -- (0,2.5);
\draw[color=black] (1,2) -- (1,2.5);
\draw[color=black, ->] (1,2.5) -- (1,2.25);
\draw[thick] plot[mark=*] (0,2.25);
} \end{array}
\Rrightarrow
\begin{array}{c}
\tikz[scale=0.5]{
\draw[color=black] (0,0) -- (0,1.5);
\draw[color=black] (2,0) -- (2,1.5);
\draw[color=black] (2,1.5) arc (0:180:1);
\draw[color=black, ->] (0,0) -- (0,0.75);
\draw[color=black, ->] (2,1.5) -- (2,0.75);
\draw[color=black] (0.5,1.5) arc (180:0:0.5);
\draw[color=black, ->] (1.5,1) arc (360:180:0.5);
\draw[color=black, ->] (0.5,1) -- (0.5,1.5);
\draw[color=black] (1.5,1) -- (1.5,1.5);
\draw[color=black, ->] (1.5,1.5) -- (1.5,1.25);
\draw[thick] plot[mark=*] (0.5,1.25);
} \end{array}
-
\begin{array}{c}
\tikz[scale=0.5]{
\draw[color=black] (0,0) -- (0,1);
\draw[color=black] (2,0) -- (2,1);
\draw[color=black] (2,1) arc (0:180:1);
\draw[color=black, ->] (0,0) -- (0,0.5);
\draw[color=black, ->] (2,1) -- (2,0.5);
} \end{array}
\Rrightarrow
\begin{array}{c}
\tikz[scale=0.5]{
\draw[color=black] (0,0) -- (0,0.5);
\draw[color=black] (1,0) -- (1,0.5);
\draw[color=black] (1,0.5) arc (0:180:0.5);
\draw[color=black, ->] (0,0) -- (0,0.25);
\draw[color=black, ->] (1,0.5) -- (1,0.25);
\draw[color=black] (0,2.5) arc (180:0:0.5);
\draw[color=black, ->] (1,2) arc (360:180:0.5);
\draw[color=black, ->] (0,2) -- (0,2.5);
\draw[color=black] (1,2) -- (1,2.5);
\draw[color=black, ->] (1,2.5) -- (1,2.25);
\draw[thick] plot[mark=*] (0,2.25);
} \end{array}
\qquad \cdots 
\end{gather*}

\begin{gather*}
\begin{array}{c}
\tikz[scale=0.5]{
\draw[color=black] (0,4) -- (0,3.5);
\draw[color=black] (1,4) -- (1,3.5);
\draw[color=black] (1,3.5) arc (360:180:0.5);
\draw[color=black, ->] (0,3.5) -- (0,3.75);
\draw[color=black, ->] (1,4) -- (1,3.75);
\draw[color=black, ->] (0,2) arc (180:0:0.5);
\draw[color=black, ->] (1,2) arc (360:180:0.5);
} \end{array}
\Rrightarrow
\begin{array}{c}
\tikz[scale=0.5]{
\draw[color=black] (0,2) -- (0,1);
\draw[color=black] (2,2) -- (2,1);
\draw[color=black] (2,1) arc (360:180:1);
\draw[color=black, ->] (0,1) -- (0,1.5);
\draw[color=black, ->] (2,2) -- (2,1.5);
\draw[color=black, ->] (0.5,1) arc (180:0:0.5);
\draw[color=black, ->] (1.5,1) arc (360:180:0.5);
} \end{array}
\Rrightarrow
\begin{array}{c}
\tikz[scale=0.5]{
\draw[color=black] (0,4) -- (0,3.5);
\draw[color=black] (1,4) -- (1,3.5);
\draw[color=black] (1,3.5) arc (360:180:0.5);
\draw[color=black, ->] (0,3.5) -- (0,3.75);
\draw[color=black, ->] (1,4) -- (1,3.75);
\draw[color=black, ->] (0,2) arc (180:0:0.5);
\draw[color=black, ->] (1,2) arc (360:180:0.5);
} \end{array}
\qquad
\qquad
\begin{array}{c}
\tikz[scale=0.5]{
\draw[color=black] (0,4.5) -- (0,4);
\draw[color=black] (1,4.5) -- (1,4);
\draw[color=black] (1,4) arc (360:180:0.5);
\draw[color=black, ->] (0,4) -- (0,4.25);
\draw[color=black, ->] (1,4.5) -- (1,4.25);
\draw[color=black] (0,2.5) arc (180:0:0.5);
\draw[color=black, ->] (1,2) arc (360:180:0.5);
\draw[color=black, ->] (0,2) -- (0,2.5);
\draw[color=black] (1,2) -- (1,2.5);
\draw[color=black, ->] (1,2.5) -- (1,2.25);
\draw[thick] plot[mark=*] (0,2.25);
} \end{array}
\Rrightarrow
\begin{array}{c}
\tikz[scale=0.5]{
\draw[color=black] (0,2.5) -- (0,1);
\draw[color=black] (2,2.5) -- (2,1);
\draw[color=black] (2,1) arc (360:180:1);
\draw[color=black, ->] (0,1) -- (0,1.75);
\draw[color=black, ->] (2,2.5) -- (2,1.75);
\draw[color=black] (0.5,1.5) arc (180:0:0.5);
\draw[color=black, ->] (1.5,1) arc (360:180:0.5);
\draw[color=black, ->] (0.5,1) -- (0.5,1.5);
\draw[color=black] (1.5,1) -- (1.5,1.5);
\draw[color=black, ->] (1.5,1.5) -- (1.5,1.25);
\draw[thick] plot[mark=*] (0.5,1.25);
} \end{array}
-
\begin{array}{c}
\tikz[scale=0.5]{
\draw[color=black] (0,2) -- (0,1);
\draw[color=black] (2,2) -- (2,1);
\draw[color=black] (2,1) arc (360:180:1);
\draw[color=black, ->] (0,1) -- (0,1.5);
\draw[color=black, ->] (2,2) -- (2,1.5);
} \end{array}
\Rrightarrow
\begin{array}{c}
\tikz[scale=0.5]{
\draw[color=black] (0,4.5) -- (0,4);
\draw[color=black] (1,4.5) -- (1,4);
\draw[color=black] (1,4) arc (360:180:0.5);
\draw[color=black, ->] (0,4) -- (0,4.25);
\draw[color=black, ->] (1,4.5) -- (1,4.25);
\draw[color=black] (0,2.5) arc (180:0:0.5);
\draw[color=black, ->] (1,2) arc (360:180:0.5);
\draw[color=black, ->] (0,2) -- (0,2.5);
\draw[color=black] (1,2) -- (1,2.5);
\draw[color=black, ->] (1,2.5) -- (1,2.25);
\draw[thick] plot[mark=*] (0,2.25);
} \end{array}
\qquad \cdots 
\end{gather*}

\begin{gather*}
\begin{array}{c}
\tikz[scale=0.5]{
\draw[color=black] (0,4) -- (0,3.5);
\draw[color=black] (1,4) -- (1,3.5);
\draw[color=black] (1,3.5) arc (360:180:0.5);
\draw[color=black, ->] (0,4) -- (0,3.75);
\draw[color=black, ->] (1,3.5) -- (1,3.75);
\draw[color=black, ->] (0,2) arc (180:0:0.5);
\draw[color=black, ->] (1,2) arc (360:180:0.5);
} \end{array}
\Rrightarrow
\begin{array}{c}
\tikz[scale=0.5]{
\draw[color=black] (0,2) -- (0,1);
\draw[color=black] (2,2) -- (2,1);
\draw[color=black] (2,1) arc (360:180:1);
\draw[color=black, ->] (0,2) -- (0,1.5);
\draw[color=black, ->] (2,1) -- (2,1.5);
\draw[color=black, ->] (0.5,1) arc (180:0:0.5);
\draw[color=black, ->] (1.5,1) arc (360:180:0.5);
} \end{array}
\Rrightarrow
\begin{array}{c}
\tikz[scale=0.5]{
\draw[color=black] (0,4) -- (0,3.5);
\draw[color=black] (1,4) -- (1,3.5);
\draw[color=black] (1,3.5) arc (360:180:0.5);
\draw[color=black, ->] (0,4) -- (0,3.75);
\draw[color=black, ->] (1,3.5) -- (1,3.75);
\draw[color=black, ->] (0,2) arc (180:0:0.5);
\draw[color=black, ->] (1,2) arc (360:180:0.5);
} \end{array}
\qquad
\qquad
\begin{array}{c}
\tikz[scale=0.5]{
\draw[color=black] (0,4.5) -- (0,4);
\draw[color=black] (1,4.5) -- (1,4);
\draw[color=black] (1,4) arc (360:180:0.5);
\draw[color=black, ->] (0,4.5) -- (0,4.25);
\draw[color=black, ->] (1,4) -- (1,4.25);
\draw[color=black] (0,2.5) arc (180:0:0.5);
\draw[color=black, ->] (1,2) arc (360:180:0.5);
\draw[color=black, ->] (0,2) -- (0,2.5);
\draw[color=black] (1,2) -- (1,2.5);
\draw[color=black, ->] (1,2.5) -- (1,2.25);
\draw[thick] plot[mark=*] (0,2.25);
} \end{array}
\Rrightarrow
\begin{array}{c}
\tikz[scale=0.5]{
\draw[color=black] (0,2.5) -- (0,1);
\draw[color=black] (2,2.5) -- (2,1);
\draw[color=black] (2,1) arc (360:180:1);
\draw[color=black, ->] (0,2.5) -- (0,1.75);
\draw[color=black, ->] (2,1) -- (2,1.75);
\draw[color=black] (0.5,1.5) arc (180:0:0.5);
\draw[color=black, ->] (1.5,1) arc (360:180:0.5);
\draw[color=black, ->] (0.5,1) -- (0.5,1.5);
\draw[color=black] (1.5,1) -- (1.5,1.5);
\draw[color=black, ->] (1.5,1.5) -- (1.5,1.25);
\draw[thick] plot[mark=*] (0.5,1.25);
} \end{array}
+
\begin{array}{c}
\tikz[scale=0.5]{
\draw[color=black] (0,2) -- (0,1);
\draw[color=black] (2,2) -- (2,1);
\draw[color=black] (2,1) arc (360:180:1);
\draw[color=black, ->] (0,2) -- (0,1.5);
\draw[color=black, ->] (2,1) -- (2,1.5);
} \end{array}
\Rrightarrow
\begin{array}{c}
\tikz[scale=0.5]{
\draw[color=black] (0,4.5) -- (0,4);
\draw[color=black] (1,4.5) -- (1,4);
\draw[color=black] (1,4) arc (360:180:0.5);
\draw[color=black, ->] (0,4.5) -- (0,4.25);
\draw[color=black, ->] (1,4) -- (1,4.25);
\draw[color=black] (0,2.5) arc (180:0:0.5);
\draw[color=black, ->] (1,2) arc (360:180:0.5);
\draw[color=black, ->] (0,2) -- (0,2.5);
\draw[color=black] (1,2) -- (1,2.5);
\draw[color=black, ->] (1,2.5) -- (1,2.25);
\draw[thick] plot[mark=*] (0,2.25);
} \end{array}
\qquad \cdots 
\end{gather*}

Then, a quasi-reduced 2-cell does not rewrite into a linear combination of others. \end{proof}

\begin{small}
\renewcommand{\refname}{\Large\textsc{References}}
\bibliographystyle{alpha}
\bibliography{biblioihdr}
\end{small}

\section{Appendix}

In this appendix, we give all confluence diagrams for the critical branchings of $\overline{\mathrm{AOB}}$ verifying the axioms of decreasingness. For each $n$ in $\mathbb{N}$, with respect to the partition defined in the proof of lemma \ref{part2}, we index the rewriting steps of $P_n$ by $n$.

\begin{gather*}

\end{gather*}

\end{document}